\documentclass[letterpaper]{amsart}

\usepackage{amssymb,mathrsfs,hyperref,enumitem,mathtools,graphicx,tikz-cd}
\usepackage[alphabetic,initials,nobysame]{amsrefs}
\usepackage{caption,subcaption}

\BibSpec{collection.article}{%
	+{}  {\PrintAuthors}				{author}
	+{,} { \textit}                     {title}
	+{.} { }                            {part}
	+{:} { \textit}                     {subtitle}
	+{,} { \PrintContributions}         {contribution}
	+{,} { \PrintConference}            {conference}
	+{}  {\PrintBook}                   {book}
	+{,} { }                            {booktitle}
	+{,} { }							{series}
	+{,} { }							{publisher}
	+{,} { }							{address}
	+{,} { \PrintDateB}                 {date}
	+{,} { pp.~}                        {pages}
	+{,} { }                            {status}
	+{,} { \PrintDOI}                   {doi}
	+{,} { available at \eprint}        {eprint}
	+{}  { \parenthesize}               {language}
	+{}  { \PrintTranslation}           {translation}
	+{;} { \PrintReprint}               {reprint}
	+{.} { }                            {note}
	+{.} {}                             {transition}
	+{}  {\SentenceSpace \PrintReviews} {review}
}

\newtheorem{theorem}{Theorem}[section]
\newtheorem{lemma}[theorem]{Lemma}
\newtheorem{corollary}[theorem]{Corollary}
\newtheorem{proposition}[theorem]{Proposition}
\newtheorem{conj}[theorem]{Conjecture}

\theoremstyle{remark}
\newtheorem{remark}[theorem]{Remark}

\theoremstyle{definition}
\newtheorem{definition}[theorem]{Definition}

\newcommand{\R}{\mathbb{R}}
\newcommand{\Z}{\mathbb{Z}}
\newcommand{\C}{\mathbb{C}}
\newcommand{\J}{\mathcal{J}}
\newcommand{\N}{\mathbb{N}}

\newcommand{\D}{\mathbb{D}}
\newcommand{\loc}{\mathrm{loc}}

\renewcommand{\bar}[1]{\overline{#1}}

\DeclareMathOperator{\dist}{dist}
\DeclareMathOperator{\inter}{int}
\DeclareMathOperator{\diam}{diam}
\DeclareMathOperator{\Teich}{Teich}
\renewcommand{\Re}{\operatorname{Re}}

\makeatletter
\DeclareFontFamily{U}{tipa}{}
\DeclareFontShape{U}{tipa}{m}{n}{<->tipa10}{}
\newcommand{\arc@char}{{\usefont{U}{tipa}{m}{n}\symbol{62}}}%

\newcommand{\arc}[1]{\mathpalette\arc@arc{#1}}
\newcommand{\arc@arc}[2]{%
  \sbox0{$\m@th#1#2$}%
  \vbox{
    \hbox{\resizebox{\wd0}{\height}{\arc@char}}
    \nointerlineskip
    \box0
  }%
}
\makeatother

\numberwithin{equation}{section}
\numberwithin{figure}{section}

\title{Uniformization of gasket Julia sets}

\author{Yusheng Luo}
\address[Yusheng Luo]{Department of Mathematics, Cornell University, 212 Garden Ave, Ithaca, NY 14853, USA}
\email{yusheng.s.luo@gmail.com}
\thanks{The first-named author was partially supported by NSF Grant DMS-2349929.}

\author{Dimitrios Ntalampekos}
\address[Dimitrios Ntalampekos]{Department of Mathematics, Aristotle University of Thessaloniki, Thessaloniki, 54152, Greece.}
\email{dntalam@math.auth.gr}

\thanks{The second-named author was partially supported by NSF Grant DMS-2246485}

\begin{document}
\begin{abstract}
    The object of the paper is to characterize gasket Julia sets of rational maps that can be uniformized by round gaskets. We restrict to rational maps without critical points on the Julia set. Under these conditions, we prove that a Julia set can be quasiconformally uniformized by a round gasket if and only if it is a fat gasket, i.e., boundaries of Fatou components intersect tangentially. We also prove that a Julia set can be uniformized by a round gasket with a David homeomorphism if and only if every Fatou component is a quasidisk; equivalently, there are no parabolic cycles of multiplicity $2$. Our theorem applies to show that gasket Julia sets and limit sets of Kleinian groups can be locally quasiconformally homeomorphic, although globally this is conjectured to be false.
\end{abstract}
\maketitle

\setcounter{tocdepth}{1}
\tableofcontents

\section{Introduction}

\subsection{Uniformization of gasket Julia sets}

We present a uniformization result for gaskets arising in complex dynamics as Julia sets of rational maps. Namely, our main theorems describe necessary and sufficient conditions so that a gasket Julia set can be uniformized by a round gasket with a quasiconformal or David homeomorphism of the sphere. 

Such results existed so far only for Sierpi\'nski carpet Julia sets. A \textit{Sierpi\'nski carpet} is a locally connected continuum in $\widehat\C$ that has empty interior and whose complement is the union of countably many Jordan regions with disjoint closures. Bonk \cite{Bonk:uniformization} established a general criterion for the quasiconformal uniformization of carpets by round carpets, and then Bonk--Lyubich--Merenkov \cite{BLM16} proved that carpet Julia sets of subhyperbolic rational maps satisfy that criterion. In contrast, such a criterion does not exist for general gaskets; our uniformization results below are closely tied with the geometry of Julia sets. Let us first introduce the required definitions.

\begin{definition}\label{definition:gasket}
A set $K\subset \widehat\C $ is a \textit{gasket} if it is a locally connected continuum with empty interior that has the following properties. 
\begin{enumerate}
    \item Each complementary component of $K$ is a Jordan region.
    \item The boundaries of any two complementary components of $K$ share at most one point, called a \textit{contact point}.
    \item No point of $\widehat\C$ belongs to the boundaries of three complementary components of $K$.
    \item The \textit{contact graph} corresponding to $K$, obtained by assigning a vertex to each complementary component and an edge if two components share a boundary point, is connected.
\end{enumerate}
\end{definition}

A gasket $K$ is said to be \textit{round} if each complementary component is a geometric disk of $\widehat\C$. In other words, a round gasket corresponds to a circle packing of $\widehat\C$. A relevant notion is that of a \textit{Schottky set} \cite{BonkKleinerMerenkov:schottky}, i.e., the complement of the union of at least three disjoint open balls in $\widehat{\C}$. Note that a set that is homeomorphic to a Schottky set necessarily satisfies conditions (1), (2), and (3) in the above definition. Moreover, condition (4) differentiates carpets from gaskets. Gaskets and round gaskets appear naturally in dynamics and low-dimensional topology.

We say that two disjoint Jordan regions $U,V\subset \widehat\C$ are \textit{tangent to each other at a point} $z_0\in \partial U\cap \partial V\cap \C$ if there exists $\theta_0\in [0,2\pi)$ such that for each $\varepsilon>0$ there exists $\delta>0$ with the property that $\{z_0+ re^{i(\theta_0+\theta)}:  r\in (0,\delta), \,\, |\theta|<\pi/2-\varepsilon\}\subset U$ and $\{z_0+ re^{i(\theta_0+\theta)}:  r\in (-\delta,0), \,\, |\theta|<\pi/2-\varepsilon\}\subset V$. Tangency at the point $\infty$ is also defined in the obvious manner, using a local chart. We say that a gasket $K\subset \widehat\C$ is  \textit{fat} if any two complementary components of $K$ that share a boundary point are tangent to each other at that point.

It is elementary to show that if a gasket Julia set is fat and contains no critical points then each Fatou component is a quasidisk; see Lemma~\ref{lem:tpp}.  Our first main result gives the much stronger conclusion that fatness is necessary and sufficient for a gasket Julia set to admit quasiconformal uniformization by a round gasket.

\begin{theorem}[QC uniformization]\label{thm:QU}
    Let $R$ be a rational map without Julia critical points whose Julia set $\J(R)$ is a gasket. Then the following are equivalent.
    \begin{enumerate}[label=\normalfont(\arabic*)]
        \item (Quasiconformal uniformization) There exists a quasiconformal homeomorphism $\phi \colon \widehat\C\to \widehat\C$ that maps $\J(R)$ onto a round gasket.
        \item (Geometric criterion) $\J(R)$ is a fat gasket.
        \item (Dynamical criterion) Every contact point is eventually mapped to a parabolic periodic point with multiplicity $3$.
    \end{enumerate}
    Moreover, if $\psi$ is any orientation-preserving homeomorphism of $\widehat \C$ that maps $\mathcal J(R)$ onto a round gasket, then $\psi|_{\mathcal J(R)}$ agrees with the map $\phi|_{\J(R)}$ in {\normalfont(1)} up to post-composition with a M\"obius transformation of $\widehat\C$.
\end{theorem}
We remark that there is an abundance of examples of rational maps whose Julia set is a fat gasket (see \cite{LuoZhang:Nonequiv} for a combinatorial classification). In particular, there are infinitely many non-homeomorphic quadratic fat gasket Julia sets. 

A more general class of maps that appear naturally in conformal dynamics is the class of David maps (see Section \ref{subsec:DavidMap}). These maps connect hyperbolic and parabolic dynamics, a relationship first noted in Ha\"issinsky’s work (see \cites{Hai98, Hai00}). Recently, there has been a growing number of applications of David maps in gluing and mating problems in dynamics (see \cites{PetersenZakeri,Zak08, LyubichMerenkovMukherjeeNtalampekos:David, LMM24, LLM24}). Our second result gives a characterization for the existence of a David uniformizing map.

\begin{theorem}[David uniformization]\label{thm:DU}
    Let $R$ be a rational map without Julia critical points whose Julia set $\J(R)$ is a gasket. Then the following are equivalent.
    \begin{enumerate}[label=\normalfont(\arabic*)]
        \item (David uniformization) There exists a David homeomorphism $\phi \colon \widehat\C\to \widehat\C$ that maps $\J(R)$ onto a round gasket.
        \item (Geometric criterion) Each Fatou component is a quasidisk.
        \item (Dynamical criterion) $R$ has no parabolic cycles of multiplicity $2$.
    \end{enumerate}
    Moreover, if $\psi$ is any orientation-preserving homeomorphism of $\widehat \C$ that maps $\mathcal J(R)$ onto a round gasket, then $\psi|_{\mathcal J(R)}$ agrees with the map $\phi|_{\J(R)}$ in {\normalfont(1)} up to post-composition with a M\"obius transformation of $\widehat\C$.
\end{theorem}

We remark that under the assumptions of Theorem \ref{thm:DU}, any contact point is eventually mapped to either a repelling periodic point or a parabolic periodic point with multiplicity $2$ or $3$; see Lemma \ref{lem:pt}. Therefore, condition (3) implies that every contact point is eventually mapped to either a repelling periodic point or a parabolic periodic point with multiplicity $3$; compare to the dynamical criterion in Theorem \ref{thm:QU}.

We remark that Theorem \ref{thm:QU} and Theorem \ref{thm:DU} have an analogue for Kleinian groups, which was known by a classical result of McMullen. More precisely, let $\Gamma$ be a geometrically finite Kleinian group with gasket limit set $\Lambda(\Gamma)$. Then McMullen showed in \cite{McM90} that there exists a unique Kleinian group with a round gasket limit set in its quasiconformal conjugacy class. In particular, there exists a quasiconformal homeomorphism $\phi: \widehat\C \to \widehat \C$ that maps $\Lambda(\Gamma)$ onto a round gasket.

\subsection{Quasiconformal (in-)equivalence of limit sets and Julia sets}
It is a central question in quasiconformal geometry to classify fractal sets up to quasiconformal homeomorphisms. This problem has attracted a good deal of attention in recent years; see, for example, \cites{BK02, HP12, BM13, Mer14, BLM16, LM18, LLMM19}. For fractal sets that emerge in conformal dynamics, ample evidence suggests that it is possible to quasiconformally distinguish Julia sets and limit sets of Kleinian groups.
The following conjecture is stated in \cite{LLMM19}. 

\begin{conj}\label{conj:gqh}
Let $\mathcal J$ be the Julia set of a rational map and $\Lambda$ be the limit set of a Kleinian group.
Suppose that $\mathcal J$ and $\Lambda$ are connected and not homeomorphic to the circle or the $2$-sphere. Then there exists no quasiconformal homeomorphism of the Riemann sphere that maps $\mathcal J$ onto $\Lambda$.
\end{conj}

This conjecture is proved in various settings when the Julia set or the limit set is a Sierpi\'nski carpet (see \cites{Mer14, BLM16, QYZ19}). When different Fatou components are allowed to touch, the situation is more subtle. While the conjecture is proved in various special cases of gasket Julia sets and limit sets (see \cite{LuoZhang:Nonequiv}), the following theorem shows that the conjecture is false in the local sense. See Figure \ref{fig:FGJ} for an illustration.

\begin{theorem}[Local QC equivalence]\label{theorem:nlo}
There exist a gasket Julia set $\mathcal{J}$ and a gasket limit set $\Lambda$ so that for any point $z \in \mathcal{J}$, there exist a neighborhood $U$ of $z$ and a quasiconformal map $\phi\colon \widehat\C \to\widehat\C$ so that $\phi(U \cap \mathcal{J}) = \phi(U) \cap \Lambda$.
\end{theorem}

We remark that in our example we can construct an atlas of the Julia set consisting of 4 charts (see Figure \ref{fig:FGJ} and Theorem \ref{theorem:snlo} for a more precise statement).
From a different perspective, the Julia set $\mathcal{J}$ and the limit set $\Lambda$ in our example can be decomposed (in two different ways) into two connected pieces $\mathcal{J}_1, \mathcal{J}_2$ and $\Lambda_1, \Lambda_2$ glued along four points, so that the pairs $\mathcal{J}_i, \Lambda_i$, $i=1,2$, are quasiconformally homeomorphic; in Figure \ref{fig:FGJ} one can take $\mathcal J_1=\mathcal J\cap U^+$, $\mathcal J_2=\mathcal J\cap U^-$, $\Lambda_1=\Lambda\cap V^+$, and $\Lambda_1=\Lambda\cap V^-$. Our construction is very flexible and should allow us to construct infinitely many pairs of gasket Julia sets and limit sets with the same property. The technique demonstrates that there are no local or analytic obstructions preventing a Julia set from being quasiconformally equivalent to a limit set. The only obstruction is combinatorial in nature and arises from gluing different quasiconformally equivalent charts. On the other hand, there are many homeomorphic gasket Julia sets and gasket limit sets (see \cite{LLM22}).

\begin{figure}[htp]
    \centering    
    \begin{tikzpicture}
    \node[anchor=south west,inner sep=0, opacity=.5] at (0,0) {\includegraphics[width=0.45\textwidth]{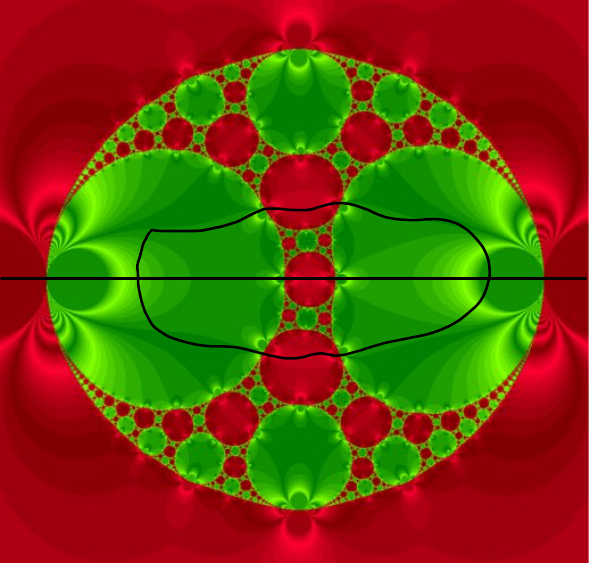}};
    \node at (1,3.05) {\begin{footnotesize}$\widetilde{U}^+$\end{footnotesize}};
    \node at (1,2.55) {\begin{footnotesize}$\widetilde{U}^-$\end{footnotesize}};
    \node at (1.6,2.8) {\begin{footnotesize}$A$\end{footnotesize}};
    \node at (2.95,3.6) {\begin{footnotesize}$B$\end{footnotesize}};
    \node at (4.3,2.8) {\begin{footnotesize}$C$\end{footnotesize}};
    \node at (2.95,1.85) {\begin{footnotesize}$D$\end{footnotesize}};
    \node at (3,2.75) {\begin{footnotesize}$E$\end{footnotesize}};
    \node at (4.5,0.8) 
    {\begin{footnotesize}$F$\end{footnotesize}};
    \node at (3.65,3.15) {\begin{footnotesize}$U^+$\end{footnotesize}};
    \node at (4,3.5) {\begin{footnotesize}$U^-$\end{footnotesize}};
    \end{tikzpicture}
    \begin{tikzpicture}
    \node[anchor=south west,inner sep=0] at (0,0) {\includegraphics[width=0.45\textwidth]{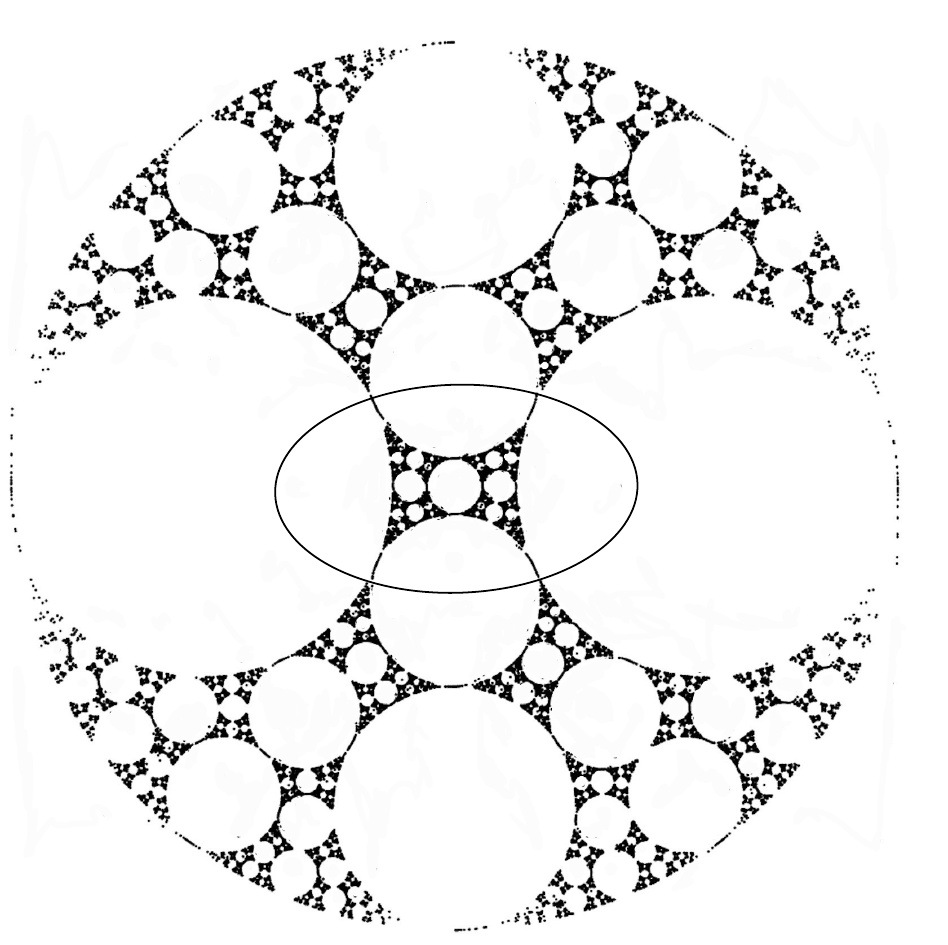}};
    \node at (1.2,2.8) {\begin{footnotesize}$A$\end{footnotesize}};
    \node at (2.75,3.6) {\begin{footnotesize}$B$\end{footnotesize}};
    \node at (4.3,2.8) {\begin{footnotesize}$C$\end{footnotesize}};
    \node at (2.75,2) {\begin{footnotesize}$D$\end{footnotesize}};
    \node at (3.45,3) {\begin{footnotesize}$V^+$\end{footnotesize}};
    \node at (3.8,3.3) {\begin{footnotesize}$V^-$\end{footnotesize}};
    \end{tikzpicture}
    \caption{Left: an example of a Julia set that is quasiconformally homeomorphic to a circle packing. Right: an example of a Kleinian circle packing. These two sets are not homeomorphic, but they are \textit{locally quasiconformally homeomorphic}.}
    \label{fig:FGJ}
\end{figure}
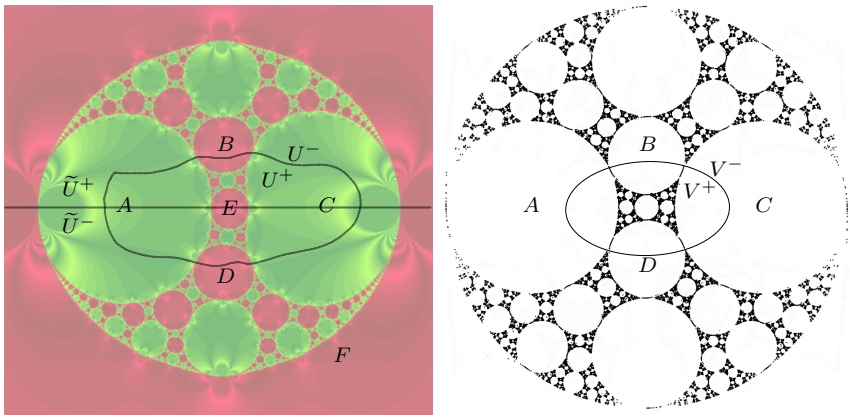

Sierpi\'nski carpet limit sets exhibit very strong local rigidity; see the works \cites{Mer14, Mer14ProcAMS}. Specifically, any quasiregular symmetry of a Sierpi\'nski carpet limit set $\Lambda$ defined on an open and connected subset of $\Lambda$ is the restriction of a M\"obius map and, in particular, it must be injective; this is a consequence of \cite{Mer14}*{Theorem 1.1}. We expect this local rigidity would preclude local quasiconformal maps between Sierpi\'nski carpet Julia sets and limit sets, as opposed to the gasket case in Theorem \ref{theorem:nlo}. As a corollary to Theorem \ref{theorem:nlo}, and in sharp contrast with the case of Sierpi\'nski carpet limit sets, we show that gasket limit sets do admit local quasiregular and non-injective symmetries.

\begin{corollary}[QR symmetries for gasket limit sets]\label{qrsymmetry}
    There exist a gasket limit set $\Lambda$, open and connected sets $W, V\subset \widehat \C$ with $W\cap \Lambda\neq \emptyset$ and a quasiregular branched covering $f\colon W \to V$ of degree $2$ so that $f(W\cap\Lambda) = V \cap \Lambda$.
\end{corollary}
We remark that in our construction (see Section \ref{sec:lqc} and Figure \ref{fig:QRS}), one can show that the subset $A = W \cap \Lambda$ can be decomposed into $4$ connected pieces $A_1,\dots A_4$, where $A_i\cap A_{i+1}$ is a single point and $f|_{A_i}$ is a M\"obius map.

\subsection{Proof outline}
There are three main ingredients in the proofs of Theorems \ref{thm:QU} and \ref{thm:DU}. First, we use the theory of circle packings with subdivision rules that was recently developed by the first-named author and Y.~Zhang \cite{LuoZhang:CirclePacking} and especially the asymptotic conformality of homeomorphisms between such circle packings. Second, we use an extension result for circle homeomorphisms conjugating two piecewise quasiconformal dynamical systems of the circle; this extension result was developed jointly in earlier work of the authors \cite{LN24}, generalizing a similar result for piecewise analytic dynamical systems of the circle \cite{LyubichMerenkovMukherjeeNtalampekos:David}. This extension result is used in combination with quasiconformal or David surgery to construct the uniformizing maps in Theorems \ref{thm:QU} and \ref{thm:DU}. Finally, the quasiconformal or David regularity of these maps is proved with the aid of a recent metric characterization of quasiconformality due to the second-named author \cite{Ntalampekos:metric}.

We now provide an outline of the basic steps of the proof. The equivalence $(2) \Leftrightarrow (3)$ between the geometric and dynamical criteria in Theorems \ref{thm:QU} and \ref{thm:DU} is proved by analyzing the local dynamics near the fixed points; see Section \ref{section:geom_dyn}. To prove that $(1) \Rightarrow (3)$ in Theorem \ref{thm:QU}, we first show that if $(3)$ does not hold, then either there is a Fatou component that is not a quasidisk or there are two Fatou components that touch at an angle. In the second case, these two Fatou components cannot be mapped to a pair of touching circles by a quasiconformal map. To prove that $(1) \Rightarrow (3)$ in Theorem \ref{thm:DU}, we first show that if $(3)$ does not hold, then there exists a Fatou component that is the immediate basin of a parabolic periodic point with multiplicity $2$. This immediate basin cannot be mapped to a round disk by a David map on $\widehat\C$ as provided by a recent result of the authors \cite{LN24}.

The most technical direction is to show that the geometric and dynamical criteria are also sufficient for uniformization. We remark that a priori, it is not clear that a gasket Julia set is homeomorphic to a circle packing. To this end, we first construct a finite subdivision rule for the the contact graph of the Julia set $\mathcal{J}(R)$. By results of \cite{LuoZhang:CirclePacking}, this allows us to construct an infinite circle packing $\mathcal{P}$ whose limit set $\Lambda(\mathcal{P})$ is homeomorphic to $\mathcal{J}(R)$ (see Section \ref{sec:cp}). The finite subdivision rule induces a Markov map $\Psi\colon\overline{\Omega^1} \to \overline{\Omega^0}$, where $\Omega^n\coloneqq \Psi^{-n}(\Omega^0)$, $n\geq 0$, is some nested sequence open sets that satisfies $\Lambda(\mathcal{P}) = \bigcap_{n=0}^\infty \overline{\Omega^n}$. After passing to some iterate of $R$ if necessary, we show that there exists a homeomorphism $h:\widehat\C \to\widehat\C$ conjugating the dynamics of $R\colon \mathcal{J}(R) \to \mathcal{J}(R)$ to $\Psi\colon \Lambda(\mathcal{P}) \to\Lambda(\mathcal{P})$. Thus, the map $\Psi|_{\Lambda(\mathcal P)}$ plays the role of $R|_{\mathcal J(R)}$. Although the limit set $\Lambda(\mathcal P)$ does not arise from a Kleinian group, so we do not have conformal symmetries, the map $\Psi$ is as conformal as possible in the sense that its dilatation is exponentially small on deep levels, as shown in \cite{LuoZhang:CirclePacking}. More precisely, this means that the quasiconformal distortion of $\Psi$ on $\Omega^n$ is $1+O(\delta^{-n})$ for some $\delta \in (0, 1)$.

To complete the proof, we show that $h|_{\mathcal{J}(R)}$ extends to a quasiconformal or David map on $\widehat\C$. To achieve this, we first show that
\begin{itemize}
       \item if each Fatou component is a quasidisk, then $h$ extends to a David homeomorphism on each periodic Fatou component, and
        \item if $\mathcal{J}(R)$ is a fat gasket, then $h$ extends to a quasiconformal homeomorphism on each periodic Fatou component.
    \end{itemize}
The extension in each Fatou component is established by a recent extension result for piecewise quasiconformal maps on the circle (see \cite{LN24} and Section \ref{sec:epqcm} here). We first extend $h$ to the periodic Fatou components and then using the dynamics, we pull back the David or quasiconformal homeomorphism to all Fatou components, and obtain a homeomorphism $h\colon (\widehat\C,\mathcal{J}(R)) \to (\widehat\C,\Lambda(\mathcal{P}))$ that has David or quasiconformal regularity on $\widehat \C \setminus \mathcal{J}(R)$.

We remark that the Sierpi\'nski gasket is not removable for quasiconformal maps, as shown in \cite{Ntalampekos:gasket}, and it is expected that all gaskets are not removable as well. Hence, the David or quasiconformal regularity of the map $h$ on $\widehat \C\setminus \mathcal J(R)$ does not immediately imply that we have the same regularity on all of $\widehat \C$. To obtain the desired regularity of $h$ on $\mathcal J(R)$ we apply a recent metric characterization of quasiconformality \cite{Ntalampekos:metric}, asserting that for a homeomorphism to be quasiconformal it is sufficient that it maps infinitesimal sets of bounded eccentricity to infinitesimal sets of bounded eccentricity; see Section \ref{section:abs}.

\section{Preliminaries on quasiconformal geometry}

\subsection{Quasisymmetric maps}

Let $(X,d_X),(Y,d_Y)$ be metric spaces. A homeomorphism $f\colon (X,d_X)\to (Y,d_Y)$ is \textit{quasisymmetric} if there exists a homeomorphism $\eta\colon [0,\infty)\to[0,\infty)$ such that for all triples of distinct points $x_1,x_2,x_3\in X$ we have
\begin{align*}
    \frac{d_Y(f(x_1),f(x_2))}{d_Y(f(x_1),f(x_3))} \leq \eta\left( \frac{d_X(x_1,x_2)}{d_X(x_1,x_3)}\right).
\end{align*}
In that case we say that $f$ is $\eta$-quasisymmetric. 

\begin{lemma}[\cite{Kelingos:qs}*{Theorem 3}]\label{lemma:qs_glue}
    Let $f\colon [-1,1]\to \R$ be a homeomorphic embedding such that $f|_{[-1,0]}$ and $f|_{[1,0]}$ are quasisymmetric and there exists $\lambda\geq 1$ with the property that
    \begin{align}\label{lemma:qs_glue_ineq}
        \lambda^{-1} |t|\leq  |f(t)-f(0)|\leq \lambda |t|
    \end{align}
    for all $t\in [-1,1]$. Then $f$ is quasisymmetric.
\end{lemma}
For example, condition \eqref{lemma:qs_glue_ineq} is met if the one-sided derivatives of $f$ at $0$ exist and are non-zero.

\subsection{Quasiconformal and quasiregular maps}

Let $U,V\subset \C$ be open sets and $f\colon U\to V$ be an orientation-preserving homeomorphism. We say that $f$ is \textit{quasiconformal} if $f$ lies in the Sobolev space $W^{1,2}_{\loc}(U)$ and there exists $K\geq 1$ such that $\|Df(z)\|^2 \leq K J_f(z)$ for a.e.\ $z\in U$. Here $Df(z)$ is the differential of $f$, which exists at a.e.\ $z\in U$, $\|Df(z)\|$ denotes its operator norm, and $J_f(z)$ is the Jacobian determinant of $Df(z)$. In that case, we say that $f$ is $K$-quasiconformal. Equivalently, one may replace the condition that $\|Df\|^2\leq KJ_f$ with the condition that the \textit{Beltrami coefficient} of $f$, defined by
\begin{align*}
\mu_f= \frac{\partial f/ \partial \bar z}{\partial f/\partial z},
\end{align*}
satisfies $\|\mu_f\|_{L^\infty(U)}<1$. One can define quasiconformal homeomorphisms between domains in the Riemann sphere $\widehat \C$ via local coordinates. A homeomorphism of the Riemann sphere is quasiconformal if and only if it is quasisymmetric in the spherical metric; see \cite{Heinonen:metric}*{Theorem 11.14}.

A \textit{quasidisk} is the image of the unit disk under a quasiconformal homeomorphism of $\widehat \C$.  Several characterizations of quasidisks are available; see \cite{Pommerenke:conformal}*{Section 5.1}. For example,  a Jordan region in $\widehat \C$ is a quasidisk if and only if its boundary is a quasisymmetric image (using the spherical metric) of the unit circle.

Let $U,V\subset \C$ be open sets. A map $f\colon U \to V$ is called \textit{quasiregular} if $f$ lies in the Sobolev space $W^{1,2}_{\loc}(U)$ and there exists $K\geq 1$ such that $\|Df(z)\|^2 \leq K J_f(z)$ for a.e.\ $z\in U$. In that case we say that $f$ is $K$-quasiregular. Equivalently, as above, one may replace the condition that $\|Df\|^2\leq KJ_f$ with the condition that $\|\mu_f\|_{L^\infty(U)}<1$. Quasiregular maps are continuous \cite{AstalaIwaniecMartin:quasiconformal}*{Corollary 5.5.2}, but they are not necessarily injective. Furthermore, it is known that $f$ is quasiregular if and only if we can write $f=g\circ h$, where $g$ is analytic and $h$ is quasiconformal; see \cite{LehtoVirtanen:quasiconformal}*{p.~247} or \cite{AstalaIwaniecMartin:quasiconformal}*{Theorem 5.5.1}.

\begin{lemma}\label{lem:qrextension}
    Let $f\colon \mathbb S^1 \to \mathbb S^1$ be an orientation-preserving covering map of degree $d\geq 2$ that satisfies one of the following conditions.
    \begin{enumerate}[label=\normalfont(\arabic*)]
        \item\label{qr:1}$f$ has a continuous and non-zero derivative on $\mathbb S^1$.
        \item\label{qr:2}$f$ is locally quasisymmetric on $\mathbb S^1$.
    \end{enumerate}
    Then $f$ extends to a quasiregular map $F\colon \D \to \D$ that has a unique critical point at $0$ and $F(0) = 0$.
\end{lemma}
\begin{proof}
    For \ref{qr:1}, we replace $\mathbb S^1$ with $\R/2\pi\mathbb{Z}$ and we treat $f$ as a function of $\theta\in  \R/2\pi \mathbb{Z}$ so that $f'(\theta)>0$ for each $\theta \in \R/2\pi\mathbb{Z}$. We define the extension of $f$ in $\D$ by $F(re^{i\theta}) = re^{if(\theta)}$, $0\leq r<1$, $\theta\in [0,2\pi)$. We have
    \begin{align*}
    \mu_F(re^{i\theta})= \frac{\frac{\partial F}{\partial r} \frac{\partial r}{\partial \bar z} + \frac{\partial F}{\partial \theta} \frac{\partial \theta}{\partial \bar z}}{\frac{\partial F}{\partial r} \frac{\partial r}{\partial z} + \frac{\partial F}{\partial \theta} \frac{\partial \theta}{\partial z}} = \frac{e^{if(\theta)}\frac{e^{i\theta}}{2} + r e^{if(\theta)}\cdot if'(\theta)\cdot \frac{-e^{i\theta}}{2ir}}{e^{if(\theta)}\frac{e^{-i\theta}}{2} + r e^{if(\theta)}\cdot if'(\theta)\cdot \frac{e^{-i\theta}}{2ir}} = e^{2i\theta} \frac{1-f'(\theta)}{1+f'(\theta)}.
    \end{align*}
    Since $f'$ is continuous and positive, we have $\|\mu_F\|_{L^\infty}(\D)<1$. Therefore, $F$ is quasiregular. It follows from the construction that $F(0)=0$ and $F$ is locally injective at each point of $\D\setminus\{0\}$, but not at $0$.
 
    Now we prove the conclusion under condition \ref{qr:2}. The Beurling--Ahlfors extension theorem \cite{BeurlingAhlfors:extension} gives a quasiconformal extension to $\D$ of an orientation-preserving quasisymmetric map of $\mathbb S^1$. The extension is given by a \textit{local} extension operator of continuous and increasing maps of $\R$ to self-maps of the upper half-plane. The locality of the extension operator implies that a locally quasisymmetric map $f\colon \mathbb S^1\to \mathbb S^1$ extends to a map $\widetilde{f} \colon \bar \D \to \bar \D$ with the following property: each point of $\mathbb S^1$ has a neighborhood in $\D$ in which $\widetilde f$ is injective and quasiconformal. Thus, $\widetilde f$ is quasiregular in a neighborhood $U$ of $\mathbb S^1$ in $\D$. We restrict $\widetilde f$ to $U\cup \mathbb S^1$.
    
    We also define $\widetilde f(z)= 1/\bar{\widetilde f(1/\bar z)}$ for points $z\in \C$ with $1/\bar z\in U$.  The removability of $\mathbb S^1$ for quasiconformal maps implies that each point of $\mathbb S^1$ has a neighborhood in $\widehat \C$ in which $\widetilde f$ is quasiconformal.  We define $\widetilde \mu(z)=\mu_{\widetilde f}(z)$ whenever $z\in U$ or $1/\bar z\in U$ and $\widetilde \mu(z)=0$ otherwise. By construction, $\widetilde \mu$ is invariant under the reflection $z\mapsto 1/\bar z$. By the measurable Riemann mapping theorem (\cite{AstalaIwaniecMartin:quasiconformal}*{Theorem 5.3.4}) there exists a quasiconformal map $h$ of $\widehat \C$ with $\mu_h=\widetilde \mu$. By symmetry, after post-composition with a M\"obius map, we may assume that $h(\mathbb S^1) = \mathbb S^1$ and $h(0) = 0$; see e.g.\ \cite{NtalampekosYounsi:rigidity}*{Proposition 8.1} for a more general statement. In particular, $h|_{\mathbb S^1}$ is orientation-preserving. Let $g = \widetilde{f} \circ h^{-1}$, and note that $g$ is locally injective on $\mathbb S^1$ and $g|_{\mathbb S^1}$ is orientation-preserving. By \cite{AstalaIwaniecMartin:quasiconformal}*{Theorem 5.5.1}, $g$ is analytic in a neighborhood of $\mathbb S^1$ in $\widehat \C$.  Since $g$ is locally injective on every point of $\mathbb S^1$, $g$ has no critical point on $\mathbb{S}^1$. By the first part of the lemma, $g|_{\mathbb S^1}$ has a quasiregular extension $G$ on $\D$ that has a unique critical point at $0$ and $G(0) = 0$. We finally set $F = G \circ h$ to obtain the desired extension.
\end{proof}

\subsection{David maps}\label{subsec:DavidMap}
An orientation-preserving homeomorphism $H\colon U\to V$ between domains in the Riemann sphere $\widehat{\C}$ is called a \textit{David homeomorphism} or else \textit{mapping of exponentially integrable distortion}, if it lies in the Sobolev space $W^{1,1}_{\loc}(U)$ (defined via local coordinates) and there exists $p>0$ such that
\begin{align}\label{exp_integral}
\int_U \exp(p K_H) \,d\sigma <\infty;
\end{align}
here $\sigma$ is the spherical measure and $K_H$ is the distortion function of $H$, given by
\begin{align*}
K_H(z)=\frac{1+|\mu_H|}{1-|\mu_H|}.
\end{align*}
Condition \eqref{exp_integral} is equivalent to the existence of constants $C,\alpha>0$ such that
$$\sigma(\{z\in U: |\mu_H(z)|>1-\varepsilon\})\leq C e^{-\alpha/\varepsilon}$$
for all $\varepsilon \in (0,1)$. We direct the reader to \cite{AstalaIwaniecMartin:quasiconformal}*{Chapter 20} for more background.  

If $U$ is an open subset of $\widehat {\C}$ and $\mu \colon U \to \D$ is a measurable function such that $K=(1+|\mu|)/(1-|\mu|)$ is exponentially integrable in $U$ (i.e., it satisfies \eqref{exp_integral}), then we say that $\mu$ is a \textit{David--Beltrami coefficient (in $U$)}.

\begin{proposition}[\cite{LyubichMerenkovMukherjeeNtalampekos:David}*{Proposition 2.5}]\label{prop:david_qc_invariance}
Let $f\colon U\to V$ be a David homeomorphism between open sets $U,V\subset \widehat {\C}$.
\begin{enumerate}[label=\normalfont(\roman*)]
\item If $g\colon V \to \widehat{\C}$ is a quasiconformal embedding then $g\circ f$ is a David map.
\item If $W\subset {\widehat{\C}}$ is an open set and $g\colon W \to U$ is a quasiconformal homeomorphism that extends to a quasiconformal homeomorphism of an open neighborhood of $\overline W$ onto an open neighborhood of $\overline U$, then $f\circ g$ is a David map.
\item If $W\subset \widehat{\C}$ is an open set, and $g\colon W\to U$ is a non-constant quasiregular map that extends to a quasiregular map in an open neighborhood of $\overline W$, then the function $$\mu_{f\circ g}=  \frac{\partial (f\circ g)/ \partial \overline z}{\partial (f\circ g)/\partial z} $$
is a David--Beltrami coefficient in $W$. 
\item If $U,W$ are quasidisks and $g\colon W\to U$ is a quasiconformal homeomorphism, then $f\circ g$ is a David map. 
\end{enumerate}
\end{proposition}

\begin{remark}\label{remark:composition}
    Following the argument in the proof of Proposition 2.5 in \cite{LyubichMerenkovMukherjeeNtalampekos:David} one can show that if $f$ is a David map on $U$ and $g$ is $K$-quasiconformal for some $K\geq 1$, then for a.e.\ $z\in U$ we have
    $$ 1-|\mu_{g\circ f}(z)| =\frac{2}{K_{g\circ f}(z)} \leq  \frac{2}{K^{-1}K_{f}(z)+1} \leq \frac{2}{K^{-1}(K_f(z)+1)}=K (1-|\mu_f(z)|).$$
    Moreover, if $g$ is conformal, then $|\mu_{f\circ g}|=|\mu_f\circ g|$ almost everywhere.
\end{remark}

\subsection{Absolute continuity}\label{section:abs}
In the proof of the main theorems we will use a recent generalization of the metric definition of quasiconformality due to the second-named author. Let $A\subset \widehat \C$ be a set. The \textit{eccentricity} $E(A)$ of $A$ is the infimum of all numbers $H\geq 1$ for which there exists a spherical open ball $B$ such that $B\subset A\subset HB$; here $HB$ denotes the ball with the same center as that of $B$ and radius multiplied by $H$. Let $f\colon U\to V$ be a homeomorphism between open sets $U,V\subset \widehat \C$. The \textit{eccentric distortion of $f$} at a point $z\in U$, denoted by $E_f(z)$, is the infimum of all values $H\geq 1$ such that there exists a sequence of open sets $A_n\subset U$, $n\in \mathbb N$, containing $z$ with $\diam A_n\to 0$ as $n\to\infty$ and with the property that $E(A_n)\leq H$ and $E(f(A_n))\leq H$ for each $n\in \mathbb N$. The main result in \cite{Ntalampekos:metric} shows that $f$ is quasiconformal if and only if $E_f$ is uniformly bounded. In fact, we will not use directly this result, but we will use the following statement about the absolute continuity of $f$.

\begin{theorem}[\cite{Ntalampekos:metric}*{Theorem 3.1}]\label{theorem:eccentric}
Let $f\colon U\to V$ be a homeomorphism between open sets $U,V\subset \widehat \C$. Let $E\subset U$ be a set of area zero and suppose that there exists $H\geq 1$ such that 
$$E_f(z)\leq H$$
for all $z\in E$. Then for a.e.\ line $L$ parallel to a coordinate direction we have
$$\mathcal H^1(f(L\cap E))=0.$$
\end{theorem}

Here and below, $\mathcal H^1$ denotes the Hausdorff $1$-measure defined using the Euclidean metric. We will also use the removability of quasicircles for continuous $W^{1,1}$ functions in the following sense.
\begin{proposition}\label{proposition:quasidisk}
    Let $U\subset \C$ be a quasidisk and $f\colon \overline U\to \R$ be a continuous function such that $f\in W^{1,1}(U)$. Then for a.e.\ line $L$ parallel to a coordinate direction we have
$$\mathcal H^1(f(\partial U\cap E))=0.$$  
\end{proposition}
The proof given in \cite{Ntalampekos:removabilitydetour}*{Proposition 5.3} poses a weaker assumption on $U$, namely $U$ is assumed to be a H\"older domain, and a stronger assumption on $f$, which is assumed to lie in $W^{1,2}(U)$. However, in the case of quasidisks (and also John domains), it suffices to assume that $f\in W^{1,1}(U)$ and the proof in \cite{Ntalampekos:removabilitydetour}*{Proposition 5.3} can be followed with appropriate simplifications. 

Combining the above two results, we obtain the next statement.

\begin{theorem}\label{theorem:qc_david_remov}
    Let $J\subset \widehat\C$ be a gasket of area zero such that the components $U_i$, $i\in \N$, of $\widehat\C \setminus J$ are quasidisks. Let $f\colon \widehat \C\to \widehat\C$ be a homeomorphism and suppose that there exists $H\geq 1$ such that
    $$E_f(z)\leq H$$
    for all $z\in  J\setminus \bigcup_{i\in \N}\partial U_i$. 
    \begin{enumerate}[label=\normalfont(\arabic*)]
        \item If $f$ is quasiconformal on $\widehat\C\setminus J$, then $f$ is quasiconformal on $\widehat\C$.
        \item If $f$ is a David map on $\widehat\C\setminus J$, then $f$ is a David map on $\widehat\C$.
    \end{enumerate}
\end{theorem}
\begin{proof}
    By composing with M\"obius transformations, we may assume that $\infty \in J\setminus \bigcup_{i\in \N}\partial U_i$, $f(\infty)=\infty$, and $E_f(z)\leq H$ for all $z\in  J\setminus \bigcup_{i\in \N}\partial U_i$. Note that the property of $f$ being a quasiconformal or David map is not affected by such compositions in view of Proposition \ref{prop:david_qc_invariance}. 
    Since $J$ has area zero and Sobolev functions are characterized as functions that are absolutely continuous on almost every line parallel to a coordinate direction (see \cite{Ziemer:Sobolev}*{Theorem 2.1.4}), if suffices to show that $f$ is absolutely continuous on almost every such line.   
    
    By Theorem \ref{theorem:eccentric} and Proposition \ref{proposition:quasidisk}, we have $\mathcal H^1(f(J\cap L))=0$ for a.e.\ line $L$ parallel to a coordinate direction. Moreover, the assumption that $f$ is a quasiconformal or a David map implies that for each bounded open set $V\subset \C$ we have $f\in W^{1,1}(V\setminus J)$; see the argument in the proof of Lemma 2.9 in \cite{LyubichMerenkovMukherjeeNtalampekos:David}. In particular, for a.e.\ line $L$ parallel to a coordinate direction we have $\int_{L\cap V\setminus J} \|Df\|\, d\mathcal H^1<\infty$ and $f|_{L\cap V\setminus J}$ is absolutely continuous. We fix a line $L$ satisfying all the above. We fix a segment $(z_1,z_2)\subset L\cap V$ and denote by $A_i$, $i\in I$, the components of $(z_1,z_2)\setminus J$.  We have 
    \begin{align*}
        |f(z_1)-f(z_2)|&\leq \mathcal H^1(f((z_1,z_2))) \leq \sum_{i\in I} \mathcal H^1(f(A_i)) + \mathcal H^1(f(J\cap L))\\
        &\leq \sum_{i\in I} \int_{A_i} \|Df\|\, d\mathcal H^1= \int_{(z_1,z_2)\setminus J} \|Df\|\, d\mathcal H^1.
    \end{align*}
    This implies that $f$ is absolutely continuous on $L\cap V$, as desired.    
\end{proof}

\subsection{Hyperbolic geometry}
If $X$ is a hyperbolic Riemann surface, we denote by $d_X$ its hyperbolic metric.  

\begin{lemma}[\cite{LuoZhang:CirclePacking}*{Lemma 5.5}]\label{lem:qc}
    Let $K\geq 1$ and $\phi\colon  \widehat \C \to \widehat \C$ be a $K$-quasiconformal map that fixes the points $0,1,\infty$. Then for $X=\widehat\C \setminus  \{0, 1,\infty\}$ and $t\in X$ we have
    $$
    d_{X}(t, \phi(t)) \leq \log K.
    $$
More generally, let $z_i,w_i\in \widehat \C$, $i\in \{1,2,3\}$, $\phi\colon \widehat \C \to \widehat \C$ be a $K$-quasiconformal map with $\phi(z_i)=w_i$, $i\in \{1,2,3\}$, and $M$ be a M\"obius map with $M(z_i)=w_i$, $i\in \{1,2,3\}$.  Then for $t\in \widehat \C \setminus  \{z_1, z_2, z_3\}$ and $X=\widehat \C \setminus   \{w_1,w_2,w_3\}$ we have
    $$
    d_{X}(M(t), \phi(t)) \leq \log K.
    $$
\end{lemma}

\begin{lemma}\label{lemma:qc_identity}
    Let $K\geq 1$, $U\subsetneq \C$ be a simply connected domain, and $\phi\colon U \to U$ be a $K$-quasiconformal homeomorphism that extends to identity map on $\partial U$.
    Then  
    $$\sup_{z\in U}d_{U}(z,\phi(z))\leq C(K).$$
    Moreover, for each $K_0>1$ there exists some constant $A= A(K_0)>0$ so that if $K\leq  K_0$, then
    $$\sup_{z\in {U}}d_{U}(z,\phi(z))\leq A \log K.$$
\end{lemma}

\begin{proof}
    Since $\phi$ is orientation-preserving, by \cite{McM88}*{Proposition 5.2}, $\phi$ is the identity on the ideal boundary $I(U)$, i.e., the space of prime ends of $U$. By conjugating $\phi$ with the uniformization map $\psi\colon  \D \to U$, we may assume that $U = \D$ and $\phi$ extends to the identity map on $\partial \D$.  We extend $\phi$ to $\widehat\C$ by the identity map in $\widehat \C\setminus \D$ to obtain a $K$-quasiconformal map on $\widehat\C$; here one uses the removability of the circle \cite{LehtoVirtanen:quasiconformal}*{Theorem I.8.3}.  Let $z\in \D$ and $M\colon \bar{\D}\to \bar \D$ be a M\"obius transformation such that $M(z)=0$. Then $\psi=M\circ \phi\circ M^{-1}$ is $K$-quasiconformal in $\D$ and its restriction to $\partial \D$ is the identity map. By replacing $\phi$ with $\psi$, it suffices to show the statement at the origin $z=0$. 
    
    Since $\phi|_{\partial \D}$ is the identity map, in particular, it fixes the points $1, i, -1$.
    By compactness of the space of $K$-quasiconformal maps fixing $1, i, -1$ (see \cite{LehtoVirtanen:quasiconformal}*{Theorem II.5.1}), we conclude that $d_{\D}(0,\phi(0)) \leq C(K)$. This proves the first part. For the second part, if $K \leq K_0$, then by the first part $\phi(0)$ lies in the hyperbolic ball $B_{d_\D}(0, C(K_0))$. We set $X=\widehat \C \setminus  \{1, i, -1\}$. In the ball $B_{d_\D}(0, C(K_0))$ the hyperbolic metrics $d_X$ and $d_\D$ are comparable with multiplicative constants depending on $K_0$. We apply Lemma \ref{lem:qc} to the map $\phi\colon  X\to X$ to conclude that 
    $$
    d_X(0, \phi(0)) \leq \log K.
    $$
    Therefore,  $d_{\D}(0,\phi(0))\leq A\log K$ for some constant $A = A(K_0)$.
\end{proof}

\begin{lemma}\label{lemma:qn}
    Let $\Omega_n=\C \setminus  [n,\infty)$ and $q_n\in \Omega_n$, $n\in \N\cup \{0\}$. Suppose that there exist constants $C>0$ and $\delta\in (0,1)$ such that $$d_{\Omega_n}(q_{n-1},q_n)< C \delta^n$$
    for each $n\in \N$. Then the sequence $\{q_n\}_{n\in \N}$ is convergent in the Euclidean metric and, in particular, it is bounded. 
\end{lemma}
\begin{proof}    
    Fix $n\in \N$ and let $\psi_n\colon \D\to \Omega_n$ be a conformal map with $\psi_n(0)=q_{n-1}$; note that $q_{n-1}\in \Omega_{n-1}\subset \Omega_n$. The hyperbolic ball $B_{d_{\D}}(0,C\delta^n)$ is contained in a Euclidean ball $B(0,C_1\delta^n)$, where $C_1>0$ does not depend on $n$.  By Koebe's distortion theorem (\cite{Pommerenke:conformal}*{Section 1.3}), for each $z\in B_{d_{\D}}(0,C\delta^n)$ we have
    $$|\psi_n(z)-q_{n-1}|\leq  C_2|\psi_n'(0)|\delta^n\leq 4C_2\dist(q_{n-1},\partial \Omega_n)\delta^n,$$
    where $C_2>0$ does not depend on $n$. Equivalently, the above holds whenever $\psi_n(z)\in B_{d_{\Omega_n}}(q_{n-1},C\delta^n)$. Since $d_{\Omega_n}(q_{n},q_{n-1})<C\delta^n$, we conclude that 
    $$ |q_{n-1}-q_n|\leq \frac{1}{2}M\dist(q_{n-1},\partial \Omega_n)\delta^n,$$
    where $M=8C_2$. Note that $\dist (q_{n-1}, \partial \Omega_n)\leq n+|q_{n-1}| \leq 2\max\{n, |q_{n-1}|\}$, so
    \begin{align}\label{eq:qn_qn-1}
        |q_{n-1}-q_n|  \leq M\max\{n, |q_{n-1}|\}\cdot \delta^n.
    \end{align}
    Therefore,
    \begin{align}
    \label{eq:iq}    |q_{n}| \leq |q_{n-1}| + M\max\{n, |q_{n-1}|\}\cdot \delta^n
    \end{align}
    for all $n\geq 1$.

    Let $N_0\in \N$ such that $MN_0\delta^{N_0} < 1$. Suppose that $|q_{n-1}| \geq n$ for all $n > N_0$. Then by \eqref{eq:iq}, we have 
    $$
    n+1\leq |q_{n}| \leq |q_{n-1}| + M|q_{n-1}|\cdot \delta^n = (1+M\delta^n)|q_{n-1}| \leq  |q_{N_0}|\prod_{i=N_0+1}^\infty (1+M\delta^i)
    $$
    for $n>N_0$. This is a contradiction. Therefore, there exists $n_0 > N_0$ with $|q_{n_0-1}| < n_0$. By \eqref{eq:iq}, we have
    $$
    |q_{n_0}| \leq  |q_{n_0-1}| + M n_0 \delta^{n_0} \leq |q_{n_0-1}|+ MN_0\delta^{N_0}< |q_{n_0-1}| + 1 < n_0+1.
    $$
    Inductively, we conclude that $|q_{n_0+k-1}| \leq n_0+k$ for $k\geq 0$. Therefore, by \eqref{eq:qn_qn-1} we have
    $$|q_{n_0+k-1}-q_{n_0+k}| \leq M (n_0+k)\delta^{n_0+k}.$$
    for all $k\geq 0$. This implies that $\{q_n\}_{n\in \N}$ is convergent in the Euclidean metric.  
    \end{proof}

\begin{corollary}\label{corollary:pn}
    Let $C>0$, $\delta\in (0,1)$, and for each $n\in \N$ let $\phi_n\colon \widehat \C\to \widehat \C$ be a $(1+C\delta^n)$-quasiconformal map such that $\phi_n|_{[0,1/n]}$ is the identity map. Then for each $p_0\in \widehat \C\setminus [0,1]$ and for $p_n=\phi_n^{-1}(p_{n-1})$, $n\in \N$, the sequence $\{p_n\}_{n\in \N}$ is bounded away from $0$. 
\end{corollary}

\begin{proof}
    Let $\psi_n(z)= 1/\phi_n(1/z)$, $n\in \N$, and note that $\psi_n$ is $(1+C\delta^n)$-quasi\-conformal and $\psi_n|_{[n,\infty]}$ is the identity map. Let $\Omega_n=\C\setminus [n,\infty)$ and $q_n=1/p_n$, $n\in \N\cup \{0\}$. By Lemma \ref{lemma:qc_identity}, there exists a constant $A>0$ such that  
    $$d_{\Omega_n}(q_n, q_{n-1})= d_{\Omega_n}(q_n,\psi_n(q_n))\leq A \log(1+C\delta^n)\leq AC\delta^n$$
    for each $n\in \N$. By Lemma \ref{lemma:qn}, $\{q_n\}_{n\in \N}$ is bounded away from $\infty$, so $\{p_n\}_{n\in \N}$ is bounded away from $0$. 
\end{proof}

\begin{lemma}\label{lemma:hyperbolic_euclidean}
For $t>0$, let $X_t=\widehat \C \setminus  \{0,t,\infty\}$. For $z,w\in (0,t)$ we have
$$ |z-w|\leq Ct \cdot d_{X_t}(z,w),$$
where $C>0$ is a uniform constant.
\end{lemma}
\begin{proof}
    The hyperbolic density $\rho_{X_1}$ of $X_1$ satisfies $\rho_{X_1}\geq C^{-1}$ on $(0,1)$ for some uniform constant $C>0$ (see e.g.\ \cite{BeardonPommerenke}*{Theorem 1}). Therefore, $\rho_{X_t}\geq C^{-1}/t$ on $(0,t)$. By symmetry, the hyperbolic geodesic between two points $z,w\in (0,t)$ in $X_t$ is the segment $[z,w]$; e.g.\ this follows from \cite{Kobayashi:transformation}*{Theorem II.5.1}. This gives the desired inequality. 
\end{proof}

\section{Expansive circle maps and conjugating homeomorphisms}\label{sec:epqcm}
In this section, we summarize the main results of \cite{LN24} on piecewise quasiconformal expansive circle maps and on extending a circle homeomorphism conjugating two such systems to a quasiconformal or David homeomorphism of the disk. See also \cite{LyubichMerenkovMukherjeeNtalampekos:David} for the case of piecewise analytic maps.

\subsection{Markov partitions}\label{section:markov_def}
If $a,b\in \mathbb{S}^1$, we denote by $\arc{[a,b]}$ and $\arc{(a,b)}$ the closed and open arcs, respectively, from $a$ to $b$ in the positive orientation. The arc $\arc{(b,a)}$, for example, is the complementary arc of $\arc{[a,b]}$. We also denote the arc $\arc{(a,b)}$ by $\inter{\arc{[a,b]}} $. We say that two non-overlapping arcs $I,J\subset \mathbb{S}^1$ are \textit{adjacent} if they share an endpoint.

\begin{definition}[Markov partition]\label{definition:markov_partition}
A \textit{Markov partition} associated to a covering map $f\colon \mathbb{S}^1\to \mathbb{S}^1$ is a covering of the unit circle by closed arcs $A_k=\arc{[a_k, a_{k+1}]}$, $k\in \{0,\dots, r\}$, $r\geq 1$, where $a_{r+1}=a_0$, that have disjoint interiors and satisfy the following conditions.
\begin{enumerate}[label={(\roman*)}]
\item\label{markov:i} The map $f_k=f|_{\inter{A_k}}$ is injective for $k\in \{0,\dots,r\}$.
\item\label{markov:ii} If $f(\inter{A_k})\cap \inter{A_j}\neq \emptyset$ for some $k,j\in \{0,\dots,r\}$, then $\inter{A_j}\subset f(\inter{A_k})$. 
\item\label{markov:iii} The set $\{a_0,\dots,a_r\}$ is invariant under $f$.
\end{enumerate}    
We denote the above Markov partition by $\mathcal P(f;\{a_0,\dots,a_r\})$. 
\end{definition}

Note that by definition, the points $a_0,\dots,a_r$ are ordered in the positive orientation if $r\geq 2$; if $r=1$, there is no natural order. Moreover, \ref{markov:ii} and \ref{markov:iii} are equivalent under condition \ref{markov:i}. 

Let $f\colon \mathbb{S}^1\to \mathbb{S}^1$ be a covering map and consider a Markov partition $\mathcal P=\mathcal P(f;\{a_0,\dots,a_r\})$. We can associate a matrix $B=(b_{kj})_{k,j=0}^r$ to $\mathcal P$ so that $b_{kj}=1$ if $f_k(A_k)\supset A_j$ and $b_{kj}=0$ otherwise. In the case $b_{kj}=1$ we define $A_{kj}$ to be the closure of $f_k^{-1}(\inter A_j)$. If $w=(j_1,\dots,j_n)\in \{0,\dots,r\}^n$, $n\in \N$, $k\in \{0,\dots,r\}$, and once $A_w$ has been defined, we define $A_{kw}$ to be the closure of $f_k^{-1}(\inter A_w)$ whenever $b_{kj_1}=1$. A \textit{word} $w=(j_1,\dots,j_n)\in \{0,\dots,r\}^n$, $n\in \N$, is \textit{admissible (for the Markov partition $\mathcal P$)} if $b_{j_1j_2} =\dots=b_{j_{n-1}j_n}=1$. The empty word $w$ is considered to be admissible and we define $A_w=\mathbb S^1$. We also define $A_w=\emptyset$ if $w$ is not admissible. The \textit{length} of a word $w=(j_1,\dots,j_n)\in \{0,\dots,r\}^n$ is defined to be $|w|=n$ and the length of the empty word is $0$. It follows from properties \ref{markov:i} and \ref{markov:ii} that for each $n\in \N$ the arcs $A_w$, where $|w|=n$, have disjoint interiors and their union is equal to $\mathbb{S}^1$. 

Inductively, we have $A_{wj}\subset A_w$ for all admissible words $w$ and $j\in \{0,\dots,r\}$. If $A_{wj}$ is non-empty, we say that $A_{wj}$ is a \textit{child} of $A_w$ and $A_w$ is the \textit{parent} of $A_{wj}$. Thus, $A_w$ has at most $r+1$ children. Also, if $k\in \{0,\dots,r\}$, $w$ is a non-empty word, and $(k,w)$ is admissible, then $f$ maps $\inter A_{kw}$ homeomorphically onto $\inter A_w$. Thus, $f$ acts as a subshift on the space of admissible words. 

We let $F_n$ be the preimages of $F_1= \{a_0,\dots,a_r\}$ under $n-1$ iterations of $f$ and $F_{0}=\emptyset$. Observe that $F_{n}\supset F_{n-1}$ for each $n\in \N$. For each $n\in \N$, the set $\mathbb S^1\setminus F_n$ consists of open arcs. For practical purposes we will use the terminology \textit{complementary arcs of $F_n$} to indicate the family of the \textit{closures} of the components of $\mathbb S^1 \setminus F_n.$ Hence, all complementary arcs of $F_n$ are closed arcs. Note that the complementary arcs of $F_n$ are the arcs $A_w$, where $w$ is an admissible word with $|w|=n$.

\subsection{Expansive maps}
We give the definition of an expansive map of $\mathbb S^1$.

\begin{definition}[Expansive map]\label{definition:expansive}
Let $f\colon \mathbb{S}^1\to \mathbb{S}^1$ be a continuous map. We say that $f$ is \textit{expansive} if there exists $\delta>0$ such that for every $a,b\in \mathbb{S}^1$ with $a\neq b$ we have $|f^{\circ n}(a)-f^{\circ n}(b)|>\delta$ for some $n\in \N\cup \{0\}$.
\end{definition}

An expansive map of $\mathbb S^1$ is necessarily a covering map of degree $d\geq 2$ \cite{HemmingsenReddy:expansive}*{Theorem 2}. 
Using Markov partitions we can give an equivalent definition of an expansive map. 
\begin{enumerate}[label={($E$\arabic*)}]
\item\label{expansive:diameters}
Let $f\colon \mathbb S^1\to \mathbb S^1$ be a continuous map and  $\mathcal P(f;\{a_0,\dots,a_r\})$ be a Markov partition. Then $f$ is expansive if and only if
\begin{align*}
\lim_{n\to\infty}\max\{\diam{A_w}: |w|=n \} =0.
\end{align*}
\end{enumerate}
This can be proved easily using \cite{PrzytyckiUrbanski:book}*{Theorem 3.6.1, p.~143}. We refer to \cite{LN24}*{Section 3.4} for a list of other important properties of expansive maps.

\subsection{Hyperbolic and parabolic points}\label{section:hyperbolic_parabolic_points}

Let $f\colon \mathbb S^1\to \mathbb S^1$ be a covering map and $\mathcal P(f;\{a_0,\dots,a_r\})$ be a Markov partition.

\begin{definition}[Hyperbolic points]\label{definition:hyperbolic}
Let $a \in\{a_0,\dots,a_r\}$. We say that $a^+$ (resp.\ $a^-$) is \textit{hyperbolic} if there exist $\lambda>1$ and $L\geq 1$ such that the following statement is true. 

\smallskip
\noindent
If $I_1,I_2$ are adjacent complementary arcs of $F_n$, $n\geq 1$, $a$ is an endpoint of $I_1$, and $I_1,I_2\subset \arc{[a,z_0]}$ (resp.\ $I_1,I_2\subset \arc{[z_0,a]}$) for some $z_0\neq a$, then for $i\in \{1,2\}$ we have
\begin{align}\label{inequality:adjacenthyperbolic}
L^{-1}\lambda^{-n}&\leq \diam {I_i} \leq L\lambda^{-n}.
\end{align}
In that case we set $\lambda(a^+)=\lambda$ (resp.\ $\lambda(a^-)=\lambda$) and call this number the \textit{multiplier} of $a^+$ (resp.\ $a^-$).
\end{definition}

\begin{definition}[Parabolic points]\label{definition:parabolic}
Let $a \in\{a_0,\dots,a_r\}$. We say that $a^+$  (resp.\ $a^-$) is \textit{parabolic} if there exist $N\in \N$ and $L\geq 1$ such that the following statement is true.

\smallskip
\noindent
If $I_1,I_2$ are adjacent complementary arcs of $F_n$, $n\geq 1$, $a$ is an endpoint of $I_1$, and $I_1,I_2\subset \arc{[a,z_0]}$ (resp.\ $I_1,I_2\subset \arc{[z_0,a]}$) for some $z_0\neq a$, then we have
\begin{align}\label{inequality:parabolic_1}
\frac{L^{-1}}{n^{1/N}}&\leq \diam {I_1} \leq \frac{L}{n^{1/N}} \quad \textrm{and}\quad \frac{L^{-1}}{n^{1/N +1}}\leq \diam {I_2} \leq \frac{L}{n^{1/N +1}}.
\end{align}
In that case we set $N(a^+)=N$ (resp.\ $N(a^-)=N$) and call the number $N(a^+)+1$ (resp.\ $N(a^-)+1$) the \textit{multiplicity} of $a^+$ (resp.\ $a^-$).
\end{definition}

\begin{definition}\label{definition:parabolic_hyperbolic_symmetric}
    Let $a\in \{a_0,\dots,a_r\}$. We say that $a$ is \textit{symmetrically hyperbolic} if $a^+$ and $a^-$ are hyperbolic with $\lambda(a^+)=\lambda(a^-)$. In this case we denote by $\lambda(a)$ the common multiplier. We say that $a$ is \textit{symmetrically parabolic} if $a^{+}$ and $a^-$ are parabolic with $N(a^+)=N(a^-)$. In this case we denote this common number by $N(a)$. 
\end{definition}

\begin{remark}\label{remark:hyp_par_analytic}
If $f\colon \mathbb S^1\to \mathbb S^1$ is analytic and expansive, it is shown in Lemma 4.17 and Lemma 4.18 in \cite{LyubichMerenkovMukherjeeNtalampekos:David} that each point $a\in \{a_0,\dots,a_r\}$ is either symmetrically hyperbolic or symmetrically parabolic. In that case, if $a$ is periodic and $q$ is its period, then $\lambda(a)^q$ (resp.\ $N(a)+1$) is the usual multiplier (resp.\ multiplicity) of the analytic map $f^{\circ q}$.
\end{remark}

\subsection{Extension theorem}\label{section:extension}
Let $f,g\colon \mathbb S^1\to \mathbb S^1$ be covering maps with corresponding Markov partitions $\mathcal P(f;\{a_0,\dots,a_r\})$, $\mathcal P(g;\{b_0,\dots,b_r\})$. We give some definitions below using the notation associated to the map $f$. First we require the following condition.
\begin{enumerate}[label=\normalfont(M1)]
\medskip
    \item\label{condition:hp} For each $a\in \{a_0,\dots,a_r\}$, each of $a^+$ and $a^-$ is either hyperbolic or parabolic, as defined in Section \ref{section:hyperbolic_parabolic_points}.
\medskip
\end{enumerate}
Define $A_k=\arc{[a_k,a_{k+1}]}$ for $k\in \{0,\dots,r\}$ and recall that $f_k=f|_{\inter{A_k}}$ is injective by the definition of a Markov partition. We impose the following condition.
\begin{enumerate}[label=\normalfont(M2)]
\medskip 
    \item\label{condition:uv} For each $k\in \{0,\dots,r\}$ there exist open neighborhoods $U_k$ of $\inter{A_k}$ and $V_k$ of $f_k(\inter{A_k})$ in $\C$ such that $f_k$ has an extension to a {homeomorphism} from $U_k$ onto $V_k$. We denote the extension by $f_k$. Furthermore, we assume that
    \begin{align*}
    \bigcup_{\substack{0\leq j\leq r\\(k,j) \,\, \textrm{admissible}}}U_j \subset V_k
    \end{align*} 
    for all $k\in \{0,\dots,r\}$ and that the regions $U_k$, $k\in \{0,\dots,r\}$, are pairwise disjoint. 
\medskip
\end{enumerate}
We denote by $f$ the map that is equal to $f_k$ on $U_k$, $k\in \{0,\dots,r\}$. 
Using condition \ref{condition:uv}, for each admissible word $w$ we can find open regions $U_w$ with the following properties:
\begin{enumerate}[label=\normalfont{(\roman*)}]
\item $U_{wj} \subset U_w$, if $(w,j)$ is admissible, and 
\item $f$ maps $U_{kw}$ homeomorphically onto $U_w$, if $(k,w)$ is admissible. 
\end{enumerate}
Next, we impose the following condition.
\begin{enumerate}[label=\normalfont(M3)]
\medskip
    \item\label{condition:qs} There exists $K\geq 1$ such that for each $m\in \N\cup \{0\}$ and each admissible word $w$ with $|w|\geq m$ the map $f^{\circ m}|_{U_w}$ is $K$-quasiconformal.  
\medskip 
\end{enumerate}
Finally, we consider a technical condition that is a {refinement} of \ref{condition:qs} and compensates for the unavailability of Koebe's distortion theorem in our setting. For the sake of simplicity we state here a stronger version of this condition than the one given in \cite{LN24}, which is sufficient for our purpose.
\begin{enumerate}[label=\normalfont(M$3^\ast$)]
\medskip
    \item  \label{condition:qs_strong} There exists $C>0$ such that for each $n,m\in \N\cup \{0\}$ and each admissible word $w$ with $|w|=n+m$ the map $f^{\circ m}|_{U_w}$ is $K_n$-quasiconformal for $K_n=1+C(1+\log (n+1))^{-1}$. 
\medskip
\end{enumerate}

The following extension theorem is the main theorem in \cite{LN24}, stated in a slightly simplified form here.

\begin{theorem}[\cite{LN24}*{Theorem 4.1}]\label{theorem:extension_generalization}Let $f,g\colon \mathbb{S}^1\to \mathbb{S}^1$ be expansive covering maps with the same orientation and $\mathcal P(f;\{a_0,\dots,a_r\})$, $\mathcal P(g;\{b_0,\dots,b_r\})$ be Markov partitions satisfying conditions \ref{condition:hp}, \ref{condition:uv}, and \ref{condition:qs}. Suppose that the map $h\colon \{a_0,\dots,a_r\} \to \{b_0,\dots,b_r\}$ defined by $h(a_k)=b_k$, $k\in \{0,\dots,r\}$, conjugates $f$ to $g$ on the set $\{a_0,\dots,a_r\}$ and assume that for each point $a\in \{a_0,\dots,a_r\}$ and for $b=h(a)$ one of the following alternatives occurs.

\smallskip
    
\begin{enumerate}[label= {\textup{\scriptsize(\textbf{H/P${\rightarrow}$H/P})}}, leftmargin=6em] 
\item\label{HH} 
$a$ is symmetrically hyperbolic (resp.\ parabolic) if and only if $b$ is symmetrically hyperbolic (resp.\ parabolic).
\end{enumerate}

\begin{enumerate}[label={\textup{\scriptsize(\textbf{H$\to$P})}}, leftmargin=6em]
\item\label{HP} $a$ is symmetrically hyperbolic and $b$ is symmetrically parabolic.
\end{enumerate}

\smallskip 
\noindent
If the alternative \ref{HP} does not occur, then the map $h$ extends to a  homeomorphism $\widetilde h$ of $\bar \D$ such that $\widetilde h|_{\mathbb S^1}$ conjugates $f$ to $g$ and $\widetilde h|_{\D}$ is a quasiconformal map. 

\smallskip 
\noindent
If $g$ satisfies condition \ref{condition:qs_strong}, then the map $h$ extends to a homeomorphism $\widetilde h$ of $\bar \D$ such that $\widetilde h|_{\mathbb S^1}$ conjugates $f$ to $g$ and $\widetilde h|_{\D}$ is a David map. 
\end{theorem}

\section{Subdivision rules}\label{sec:cpsr}

In this section we summarize some elementary results on subdivision rules. We refer the reader to \cite{LuoZhang:CirclePacking} for more details.

\subsection{Subdivision rules}
We first recall that a CW complex $Y$ is a \textit{subdivision} of a CW complex $X$ if $X = Y$ (as topological spaces) and every closed cell of $Y$ is contained in a closed cell of $X$.
We define a \textit{polygon} as a finite CW complex homeomorphic to a closed disk 
that consists of one 2-cell, {whose 1-skeleton is a Jordan curve} with at least three 0-cells.
We will also call 0-cells, 1-cells, and 2-cells the {vertices, edges}, and {faces} respectively. We say two vertices are \textit{adjacent} if they are on the boundary of an edge and \textit{non-adjacent} otherwise.
\begin{definition}\label{defn:fsr}
A \textit{finite subdivision rule} $\mathcal{R}$ consists of 
\begin{enumerate}[label=(\arabic*)]
\item\label{defn:fsr:1} a finite collection of {oriented} polygons $\{P_i: i=1,\dots, k\}$,
\item\label{defn:fsr:2} a subdivision $\mathcal{R}(P_i)$ that decomposes each polygon $P_i$ into $m_i \geq 2$ closed 2-cells
$$P_i = \bigcup_{j=1}^{m_i} P_{i, j}$$
so that each edge of $\partial P_i$ contains no vertices of $\mathcal{R}(P_i)$ in its interior and each $2$-cell $P_{i,j}$ inherits the orientation of $P_i$, and
\item\label{defn:fsr:3} a map $\sigma \colon \bigcup_{i=1}^k (\{i\}\times \{1,\dots,m_i\})\to \{1,\dots,k\}$ and a collection of {ori\-en\-tation-preserving} cellular maps, denoted by
$$
\psi_{i,j}: P_{\sigma(i,j)} \to P_{i,j},\quad i\in \{1,\dots,k\}, \,\, j\in \{1,\dots,m_i\},
$$
that are homeomorphisms between the open 2-cells. We call $P_{\sigma(i,j)}$ the \textit{type} of the closed 2-cell $P_{i,j}$.
\end{enumerate}
\end{definition}
For simplicity, we use the notation $\mathcal R=\{P_i\}_{i=1}^k$ for the finite subdivision rule $\mathcal R$, suppressing the decomposition in \ref{defn:fsr:2} and the cellular maps in \ref{defn:fsr:3}.

By condition \ref{defn:fsr:3}, we can iterate the subdivision procedure and obtain for each $n\in \N\cup \{0\}$ a decomposition $\mathcal{R}^n(P_i)$ of $P_i$, where $\mathcal{R}^0(P_i)$=$P_i$ and $\mathcal R^1(P_i)=\mathcal R(P_i)$. By condition \ref{defn:fsr:2}, we can identify the $1$-skeleton $\mathcal{G}^n(P_i)$ of $\mathcal{R}^{n}(P_i)$ as a subgraph of the $1$-skeleton $\mathcal{G}^{n+1}(P_i)$ of $\mathcal{R}^{n+1}(P_i)$.
We consider the direct limit 
$$
\mathcal{G}(P_i) = \lim_{\rightarrow} \mathcal{G}^n(P_i)= \bigcup_{n=0}^{\infty} \mathcal{G}^n(P_i)
$$ 
for $i\in \{1,\dots,k\}$. We call $\mathcal{G}(P_i)$, $i\in \{1,\dots,k\}$, the \textit{subdivision graphs} for $\mathcal{R}$. We consider each polygon $P_i$ as a subset of $\widehat\C$ and we call the set $F_i\coloneqq\overline{\widehat\C \setminus P_i}$ the \textit{external face}. Note that $F_i$ is not subdivided, and remains a face of $\mathcal{G}^n(P_i)$ for all $n\in \N\cup \{0\}$. 

\begin{definition}\label{defn:cyl}
Let $\mathcal{R}=\{P_i\}_{i=1}^k$ be a finite subdivision rule. 
\begin{itemize}
    \item We say that $\mathcal R$ is \textit{simple} if, for each $i\in \{1,\dots,k\}$, $\mathcal{G}(P_i)$ is a simple graph, i.e., no pair of vertices is connected by multiple edges and no edge connects a vertex to itself.
    \item We say that $\mathcal R$ is \textit{irreducible} if, for each $i\in \{1,\dots,k\}$, $\partial P_i$ is an induced subgraph of $\mathcal{G}(P_i)$, i.e., $\partial P_i$ contains all edges of $\mathcal{G}(P_i)$ that connect vertices in $\partial P_i$.
    \item We say that $\mathcal R$ is \textit{acylindrical} if, for each $i\in \{1,\dots,k\}$, for any pair of non-adjacent vertices $v,w \in \partial P_i$, the two components of $\partial P_i \setminus \{v,w\}$ are connected in $\mathcal{G}(P_i) \setminus \{v, w\}$.
    We call $\mathcal R$ \textit{cylindrical} otherwise.
\end{itemize}
\end{definition}

By cutting each polygon $P_i$ into finitely many pieces if necessary, we can always make a finite subdivision rule $\mathcal R$ irreducible.

\begin{figure}
    \centering
    \begin{subfigure}[b]{0.7\textwidth}
        \centering
        \includegraphics[width=\textwidth]{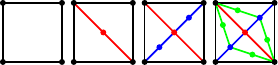}
    \caption{}
    \label{fig:subda}
    \end{subfigure}
    \begin{subfigure}[b]{0.7\textwidth}
        \centering
        \includegraphics[width=\textwidth]{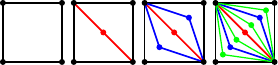}
    \caption{}
    \label{fig:subdb}
    \end{subfigure}
    \caption{An example of an acylindrical subdivision rule on the top, and a cylindrical subdivision rule on the bottom. The top subdivision satisfies condition \ref{condition:si} if we take the second iterate $\mathcal{R}^2$.}
    \label{fig:subd}
\end{figure}

{We remark that in Definition \ref{defn:fsr}, the identification maps $\psi_{i,j}$ are only assumed to be homeomorphisms between the open 2-cells. Thus, a face $F^n$ of $\mathcal{G}^n(P_i)$ may not be a  polygon, i.e., the $1$-skeleton of the face need not be a Jordan curve. The following proposition says that if $\mathcal{R}$ is simple, irreducible, and acylindrical, then one can always enlarge the subdivision and modify the identification maps so that every face is a polygon.

\begin{proposition}\label{prop:jd}
    Let $\mathcal{R}=\{P_i\}_{i=1}^k$ be a simple, irreducible, acylindrical finite subdivision rule. Then there exist $K\in \N$ and a simple, irreducible, acylindrical finite subdivision rule $\widetilde{\mathcal{R}} = \{P_i\}_{i=1}^k \cup \{Q_i\}_{i=1}^l$ with the following properties.
    \begin{enumerate}[label=\normalfont(\arabic*)]
        \item For each $i \in \{ 1,\dots, k\}$ and $n \geq 1$, we have
        $$
        \mathcal{G}^{nK}(P_i) \subset \widetilde{\mathcal{G}}^n(P_i) \subset \mathcal{G}^{(n+1)K}(P_i).
        $$
        In particular, $\mathcal{G}(P_i) = \widetilde{\mathcal{G}}(P_i)$.
        \item For all $i\in \{1,\dots, k\}$ (resp.\ $j\in\{1,\dots, l\}$) and $n \geq 1$, each face $F^n$ of $\widetilde{\mathcal{G}}^n(P_i)$ (resp.\ $\widetilde{\mathcal{G}}^n(Q_j)$) is a polygon and $\partial F^n$ is an induced subgraph of $\widetilde{\mathcal{G}}^n(P_i)$ (resp.\ $\widetilde{\mathcal{G}}^n(Q_j)$).
    \end{enumerate}
\end{proposition}
The proof can be found in \cite{LuoZhang:CirclePacking}*{Proposition 4.3}.
Although the assertion that $\partial F^n$ is an induced subgraph is not included in \cite{LuoZhang:CirclePacking}*{Proposition 4.3}, the proof is identical to the proof of $F^n$ being a polygon.
}

\subsection{Standing assumptions on subdivision rules}
We will impose some additional assumptions on finite subdivision rules when necessary (e.g., in Sections \ref{sec:cp} and \ref{sec:gmp}). These assumptions are not restrictive in the sense that they can always be achieved by slightly modifying the subdivision rule or by taking an iterate. We will state explicitly when these assumptions are imposed.

Let $\mathcal{R} = \{P_i\}_{i=1}^k$ be a simple, irreducible, acylindrical finite subdivision rule.
By replacing $\mathcal{R}$ with $\widetilde{\mathcal{R}}$ from Proposition \ref{prop:jd}, we will impose the following assumption when necessary.
\begin{enumerate}[label=\normalfont(S1)]
    \item\label{condition:siL} For each $i\in \{1,\dots,k\}$ and $n\geq 0$, each face $F^n$ of $\mathcal{G}^n(P_i)$ is a polygon and $\partial F^n$ is an induced subgraph of $\mathcal{G}^n(P_i)$.
\end{enumerate}

Since $\mathcal{R}$ is acylindrical, there exists $K\in \N$ so that for each $i\in \{1,\dots,k\}$ and for each pair of non-adjacent vertices $v, w \in \partial P_i$, the two components of $\partial P_i \setminus \{v, w\}$ are connected in $\mathcal{G}^K(P_i) \setminus \{v, w\}$.
This implies that for each $i\in \{1,\dots,k\}$ and each face $F^K$ of $\mathcal{G}^K(P_i)$, the set $\partial P_i \cap \partial F^K$ does not contain two non-adjacent vertices of $P_i$; hence, $\partial P_i \cap \partial F^K$ is contained in an edge of $\partial P_i$.
By the properties of the subdivision rule, if $n\geq 0$ and $F^{n+K},F^n$ are faces of $\mathcal{G}^{n+K}(P_i),\mathcal{G}^n(P_i)$, respectively, with $F^{n+K}\subset F^n$, then $\partial F^{n+K} \cap \partial F^n$ is contained in an edge of $\partial F^n$.
Thus by replacing $\mathcal{R}$ with $\mathcal{R}^K$, we may assume that for each $i\in \{1,\dots,k\}$ and $n\geq 0$ the following statement is true.
 
\begin{enumerate}[label=\normalfont(S2)]
    \item\label{condition:si} If $F^{n+1},F^n$ are faces of $\mathcal{G}^{n+1}(P_i),\mathcal{G}^n(P_i)$, respectively, with $F^{n+1}\subset F^n$, then $\partial F^{n+1} \cap \partial F^n$ is either empty, or a vertex, or an edge.
\end{enumerate}

\subsection{$\mathcal{R}$-complexes}
Let $\mathcal R$ be a finite subdivision rule. An $\mathcal{R}$-complex is a finite CW complex $X$ so that for any closed 2-cell $s$ of $X$, there exist $P_{\sigma(s)} \in \{P_i\}_{i=1}^k$ and a cellular map $\psi_s: P_{\sigma(s)} \to s$ that is a homeomorphism between the open 2-cells. We also require that each $2$-cell $s$ of $X$ is oriented and $\psi_s$ preserves orientation (recall that each polygon of $\mathcal R$ is oriented).

The subdivision rule gives a sequence of subdivisions $\mathcal{R}^n(X)$ for $X$, $n\geq 0$.
We always assume that $\mathcal{R}$ is {a} \textit{minimal} finite subdivision rule that supports $X$, i.e., every polygon in $\{P_i\}_{i=1}^k$ appears as the type of some closed 2-cell of $\mathcal{R}^n(X)$ for some $n\geq 0$. For each $n\geq 0$ the 1-skeleton $\mathcal G^n(X)$ of the subdivision $\mathcal{R}^n(X)$ of $X$ is contained in the $1$-skeleton $\mathcal G^{n+1}(X)$ and we consider the direct limit
$$
\mathcal{G}(X) = \lim_{\rightarrow} \mathcal{G}^n(X) = \bigcup_{n=0}^\infty \mathcal{G}^n(X).
$$
If the $\mathcal{R}$-complex $X$ is homeomorphic to a sphere, the plane graph $\mathcal{G}(X)$ is called a \textit{spherical subdivision graph for $\mathcal{R}$}. By enlarging the subdivision rule (see \cite{LuoZhang:CirclePacking}*{Proposition 4.4}), we will impose the following assumption on a spherical subdivision graph $\mathcal G(X)$ when necessary.
\begin{enumerate}[label=\normalfont(S1$^\prime$)]
    \item\label{condition:siG} For each $n\geq 0$, each face $F^n$ of $\mathcal{G}^n(X)$ is a polygon and $\partial F^n$ is an induced subgraph of $\mathcal{G}^n(X)$.
\end{enumerate}

\subsection{Gaskets with finite subdivision rules}\label{section:gasket_subdivision}
Let $J\subset \widehat \C$ be a gasket, as in Definition \ref{definition:gasket}. Denote by $\mathcal G$ the contact graph of $J$. We consider the graph $\mathcal G$ as an embedded subset of $\widehat \C$ so that each vertex of $\mathcal G$ is a point in a component of $\widehat\C\setminus J$ and each edge of $\mathcal G$ is a Jordan arc between two vertices that intersects $J$ at a single point. Suppose that there exists a finite subdivision rule $\mathcal R$   so that $\mathcal G$ is a spherical subdivision graph for $\mathcal R$. This means that there is a CW complex $X$ with $1$-skeleton $\mathcal G^0$ so that each $2$-cell of $X$ corresponds under (the inverse of) a cellular map to a polygon of $\mathcal R$. Then one obtains the $1$-skeletons $\mathcal G^n$ of the subdivisions $\mathcal R^n(X)$, $n\geq 0$, of $X$ and $\mathcal G$ is the direct limit of $\mathcal G^n$ as $n\to\infty$. 

We note that for each $n\geq 0$ the interior of each face $F^n$ (i.e.\ the open $2$-cell of the CW complex $F^n$) of $\mathcal G^n$ naturally embeds in $\widehat \C$. We define $W(F^n)\subset \widehat\C$ to be the open set that is the union of the embedding of the interior of $F^n$ into $\widehat \C$ and of the components of $\widehat \C\setminus J$ corresponding to vertices of $\partial F^n$. We also define $W(\inter F^n)$ to be the set $W(F^n)$ minus the closures of the components of $\widehat \C\setminus J$ corresponding to vertices of $\partial F^n$. The following lemma gives a necessary condition for the contact graph of a gasket to be a spherical subdivision graph (c.f.\ Theorem \ref{thm:qch} \ref{RealizationAcylindrical}).

\begin{lemma}\label{lemma:gasket_subdivision_rule}
    Let $J\subset \widehat  \C$ be a gasket whose contact graph $\mathcal G$ is a spherical subdivision graph for an irreducible finite subdivision rule $\mathcal R$. Then $\mathcal R$ is simple and acylindrical. 
\end{lemma}
\begin{proof}
    Since $J$ is a gasket, there are no multiple edges connecting two vertices, and no edge connecting a vertex to itself in $\mathcal{G}$, so $\mathcal{R}$ is simple. We prove by contradiction that $\mathcal{R}$ is acylindrical. Suppose that $\mathcal{R}$ is cylindrical. Then by \cite{LuoZhang:CirclePacking}*{Proposition 4.5}, there exist a number $K>0$, a polygon $P_i$ of $\mathcal R$, and a pair of non-adjacent vertices $v, w \in \partial P_i$ so that there are infinitely many paths $\gamma_n$, $n\in \N$, of combinatorial length at most $K$ in $\mathcal{G}(P_i)$ connecting $v, w$ that are pairwise disjoint except at the endpoints; see Figure \ref{fig:subdb} for an illustration.

    Recall that $\mathcal R$ is assumed to be a minimal finite subdivision rule that supports $X$, so $P_i$ appears as the type of a $2$-cell of $\mathcal R^n(X)$ for some $n\geq 0$. By replacing $\mathcal R$ with an iterate, we assume that a face $F^0$ of $\mathcal{G}^0$ is associated to the polygon $P_i$. Let $a,b$ be the vertices of $\partial F^0$ associated to $v,w$, respectively. We remark that we might have $a=b$ but for the simplicity of the presentation, we assume that $a\neq b$.
    
    Let $\Omega_a, \Omega_b$ be the complementary components of $\widehat \C\setminus J$ associated to $a, b$, respectively.  Let $I_a,I_b$ be the intersection of $\Omega_a,\Omega_b$, respectively, with the embedding of the interior of the face $F^0$ in $\widehat \C$.  Since $v,w$ are not adjacent,  we have  $\dist(I_a, I_b) > 0$.  For $n\in \N$, let $aa_1^n, a_1^na_2^n, \dots , a_{k_n}^nb$ be the chain of edges in $F^0 \cap \mathcal{G}$ associated to $\gamma_n$.  Note that $k_n \leq K-1$. Let $\Omega^n_j$ be the component of $\widehat \C\setminus J$ that contains $a_j^n$, $j\in \{1,\dots,k_n\}$.  Note that the chain of components $\Omega^n_1,\dots, \Omega^n_{k_n}$ connects $I_a$ and $I_b$. Thus one of the components has diameter at least $ \dist(I_a, I_b)/(K-1)$. Since the paths $\gamma_n$, $n\in \N$, are pairwise disjoint except at the endpoints, there exist infinitely many components of $\widehat \C\setminus J$ with diameter larger than $\dist (I_a, I_b)/(K-1)$. This contradicts the local connectivity of the gasket $J$; see \cite{Milnor:dynamics}*{Problem 19-f}.

    We remark that if $a = b$, then each path $\gamma_n$ corresponds to a chain of components connecting $\partial \Omega_a$ to itself. Given that $v$ and $w$ are not adjacent, we see that the union of the components in this chain has diameter uniformly bounded from below, so the above argument can be applied with appropriate modifications.  
\end{proof}

A gasket $J$ is \textit{quasiround} if there exists $M\geq 1$ such that for each component $U$ of $\widehat \C\setminus J$ we have $E(U)\leq M$; that is, there exists a spherical ball $B$ such that  $B\subset U\subset MB$; see also Section \ref{section:abs}.

\begin{lemma}\label{lemma:gasket_sub}
    Let $J\subset \widehat  \C$ be a gasket whose contact graph $\mathcal G$ is a spherical subdivision graph for a simple, irreducible, acylindrical finite subdivision rule $\mathcal R$.
    \begin{enumerate}[label=\normalfont(\arabic*)]
        \item\label{gs:0}We have $\max\{ \diam W(\inter F^n): \textrm{$F^n$ face of $\mathcal G^n$}\}\to 0$ as $n\to\infty$.
        \item\label{gs:1} For each point $z\in J$ that does not lie on the boundary of any component of $\widehat \C\setminus J$ there exists a sequence of faces $F^n$ of $\mathcal G^n$, $n\geq 0$, such that
        $$\diam W(F^n)\to 0\,\,\textrm{as $n\to \infty$ and }\{z\}=\bigcap_{n=0}^\infty W(F^n).$$ 
        \item\label{gs:2} If $J$ is quasiround, then there exists $M\geq 1$ such that for each face $F^n$ of $\mathcal G^n$, $n\geq 0$, there exists a ball $B\subset \widehat \C\setminus J$ with $B\subset W(F^n) \subset MB$.
        \item\label{gs:3} If $J$ is quasiround and the boundary of each component of $\widehat \C\setminus J$ has area zero, then $J$ has area zero.  
    \end{enumerate}
\end{lemma}

\begin{proof}
    For the proof of \ref{gs:0}, by replacing $\mathcal R$ with an iterate, it suffices to assume that $\mathcal R$ satisfies condition \ref{condition:si}. We argue by contradiction. Suppose that there exists a sequence of faces $F^n$ of $\mathcal G^n$, $n\geq 0$, such that $\diam W(\inter F^n)$ is bounded below away from $0$. We may assume that $F^{n+1}\subset F^n$ for all $n\geq 0$.  In view of \ref{condition:si}, there are three cases: (a) no vertex of $\mathcal G$ lies in $\partial F^n$ for infinitely many $n\geq 0$; (b) there exists a unique vertex $v\in \mathcal G$ that lies in $\partial F^n$ for all but finitely many $n\geq 0$; (c) there exist exactly two vertices $v,w\in \mathcal G$ that lie in $\partial F^n$ for all but finitely many $n\geq 0$, in which case, $v$ and $w$ are adjacent. By the definition of a subdivision rule (see Definition \ref{defn:fsr}), $\partial F^n$ consists of a uniformly bounded number of vertices for all $n\geq 0$. Thus, given that $J$ is locally connected, the diameter of the union of components of $\widehat \C\setminus J$ corresponding to vertices of $\partial F^n$, except for the ones corresponding to $v$ and $w$ in the cases (b) and (c), tends to $0$ as $n\to\infty$. If (a) is true, then the above remark implies that $\diam W(F^n)\to 0$, a contradiction. If (b) is true, then one can see that $W(\inter F^n)$ would converge to a point on the boundary of the component of $\widehat \C\setminus J$ corresponding to $v$, a contradiction to the initial assumption. Similarly, if (c) is true, then $W(\inter F^n)$ would converge to the intersection point of the boundaries of the components corresponding to $v,w$, a contradiction. This completes the proof of \ref{gs:0}.
    
    For the proof of \ref{gs:1}, let $z\in J$ be a point that does not lie on the boundary of any component of $\widehat \C\setminus J$. There exists a unique nested sequence of faces $F^n$ of $\mathcal G^n$, $n\geq 0$, such that $z$ lies in the embedding of the interior of $F^n$ into $\widehat \C$. By performing the above analysis, we see that only alternative (a) can hold. Hence, $\diam W(F^n)\to 0$ and $\{z\} = \bigcap_{n=0}^\infty W(F^n)$.
    
    Let $F^n$ be an arbitrary face of $\mathcal G^n$, $n\geq 0$. Let $D^n$ be a largest component of $\widehat \C\setminus J$ among components corresponding to the vertices of $\partial F^n$. Recall that $\partial F^n$ consists of a uniformly bounded number of vertices, say at most $K$. Therefore, if $W(\inter F^n)$ is sufficiently small (given that we are dealing with subsets of the sphere), then
    \begin{align}\label{lemma:gasket_sub:inclusions}
    D^n\subset W(F^n)\subset B(w, K\diam D^n)
    \end{align}
    for any point $w\in D^n$.  By \ref{gs:0}, there exists $N\in \N$ such that \eqref{lemma:gasket_sub:inclusions} holds for $n\geq N$.  If, in addition $J$ is quasiround, as in \ref{gs:2}, then there exists a uniform $L\geq 1$ such that 
    $$B=B(w,L^{-1}\diam D^n)\subset D^n \subset B(w,K\diam D^n)$$
    for all $n\geq N$. Also, the above inclusions hold trivially with different constants for $n<N$. This proves \ref{gs:2}.  
    
    Finally, we show \ref{gs:3}. By \ref{gs:2} we see that points of $J$ that do not lie on the boundary of any component of $\widehat \C\setminus J$ are not Lebesgue density points of $J$. Since, by assumption, the boundaries of the countably many components of $\widehat \C\setminus J$ have area zero, we conclude that $J$ has area zero.
\end{proof}

\begin{theorem}\label{theorem:qc_david_gasket}
    Let $J\subset \widehat\C$ be a gasket whose contact graph is a spherical subdivision graph for a simple, irreducible, acylindrical finite subdivision rule. Let $f\colon \widehat \C\to \widehat \C$ be a homeomorphism. Suppose that the components of $\widehat \C\setminus J$ and $\widehat \C\setminus f(J)$ are uniform quasidisks.       
    \begin{enumerate}[label=\normalfont(\arabic*)]
        \item If $f$ is quasiconformal on $\widehat\C\setminus J$, then $f$ is quasiconformal on $\widehat\C$.
        \item If $f$ is a David map on $\widehat\C\setminus J$, then $f$ is a David map on $\widehat\C$.
    \end{enumerate}
\end{theorem}
\begin{proof}
    Since the components of $\widehat \C\setminus J$ and $\widehat \C\setminus f(J)$ are uniform quasidisks, we see that the gaskets $J$ and $f(J)$ are quasiround; see \cite{Bonk:uniformization}*{Proposition 4.3} for an argument. By Lemma \ref{lemma:gasket_sub} \ref{gs:3}, $J$ has area zero. Let $z\in J$ be a point that does not lie on the boundary of any component of $\widehat \C\setminus J$.  By \ref{gs:1} and \ref{gs:2} in Lemma \ref{lemma:gasket_sub} there exists a sequence of faces $F^n$ of $\mathcal G^n$, $n\geq 0$, shrinking to $z$ such that $E(W(F^n))\leq M$ for all $n\geq 0$, where $M\geq 1$ is a uniform constant. By applying Lemma \ref{lemma:gasket_sub} \ref{gs:2} to the quasiround gasket $f(J)$, we see that $E(f(W(F^n)))\leq M'$ for all $n\geq 0$ and for a uniform constant $M'\geq 1$. We conclude that $E_f(z)\leq \max\{M,M'\}$ for every such point $z$. Finally, we invoke Theorem \ref{theorem:qc_david_remov} to obtain the conclusion. 
\end{proof}

\section{Circle packings with finite subdivision rules}\label{sec:cp}

We assume that all finite subdivision rules appearing in this section satisfy the standing assumptions \ref{condition:siL}, \ref{condition:siG}, and \ref{condition:si} from Section \ref{sec:cpsr}.

Consider a collection $\mathcal{P}$ of geometric disks of $\widehat\C$ with disjoint interiors.
We can associate to it a graph $\mathcal{G}$, called the \textit{tangency graph} of $\mathcal{P}$.
The tangency graph $\mathcal{G}$ assigns a vertex $v$ to each disk $D_v$ of $\mathcal{P}$ and an edge between $v, w$ if $\partial D_v \cap \partial D_w \neq \emptyset$.
The tangency graph $\mathcal{G}$ naturally embeds in $\widehat\C$, so it is a {plane graph}.
We say that the collection $\mathcal{P}$ is a \textit{circle packing} if the tangency graph $\mathcal{G}$ is connected.
The  \textit{limit set} $\Lambda(\mathcal{P})$ of $\mathcal P$ is the closure of the union of all boundary circles corresponding to disks in $\mathcal{P}$. {Note that the limit set $\Lambda(\mathcal P)$ of a circle packing $\mathcal P$ is a gasket according to Definition \ref{definition:gasket}. In the following considerations we will frequently refer to circles in $\mathcal P$, meaning boundaries of disks in $\mathcal P$.
We say that two plane graphs are isomorphic if they are isomorphic as graphs, and the isomorphism respects the embeddings in $\widehat\C$.

In this paper, we will always \textit{mark} the circle packing $\mathcal{P}$ with tangency graph $\mathcal{G}$, i.e, label each circle in $\mathcal{P}$ with tangency graph $\mathcal{G}$ by the corresponding vertex in $\mathcal{G}$.
By a homeomorphism $\psi$ between two marked circle packings $\mathcal{P}, \mathcal{P}'$ with tangency graph $\mathcal{G}$, we mean an {blue}{orientation-preserving} homeomorphism $\psi$ of the sphere $\widehat \C$ such that $\psi(\Lambda(\mathcal{P})) = \Lambda(\mathcal{P}')$ and $\psi$ preserves the labeling.

\begin{definition}
    Let $\mathcal{R}$ be a finite subdivision rule and $\mathcal P$ be a circle packing.
    \begin{itemize}
        \item We say that $\mathcal P$ is a circle packing with finite subdivision rule of $\mathcal{R}$ if its tangency graph is a subdivision graph for $\mathcal{R}$.
        \item We say that $\mathcal P$ is a circle packing with spherical subdivision rule of $\mathcal{R}$ if its tangency graph is a spherical subdivision graph for $\mathcal{R}$.
    \end{itemize}   
\end{definition}

Let $\mathcal R$ be a finite subdivision rule and $\mathcal{G}$ be a subdivision graph or a spherical subdivision graph for $\mathcal{R}$.  Let $\mathcal P$ be a circle packing with tangency graph $\mathcal G$. We denote by $\mathcal{M}(\mathcal{G})$ the (marked) moduli space of circle packings with tangency graph $\mathcal{G}$. More precisely, $\mathcal{M}(\mathcal{G})$ consists of equivalence classes of marked circle packings {whose tangency graph is isomorphic to $\mathcal{G}$ as plane graphs}, where two marked circle packings $\mathcal{P}$ and $\mathcal{P}'$ are equivalent if there exists a M\"obius map sending each circle $C$ of $\mathcal{P}$ to the corresponding circle $C'$ in $\mathcal{P}'$. Abusing notation, we will often write $\mathcal{P} \in \mathcal{M}(\mathcal{G})$ to represent a marked circle packing with tangency graph $\mathcal{G}$. 
 
\begin{theorem}[Realization and uniqueness; \cite{LuoZhang:CirclePacking}*{Theorems A, E}]\label{thm:qch}
Let $\mathcal{R}=\{P_i\}_{i=1}^k$ be a simple, irreducible finite subdivision rule.
\begin{enumerate}[label=\normalfont(\roman*)]
\item\label{thm:qch:i} Each subdivision graph $\mathcal{G}(P_i)$, $i\in \{1,\dots,k\}$, is isomorphic to the tangency graph of an infinite circle packing if and only if $\mathcal{R}$ is acylindrical.
\item\label{thm:qch:ii} If $\mathcal R$ is acylindrical, then the following statements are true for each $i\in \{1,\dots,k\}$.
\begin{enumerate}[label=\normalfont(\alph*)]
    \item\label{thm:qch:a} Any two marked circle packings $\mathcal{P}_i, \mathcal{P}_i'$ with tangency graph $\mathcal{G}(P_i)$ are quasiconformally homeomorphic, i.e., there exists a quasiconformal map $\Psi\colon \widehat\C \to \widehat\C$ between $\mathcal{P}_i$ and $\mathcal{P}_i'$.
    
    \item The moduli space $\mathcal{M}(\mathcal{G}(P_i))$ is isometrically homeomorphic to the Teichm\"uller space $\Teich(\Pi_{F_i})$ of the ideal polygon associated to the external face of $F_i$. 
\end{enumerate}
\end{enumerate}
Let $\mathcal{G}(X)$ be a simple spherical subdivision graph for $\mathcal{R}$.
\begin{enumerate}[label=\normalfont(\Roman*)]
\item {\normalfont (Realization)}\label{RealizationAcylindrical} The graph $\mathcal{G}(X)$ is isomorphic to the tangency graph of a circle packing $\mathcal{P}$ if and only if $\mathcal{R}$ is acylindrical.
\item {\normalfont (Combinatorial rigidity)} If the graph $\mathcal{G}(X)$ is isomorphic to the tangency graph of a circle packing, then this circle packing is unique up to M\"obius maps.
\end{enumerate}
\end{theorem}
The metric used in part \ref{thm:qch:ii} of the above theorem is the \textit{Teichm\"uller distance} in $\mathcal{M}(\mathcal{G}(P_i))$, namely
$$
d(\mathcal{P}, \mathcal{P}') \coloneqq  \frac{1}{2}\log K,
$$
where $K$ is the smallest dilatation of quasiconformal homeomorphisms between $\mathcal{P}$ and $\mathcal{P}'$.

\subsection{Markov Tiles}\label{subset:mt}
Let $\mathcal{R} = \{P_i\}_{i=1}^k$ be a simple, irreducible, acylindrical finite subdivision rule.
Let $i\in \{1,\dots,k\}$ and $\mathcal{G}=  \mathcal{G}(P_i)$ be a subdivision graph for $\mathcal{R}$ or let $\mathcal G=\mathcal G(X)$ be a spherical subdivision graph for $\mathcal R$.   Consider a a non-external face $F$ of $\mathcal{G}^n$ for some $n\geq 0$. We fix a circle packing $\mathcal{P} \in \mathcal{M}(\mathcal{G})$. {By Assumptions \ref{condition:siL} and \ref{condition:siG}, the circles corresponding to the vertices of $\partial F$ form a chain, i.e., each circle is tangent to exactly two circles.}
Let $\Gamma_{F}$ be the group generated by reflections along circles associated to the finite set of vertices in $\partial F$. Then the index two subgroup of $\Gamma_{F}$ consisting of orientation-preserving elements is a quasi-Fuchsian group, and the limit set of $\Gamma_{F}$ is a quasicircle; see \cite{Bishop:welding}*{Section 4} for an argument. The domain of discontinuity of $\Gamma_{F}$ has two components, and one has non-empty intersection with circles in $\mathcal{P}$ that correspond to vertices of $\mathcal G$ lying in $F\setminus \partial F$.
We denote this component by $\Omega_F$ and the limit set by $\Lambda_F = \partial \Omega_F$.
We call $\Omega_F$ the \textit{open (Markov) tile} associated to $F$ and $\overline{\Omega_F}$ the \textit{(closed Markov) tile} associated to $F$. The \textit{level} of $\Omega_F$ and $\overline{\Omega_F}$ is defined to be $n$.

Denote the vertices of the boundary $\partial F$ by $v_1,\dots, v_r$ and the corresponding circles by $C_1,\dots, C_r$.
For $i\in \{1,\dots,r\}$, denote by $D_i$ the closed disk bounded by $C_i$ and the corresponding reflection by $g_i$.
We call a sequence $(i_j)_{j=1}^\infty$ \textit{admissible} if $i_j \in \{1,\dots , r\}$ and $i_j \neq i_{j+1}$ for $j\in \N$. If $(i_j)_{j=1}^\infty$ is admissible, we define inductively $D_{i_1i_2\dots i_l} = g_{i_1}(D_{i_2\dots i_l})$ for $l\in \N$.
Note that $D_{i_1i_2 \dots i_l}$, $l\in \N$, is a sequence of nested closed disks. We set $D_{(i_j)}= \bigcap_{l=1}^\infty D_{i_1i_2\dots i_l}$, which is a single point, and observe that
$$
\partial \Omega_F = \bigcup_{(i_j)_{j=1}^\infty\text{ admissible}} D_{(i_j)}.
$$
Using this, one can verify the following statement for Markov tiles (see Figure \ref{fig:nestedF}).

\begin{lemma}[Properties of tiles]\label{lem:intersectionBoundary}Let $\mathcal{R}$ be a simple, irreducible, acylindrical finite subdivision rule, $\mathcal{G}$ be a (spherical) subdivision graph for $\mathcal R$, and $\mathcal P\in \mathcal M(\mathcal G)$. For $n,m\geq 0$, let $F^n, F^m$ be faces of $\mathcal{G}^n,\mathcal{G}^m$, respectively, and define $U^n= \Omega_{F^n}$ and $U^m= \Omega_{F^m}$. 
    \begin{enumerate}[label=\normalfont(\arabic*)]
        \item If $\partial F^n \cap \partial F^m$ is empty or a single vertex, then $\partial U^n \cap \partial U^m = \emptyset$.
        \item If $\partial F^n \cap \partial F^m$ is an edge, then $\partial U^n \cap \partial U^m$ is a single point.
        \item If $\partial F^n \cap  \partial F^m$ is the union of two non-adjacent vertices, then $\partial U^n \cap \partial U^m$ consists of two points.
        \item If $\partial F^n \cap  \partial F^m$ contains at least three vertices, then $\partial U^n \cap \partial U^m$ contains infinitely many points.
    \end{enumerate}
    In the above cases the common points of $\partial U^n$ and $\partial U^m$ correspond precisely to admissible sequences $(i_j)$ that contain only indices of the common vertices of $\partial F^n$ and $\partial F^m$. Moreover, the following statements are true. 
    \begin{enumerate}[label=\normalfont(\arabic*)]\setcounter{enumi}{4}
        \item\label{lem:boundary:5} If the faces $F^n$ and $F^m$ have disjoint interiors, then $U^n\cap U^m=\emptyset$. 
        \item\label{lem:boundary:6} If $F^m\subsetneq F^n$, then $U^m\subsetneq U^n$ and  $\partial U^n\cap \partial U^m$ contains at most one point, as a consequence of the above and \ref{condition:si}. 
        \item\label{lem:boundary:7} $\partial U^n \cap \Lambda(\mathcal{P}) = \bigcup_{i=1}^{r-1} (C_i \cap C_{i+1})$, where $\{C_i: i=1,\dots,r\}$ is the collection of circles of $\mathcal P$ corresponding to the vertices of $\partial F^n$.
    \end{enumerate}
\end{lemma}

\begin{figure}[htp]
    \centering
    \begin{tikzpicture}
    \node[anchor=south west,inner sep=0] at (0,0) {\includegraphics[width=0.75\textwidth]{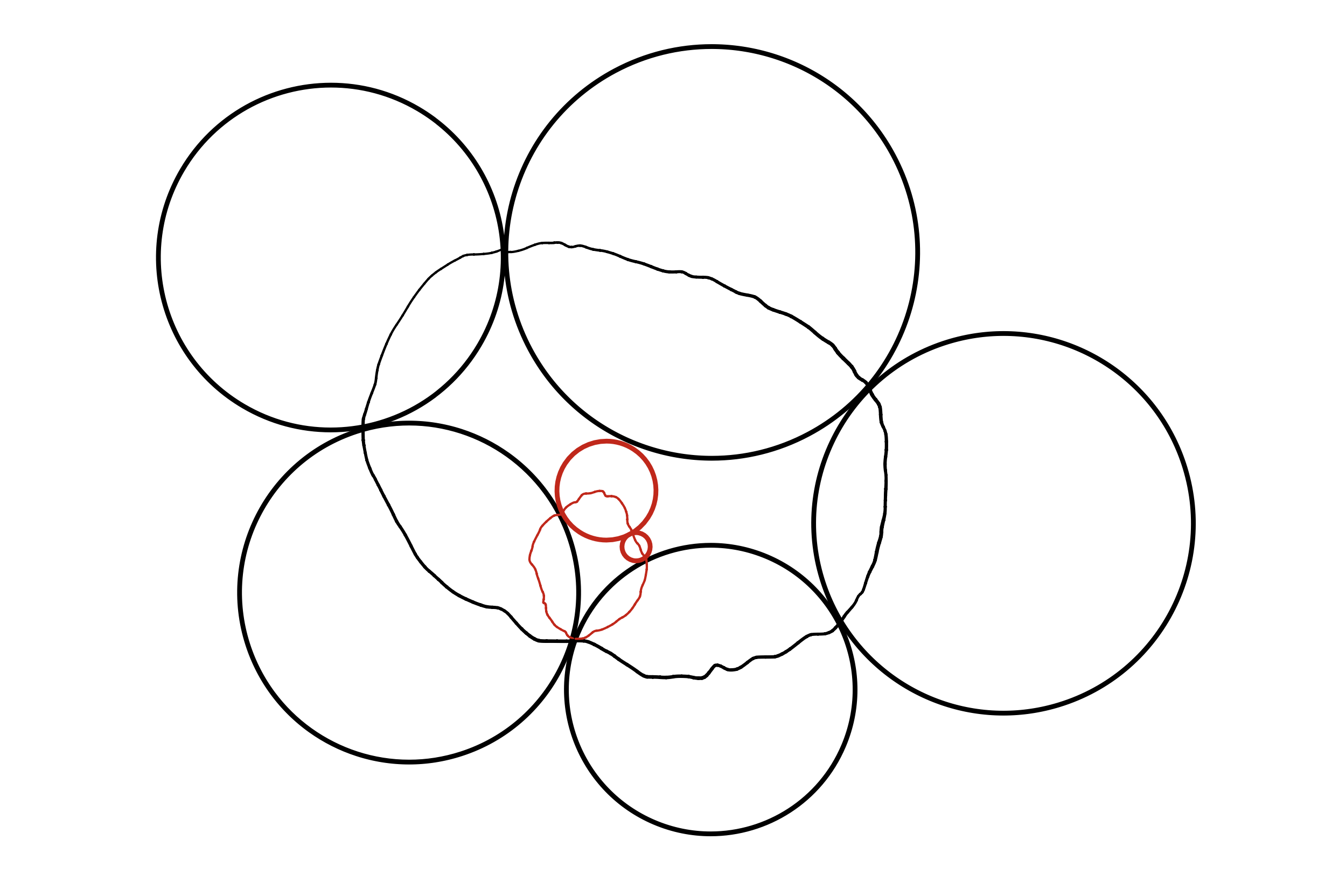}};
    \node at (4.25,2.4) {\begin{footnotesize}\textcolor{red}{$\Omega_{F^2}$}\end{footnotesize}};
    \node at (4.9,2.1) {\begin{footnotesize}\textcolor{red}{$\Lambda_{F^2}$}\end{footnotesize}};
    \node at (5.1,4.1) {\begin{footnotesize}$\Omega_{F^1}$\end{footnotesize}};
    \node at (5.5,4.7) {\begin{footnotesize}$\Lambda_{F^1} = \partial \Omega_{F^1}$\end{footnotesize}};
    \end{tikzpicture}
    \caption{An illustration of open tiles $\Omega_{F^2}, \Omega_{F^1}$ and limit sets $\Lambda_{F^2}, \Lambda_{F^1}$ corresponding to nested faces $F^2 \subset F^1$ such that $\partial F^2 \cap \partial F^1$ is an edge.}
    \label{fig:nestedF}
\end{figure}
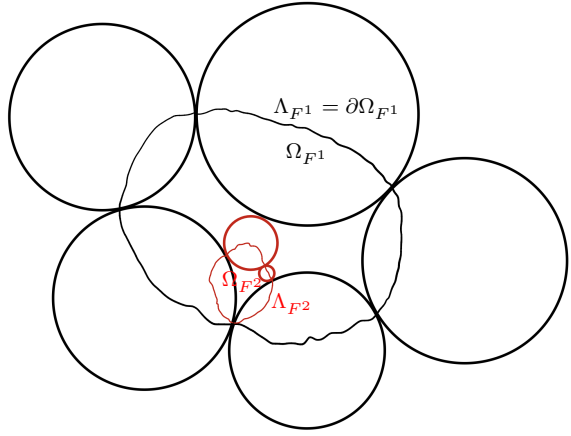

\begin{proposition}[Diameters of tiles]\label{prop:diam->0}
    Let $\mathcal{R}$ be a simple, irreducible, acylindrical finite subdivision rule, $\mathcal{G}$ be a (spherical) subdivision graph for $\mathcal R$, and $\mathcal P\in \mathcal M(\mathcal G)$. Then $\max\{\diam\Omega_{F}: F \text{ face of } \mathcal{G}^n\}\to 0$ as $n\to\infty$.
\end{proposition}
Here diameters are measured in the spherical metric. Compare to Lemma \ref{lemma:gasket_sub}.
\begin{proof}
    For the sake of contradiction, suppose that there exists $\delta>0$ and a sequence faces $F^n$ of $\mathcal{G}^n$, $n\geq 0$, with $\diam \Omega_{F^n} \geq \delta$ for all $n\geq 0$. We may assume that the faces $F^n$, $n\geq 0$ are nested. In view of \ref{condition:si}, there are three cases: (a) no vertex of $\mathcal G$ lies in $\partial F^n$ for infinitely many $n\geq 0$; (b) there exists a vertex $v\in \mathcal G$ that lies in $\partial F^n$ for all but finitely many $n\geq 0$; (c) there exist two adjacent vertices $v,w\in \mathcal G$ that lie in $\partial F^n$ for all but finitely many $n\geq 0$.  By the definition of a subdivision rule (see Definition \ref{defn:fsr}), $\partial F^n$ consists of a uniformly bounded number of vertices for all $n\geq 0$. Thus, the diameter of the union of disks corresponding to vertices of $\partial F^n$, except for the ones corresponding to $v$ and $w$ in the cases (b) and (c), tends to $0$ as $n\to\infty$. This implies that $\diam \Omega_{F^n}\to 0$ as $n\to\infty$, which is a contradiction.
\end{proof}

For $n\geq 0$ we define 
$$\Omega^n= \bigcup_{F\text{ face }\mathcal G^n}\Omega_F.$$
For $n\geq 0$, each face of $\mathcal G^{n+1}$ is contained in a face of $\mathcal G^n$ so each open tile of level $n+1$ is contained in an open tile of level $n$. Thus, $\Omega^{n+1} \subset \Omega^n$ for $n\geq 0$. Also observe that by Lemma \ref{lem:intersectionBoundary} \ref{lem:boundary:5} any two open tiles of level $n$ are disjoint. Hence, the components of $\Omega^n$ are precisely the open tiles of level $n$.

Let $\mathcal{G} = \mathcal{G}(P_i)$ be a subdivision graph for $\mathcal R$ and $\mathcal{P}\in \mathcal{M}(\mathcal{G})$. Consider the disks of $\mathcal P$ associated to $\partial P_i$. The union of those closed disks has two complementary components. We denote by $\Pi(\mathcal P)$ the component that has non-empty intersection with the limit set $\Lambda (\mathcal P)$. Compare the following lemma with Lemma \ref{lemma:gasket_sub}.

\begin{lemma}[Tiles and limit set]\label{cor:density}
    Let $\mathcal{R}$ be a simple, irreducible, acylindrical finite subdivision rule. Let $\mathcal{G}$ be a (resp.\ spherical) subdivision graph for $\mathcal R$ and $\mathcal{P} \in \mathcal{M}(\mathcal{G})$.
    Then the following statements are true. 
    \begin{enumerate}[label=\normalfont(\arabic*)]
        \item\label{d:1} ${\Lambda(\mathcal{P})\cap \overline{\Pi(\mathcal{P})}} = \bigcap_{n=0}^\infty \overline{\Omega^n}$ (resp.\ $\Lambda(\mathcal{P}) = \bigcap_{n=0}^\infty \overline{\Omega^n}$).
        \item\label{d:2} For each point $x\in {\Lambda(\mathcal{P})\cap \overline{\Pi(\mathcal{P})}}$ (resp.\ $x\in \Lambda(\mathcal P)$) there exists a nested sequence of faces $F^n$ of $\mathcal G^n$, $n\geq 0$, such that $\{x\}=\bigcap_{n=0}^\infty \overline {\Omega_{F^n}}$.  
        \item\label{d:3} For each circle $C$ in $\mathcal P$ the tangency points of $\mathcal P$ that lie on $C$ are dense in $C\cap \overline{\Pi(\mathcal{P})}$ (resp.\ $C$).
        \item\label{d:4}The tangency points of $\mathcal P$ are dense in ${\Lambda(\mathcal{P})\cap \overline{\Pi(\mathcal{P})}}$ (resp.\ $\Lambda(\mathcal P)$).
    \end{enumerate}
\end{lemma}

\begin{proof}Suppose that $\mathcal G$ is a subdivision graph for $\mathcal R$;  the proofs for spherical subdivision graphs are similar. We claim that for each $n\geq 0$, $\overline{\Omega_n}$ contains all intersections of circles of $\mathcal P$ with $\overline{\Pi(\mathcal{P})}$. Indeed, let $C_v$ be a circle of $\mathcal P$ corresponding to a vertex $v$ of $\mathcal G$. If $v$ is not a vertex of $\mathcal{G}^n$, then $v$ is contained in a face $F$ of $\mathcal{G}^n$. Lemma \ref{lem:intersectionBoundary} \ref{lem:boundary:7} implies that the circle $C_v$ is disjoint from $\partial \Omega_F$ so it is contained in $\Omega_F$. If $v$ is a vertex of $\mathcal G^n$, we denote by $F_1,..., F_j$ be the faces of $\mathcal{G}^n$ that have $v$ on their boundary. Again, Lemma \ref{lem:intersectionBoundary} \ref{lem:boundary:7} implies that each arc of $C_v\cap \overline{\Pi(\mathcal{P})}\setminus \bigcup_{i=1}^j \partial \Omega_{F_i}$ is contained in some $\Omega_{F_i}$, so $C_v \cap \overline{\Pi(\mathcal{P})}\subset \bigcup_{i=1}^j \overline{\Omega_{F_i}}$.

Since $\Lambda(\mathcal P)$ is the closure of the union of circles in $\mathcal{P}$, by the above claim we conclude that ${\Lambda(\mathcal{P})\cap \overline{\Pi(\mathcal{P})}}\subset \bigcap_{n=0}^\infty\overline{\Omega_n}$. Conversely, let $x\in \bigcap_{n=0}^\infty \overline{\Omega_n}$. Then there exists a nested sequence of faces $F^n$ of $\mathcal G^n$, $n\geq 0$, so that $x \in \bigcap_{n=0}^\infty \overline{\Omega_{F^n}}$. By Proposition \ref{prop:diam->0}, $\diam \Omega_{F^n} \to 0$. Thus, there exists a sequence of distinct circles of $\mathcal P$ converging to $x$. Since all but finitely many circles of $\mathcal P$ lie in ${\Lambda(\mathcal{P})\cap \overline{\Pi(\mathcal{P})}}$ and the latter is closed, we conclude that $x\in {\Lambda(\mathcal{P})\cap \overline{\Pi(\mathcal{P})}}$. This proves part \ref{d:1}. Note that \ref{d:2} follows from \ref{d:1} by repeating the above argument.

Let $C$ be a circle in $\mathcal P$ and $x\in C\cap \overline{\Pi(\mathcal P)}$.  By \ref{d:2} there exists a nested sequence of  faces $F^n$ of $\mathcal G^n$, $n\geq 0$, so that $\{x\}= \bigcap_{n=0}^\infty\overline{\Omega_{F^n}}$. For all sufficiently large $n\geq 0$, the vertex of $\mathcal G$ corresponding to $C$ lies in $\partial F^n$, so the tile $\overline{\Omega_{F^n}}$ contains a tangency point on $C$; see Lemma \ref{lem:intersectionBoundary} \ref{lem:boundary:7}. Thus, we conclude that the tangency points of $\mathcal P$ that lie on $C$ are dense in $C\cap \overline{\Pi(\mathcal{P})}$, as in \ref{d:3}. Part \ref{d:4} follows from \ref{d:3} and the fact that the circles of $\mathcal P$ are dense in $\Lambda(\mathcal P)$ by definition.
\end{proof}

The next statement is an immediate consequence of the density of tangency points in the limit set.

\begin{corollary}\label{cor:sr}
     Let $\mathcal{R}$ be a simple, irreducible, acylindrical finite subdivision rule. Let $\mathcal{G}$ be a (resp.\ spherical) subdivision graph for $\mathcal R$ and $\mathcal P,\mathcal P'\in \mathcal M(\mathcal G)$. Then any two homeomorphisms between $\mathcal{P}$ and $\mathcal{P}'$ coincide on ${\Lambda(\mathcal{P})\cap \overline{\Pi(\mathcal{P})}}$ (resp.\ $\Lambda(\mathcal P)$).
\end{corollary}

\subsection{Exponential contraction}
Let $\mathcal{R} = \{P_i\}_{i=1}^k$ be a simple, irreducible, acylindrical finite subdivision rule. 
Let $i\in \{1,\dots,k\}$ and $\mathcal{G}= \mathcal{G}(P_i)$ be a subdivision graph for $\mathcal{R}$ and let $\mathcal{P} \in \mathcal{M}(\mathcal{G})$. Let $F$ be a non-external face of $\mathcal{G}^n= \mathcal{G}^n(P_i)$ for some $n\geq 0$.
Denote by $\mathcal{P}_F$ the infinite sub-circle packing of $\mathcal{P}$ associated to the subgraph $F \cap \mathcal{G}$ in $\mathcal{G}$. Using renormalization methods, it is proved that the Teichm\"uller distance between sub-circle packings at deeper levels decays exponentially.
\begin{theorem}[Exponential contraction; \cite{LuoZhang:CirclePacking}*{Theorem C}]\label{thm:expr}
    Let $\mathcal{P}, \mathcal{P}' \in \mathcal{M}(\mathcal{G})$.
    Let $F$ be a non-external face of $\mathcal{G}^n$, $n\geq 0$.
    Then 
    $$
    d(\mathcal{P}_F, \mathcal{P}'_F) = O(\delta^n)\,\,\,\textrm{as $n\to\infty$}
    $$
    for some universal constant $\delta \in (0,1)$.   
\end{theorem}

We discuss some implications of the exponential contraction.

\subsubsection{Asymptotic conformality}
Let $X \subset \C$ be a set, not necessarily open, and $\alpha>0$.
A map $\Psi\colon X \to \C$ is \textit{$C^{1+\alpha}$-conformal} at a point $z\in  X$ if the complex derivative 
$$
\Psi'(z)\coloneqq \lim_{z+t \in X, t\to 0} \frac{\Psi(z+t) - \Psi(z)}{t}
$$ 
exists, and
$$
\Psi(z+t) = \Psi(z) + \Psi'(z)t + O(|t|^{1+\alpha})
$$
for all $z+t \in X$ and $t$ sufficiently small (cf.\ \cite{McM96}*{Section B.6}).
As a corollary of exponential contraction, we have the following theorem.
\begin{theorem}[Asymptotic conformality; \cite{LuoZhang:CirclePacking}*{Theorem D}]\label{thm:asympConf}
    Let $\Psi\colon\widehat \C \to \widehat\C$ be a homeomorphism with $\Psi(\infty)=\infty$ between two circle packings $\mathcal{P}, \mathcal{P}' \in \mathcal{M}(\mathcal{G})$. Then $\Psi|_{\Lambda(\mathcal{P})\cap \overline{\Pi(\mathcal{P})}}$ is $C^{1+\alpha}$-conformal at every point on $\Lambda(\mathcal{P})\cap \overline{\Pi(\mathcal{P})}\cap \C$ and $\Psi'$ is non-zero on $\Lambda(\mathcal P)\cap \overline{\Pi(\mathcal P)}\cap \C$.
\end{theorem}
The assumption that $\Psi$ fixes $\infty$ is only to ensure that $\Psi'$ can be defined on points $z\neq \infty$. The last statement that $\Psi'$ is non-zero is not contained in  \cite{LuoZhang:CirclePacking}*{Theorem D}, but it follows from applying  \cite{LuoZhang:CirclePacking}*{Theorem D} to $\Psi^{-1}$, which gives the existence of $(\Psi^{-1})'$.

\subsubsection{Renormalization periodic circle packings}
For $n\geq 0$, let $F^n$ be a non-external face of $\mathcal{G}^n=\mathcal{G}^n(P_i)$ and assume that $F^{n+1}\subset F^n$, $n\geq 0$.
Recall that each face $F^n$ is identified with a polygon $P_{\sigma(n)} \in \{P_i\}_{i=1}^k$, and $F^0 = P_i = P_{\sigma(0)}$.
Denote this identification by $\psi_n\colon F^n \to P_{\sigma(n)}$. (Since all faces are assumed to be polygons, the cellular identification map in Definition \ref{defn:fsr} may be taken to be a homeomorphism in the entire face, rather than in the interior only. The map $\psi_n$ here is obtained from the inverse of that identification map.)
We say that the sequence $\{F^n\}_{n\geq 0}$ is periodic of period $p$ if $\sigma(n) = \sigma(n+p)$ and $\psi_{n}(F^{n+p}) = \psi_{n+p}(F^{n+2p}) \subset P_{\sigma(n)} = P_{\sigma(n+p)}$ for all $n\geq 0$.

Let $\mathcal{P}\in \mathcal{M}(\mathcal{G})$.
Let $\{F^n\}_{n\geq 0}$ be periodic with period $p$.
Then the identification $\psi_p \colon F^p \to P_{\sigma(p)}$ gives an identification of $F^p \cap \mathcal{G}$ with $\mathcal{G}=\mathcal G(P_{\sigma(0)})=\mathcal G(P_{\sigma(p)})$. Therefore, by Theorem \ref{thm:qch}, this identification induces a quasiconformal homeomorphism $\Psi$ between $\mathcal{P}_{F^p}$ and $\mathcal{P}$. We call this map $\Psi$ a \textit{quasiconformal symmetry} of $\mathcal{P}$ for $\{F^n\}_{n\geq 0}$. By Corollary \ref{cor:sr}, the restriction $\Psi|_{\Lambda(\mathcal{P}_{F^p})\cap \overline{\Pi(\mathcal{P}_{F^p})}}$ is uniquely defined.

\begin{definition}\label{defn:pf}
    Suppose that the sequence $\{F^n\}_{n\geq 0}$ is periodic with period $p$.
We say that $\{F^n\}_{n\geq 0}$ is \textit{hyperbolic} if $\partial F^{p} \cap \partial F^0$ is either empty or a single vertex. Otherwise, we say that $\{F^n\}_{n\geq 0}$ is \textit{parabolic}.
\end{definition}

Recall that for a non-external face $F$ of $\mathcal{G}^n$, $n\geq 0$, we denote by $\Omega_F$ be the corresponding open Markov tile (see Section \ref{subset:mt}) and by $\Lambda_F$ its boundary.

\begin{theorem}[Periodic circle packing; \cite{LuoZhang:CirclePacking}*{Theorem D}]\label{thm:rp}
    Suppose that the sequence $\{F^n\}_{n\geq 0}$ is periodic with period $p$. Then there exists a unique up to M\"obius maps circle packing  $\mathcal{P}\in \mathcal M(\mathcal{G})$ so that for each $n\geq 0$ we have 
    $$
    d(\mathcal{P}_{F^n}, \mathcal{P}_{F^{n+p}}) = 0.
    $$
    In particular, there is a M\"obius map $\Psi$ with $\Psi(\mathcal{P}_{F^p}) = \mathcal{P}$.
    The map $\Psi$ is a parabolic M\"obius map if $\{F^n\}_{n\geq 0}$ is parabolic, and a loxodromic M\"obius map if $\{F^n\}_{n\geq 0}$ is hyperbolic.
    Moreover, the sequence of sets $\{\Omega_{F^{kp}}\}_{k\geq 0}$ converges to the non-attracting fixed point of $\Psi$. 
\end{theorem}

Let $\{F^n\}_{n\geq 0}$ be periodic with period $p$.
We will call the unique circle packing $\mathcal{P}$ in Theorem \ref{thm:rp} the \textit{renormalization periodic circle packing} associated to $\{F^n\}_{n\geq 0}$, and the M\"obius map $\Psi$ in Theorem \ref{thm:rp} the \textit{M\"obius symmetry} of $\mathcal{P}$ for $\{F^n\}_{n\geq 0}$.

\subsubsection{Teichm\"uller mappings}
Let $\mathcal{P}, \mathcal{P}' \in \mathcal{M}(\mathcal{G})$.
By Theorem \ref{thm:qch}, there exists a quasiconformal map between $\mathcal P$ and $\mathcal P'$. Here we define the notion of a Teichm\"uller mapping between two circle packings, by imposing some conditions in addition to quasiconformality.

\begin{definition}Let $\mathcal{P}, \mathcal{P}' \in \mathcal{M}(\mathcal{G})$.  A quasiconformal homeomorphism $\Psi: \widehat \C \to \widehat \C$ between $\mathcal{P}$ and $\mathcal{P}'$ is said to be a \textit{Teichm\"uller mapping} if for every $n\geq 0$ and for every non-external face $F$ of $\mathcal{G}^n$ we have
\begin{itemize}
    \item $\frac{1}{2} \log K(\Psi|_{\Omega_{F}}) = d(\mathcal{P}_{F}, \mathcal{P}'_{F})$ and
    \item $\Psi|_{\Lambda_{F}}$ conjugates the dynamics of $\Gamma_{F}$ and $\Gamma'_{F}$.
\end{itemize}
\end{definition}

In particular, $\Psi$ minimizes the dilatation on $\Omega_{F}$ for every face $F$ of $\mathcal{G}^n$. The conjugation condition allows us to have the following extension property.

\begin{proposition}\label{prop:ext}
    Let $\Psi: \widehat \C \to \widehat \C$ be a Teichm\"uller mapping between $\mathcal{P}, \mathcal{P}' \in \mathcal{M}(\mathcal{G})$.
    Then for every $n\geq 0$ and for every non-external face $F$ of $\mathcal{G}^n$, $\Psi|_{\Omega_{F}}$ extends to a quasiconformal homeomorphism $\Psi_F\colon \widehat\C \to\widehat\C$ between $\mathcal{P}_F$ and $\mathcal{P}'_{F}$ with dilatation
    $K(\Psi_{F})$ satisfying $\frac{1}{2} \log K(\Psi_{F}) = d(\mathcal{P}_{F}, \mathcal{P}_{F}')$.
\end{proposition}
\begin{proof}
    Since the restriction of $\Psi$ on the limit set $\Lambda_{F}$ is a conjugation, one can extend the map $\Psi|_{\Lambda_{F}}$ to a quasiconformal homeomorphism between $\widehat\C \setminus \overline{\Omega_{F}}$ and $\widehat\C \setminus \overline{\Omega'_{F}}$ sending the ideal polygon in $\widehat\C \setminus \overline{\Omega_{F}}$ associated to the external face to the corresponding ideal polygon in $\widehat\C \setminus \overline{\Omega'_{F}}$, and whose dilatation $K$ satisfies that $\frac{1}{2}\log K$ is equal to the Teichm\"uller distance between the two ideal polygons.
    By Theorem \ref{thm:qch}, the Teichm\"uller distance between the two ideal polygons equals $d(\mathcal{P}_{F}, \mathcal{P}_{F}')$.
    The proposition now follows.
\end{proof}

\begin{theorem}[\cite{LuoZhang:CirclePacking}*{Theorem 5.15}]\label{thm:existenceTeich}
    There exists a Teichm\"uller mapping between any two circle packings $\mathcal{P},\mathcal{P}'\in \mathcal M(\mathcal G)$.
\end{theorem}

\subsection{Markov maps}\label{subsec:MM}
Let $\mathcal{R}= \{P_i\}_{i=1}^k$ be a simple, irreducible, acylindrical finite subdivision rule.
Let $\mathcal{G} = \mathcal{G}(X)$ be a spherical subdivision graph for some $\mathcal{R}$-complex $X$ homeomorphic to a sphere.
Let $\mathcal{P}$ be the corresponding circle packing for $\mathcal{G}$. In this section, we discuss how the subdivision rule induces a Markov map for the circle packing $\mathcal{P}$.

Let $n\geq 0$ and $F$ be a face of $\mathcal{G}^n$.
Then $F$ is identified with some polygon $P_{\sigma(F)}\in  \{P_i\}_{i=1}^k$. Denote the identification map by 
$$
\phi_F\colon F \to P_{\sigma(F)}.
$$
This map induces a canonical identification of the subgraph $F \cap \mathcal{G}$ with the subdivision graph $\mathcal{G}(P_{\sigma(F)})$ for $\mathcal{R}$.
Recall that $\mathcal{P}_F$ is the infinite sub-circle packing of $\mathcal{P}$ associated to the subgraph $F \cap \mathcal{G}$ in $\mathcal{G}$.
Thus, $\mathcal{P}_F$ has tangency graph isomorphic to $\mathcal{G}(P_{\sigma(F)})$.
Recall that $\Omega_F$ is the Markov tile associated to the face $F$, and
$$\Omega^1=\bigcup_{F \text{ face } \mathcal{G}^1} \Omega_F \,\,\, \textrm{and}\,\,\, \Omega^0=\bigcup_{F \text{ face } \mathcal{G}^0} \Omega_{F}.$$

\begin{definition}
    A map $\psi\colon \mathcal{G}^1 \to \mathcal{G}^0$ is called a \textit{subdivision homomorphism} if it is a homomorphism between the plane graphs $\mathcal{G}^1$ and $\mathcal{G}^0$ and for each face $F$ of $\mathcal{G}^1$, there exists a face $F'$ of $\mathcal G^0$ that we denote by $\psi_*(F)$ such that $\sigma(F) = \sigma(F')$ and
    $$
    \psi|_{\partial F} = \phi_{F'}^{-1} \circ \phi_F |_{\partial F}. 
    $$
\end{definition}

In the next proposition, for a vertex $v$ of $\mathcal G^1$ we denote by $C_v$ the circle of $\mathcal P$ associated to $v$. By Lemma \ref{cor:density} \ref{d:1} we have $C_v\subset \overline {\Omega^1}$.

\begin{proposition}[Induced Markov map]\label{prop:induceMarkovMap}
Let $\psi\colon \mathcal{G}^1 \to \mathcal{G}^0$ be subdivision homomorphism.
Then $\psi$ induces a map $\Psi\colon \Omega^1 \to \Omega^0$ with the following properties.
\begin{enumerate}[label=\normalfont(\arabic*)]
    \item\label{ma:1} For each face $F$ of $\mathcal G^1$ we have $\Psi({\Omega_F})=\Omega_{\psi_*(F)}$ and $\Psi|_{\Omega_F}$ is the restriction of a Teichm\"uller mapping between $\mathcal P_F$ and $\mathcal P_{\psi_*(F)}$.
    \item\label{ma:11} For each $n\geq 0$ and each face $F$ of $\mathcal G^{n+1}$ there exists a face $F'$ of $\mathcal G^n$ so that $F$ and $F'$ are canonically identified with the same face of $\mathcal G^n(P_i)$ for some $i\in \{1,\dots,k\}$ and $\Psi(\Omega_F)=\Omega_{F'}$.
    \item\label{ma:2} $\Psi$ extends continuously to $\overline{\Omega^1}$.
    \item\label{ma:3} For each vertex $v$ of $\mathcal{G}^1$ the map $\Psi|_{C_v}$ is a covering map from $C_v$ onto $C_{\psi(v)}$.
    \item\label{ma:4} For each edge $e = [v,w]$ of $\mathcal{G}^1$ we have $\Psi(C_v \cap C_w) = C_{\psi(v)} \cap C_{\psi(w)}$.
\end{enumerate}
\end{proposition}

\begin{proof}
    Let $F$ be a face of $\mathcal{G}^1$ and $F' = \psi_*(F)$ be the corresponding face of $\mathcal{G}^0$. Then $\phi_{F'}^{-1} \circ \phi_F$ induces an identification of the tangency graph of $\mathcal P_F$ with the tangency graph of $\mathcal P_{F'}$. Let $\Omega_F$ and $\Omega_{F'}$ be the corresponding Markov tiles. Then, by Theorem \ref{thm:existenceTeich}, there exists a map $$\Psi_{F\to F'}\colon \Omega_F \to \Omega_{F'}$$ that is the restriction of a Teichm\"uller mapping between $\mathcal{P}_F$ and $\mathcal{P}_{F'}$.  Repeating this process for each face $F$ of $\mathcal G^1$ gives the desired map $\Psi\colon \Omega^1 \to \Omega^0$ satisfying \ref{ma:1}. 

    Let $F$ be a face of $\mathcal G^{n+1}$ and let $F^1$ be the face of $\mathcal G^1$ containing $F$. The subdivision homomorphism $\psi$ sends the face $F^1$ to a face $\psi_*(F^1)=F^0$ of $\mathcal G^0$. Note that by the definition of a subdivision homomorphism the faces $F^1, F^0$ are identified with the same polygon in the subdivision rule, say $P_i$.  Also, the face $F$ is identified with a face of $\mathcal G^n(P_i)$. This face corresponds canonically to a face $F'$ in $F^0$. By part \ref{ma:1}, $\Psi|_{\Omega_{F^1}}$ is the restriction of a Teichm\"uller mapping, and in particular of a homeomorphism, between the circle packings $\mathcal{P}_{F^1}$ and $\mathcal{P}_{F^0}$. Given that the labeling is preserved and that $\Psi$ conjugates $\Gamma_{F}$ to $\Gamma_{F'}$, we conclude that $\Psi(\Omega_F)=\Omega_{F'}$. This proves \ref{ma:11}.

    To prove \ref{ma:2}, suppose that two faces $F_1, F_2$ of $\mathcal{G}^1$ share part of their boundary. Then $\psi$ sends $\partial F_1 \cap \partial F_2$ to $\partial \psi_*(F_1) \cap \partial \psi_*(F_2)$. Let $\Psi_i=\Psi|_{\Omega_{F_i}}$, $i=1,2$. By the properties of a Teichm\"uller mapping, for $i=1,2$ the map $\Psi_i$ extends to $\Lambda_{F_i}$ such that $\Psi_i|_{\Lambda_{F_i}}$ conjugates the dynamics of $\Gamma_{F_i}$ and $\Gamma_{\psi_*(F_i)}$. In combination with Lemma \ref{lem:intersectionBoundary}, this implies that the extensions of $\Psi|_{\Omega_{F_1}}$ and $\Psi|_{\Omega_{F_2}}$ agree on $\partial \Omega_{F_1} \cap \partial \Omega_{F_2}$. Hence, $\Psi$ has a continuous extension to $\overline{\Omega^1}$.

    We now prove \ref{ma:3}. Let $F$ be a face of $\mathcal G^1$ such that $v\in \partial F$. Since $\Psi|_{\Omega_F}$ extends to an orientation-preserving homeomorphism between $\mathcal P_F$ and $\mathcal P_{\psi_*(F)}$, we see that it maps circles of $\mathcal P_F$ to the corresponding circles of $\mathcal P_{\psi_*(F)}$. Thus, $\Psi|_{\Omega_F}$ maps the arc $\Omega_F\cap C_v$ injectively onto the arc $\Omega_{\psi_*(F)}\cap C_{\psi(v)}$. Moreover, the extension of $\Psi|_{\Omega_F}$ maps the disk of $\mathcal P_F$ bounded by $C_v$ to the disk of $\mathcal P_{\psi_*(F)}$ bounded by $C_{\psi(v)}$. This implies that if $F_1,F_2$ are faces of $\mathcal G^1$ such that $\partial F_1\cap \partial F_2$ contains an edge that has $v$ as an endpoint, then $\Psi|_{C_v}$ is locally injective at the common endpoint of the arcs $\Omega_{F_1}\cap C_v$ and $\Omega_{F_2}\cap C_v$. Thus, $\Psi|_{C_v}$ is a covering map. Finally, observe that \ref{ma:4} follows from \ref{ma:2} and \ref{ma:3}.
\end{proof}

We call the map $\Psi\colon \Omega^1\to \Omega^0$ provided by Proposition \ref{prop:induceMarkovMap} a \textit{Markov map} for $\mathcal P$ (induced by the subdivision homomorphism $\psi\colon \mathcal{G}^1 \to \mathcal{G}^0$). By construction, we also have the following extension property.

\begin{proposition}[Uniform quasiconformality]\label{prop:qce}
    Let $\Psi$ be a Markov map for a circle packing $\mathcal P$ with spherical subdivision rule.
    There exist constants $C>0$, $K\geq 1$, and $\delta\in (0,1)$ so that for each $n,m\geq 0$ the following statements are true.
    \begin{enumerate}[label=\normalfont(\arabic*)]
        \item For each open tile $U$ of level $n+1$ the map $\Psi|_U$  has a $(1+ C\delta^n)$-quasi\-conformal extension to $\widehat\C$.
        \item For each open tile $U$ of level $n+m$ the map $\Psi^{\circ m}|_U$ has a $(1+C\delta^n)$-quasiconformal extension to $\widehat\C$. 
    \end{enumerate}
\end{proposition}

\begin{proof}
Let $F_U$ be the face of $\mathcal G^{n+1}$ corresponding to the tile $U$ and let $F^1$ be the face of $\mathcal G^1$ containing $F_U$. The subdivision homomorphism that induces the Markov map $\Psi$ sends the face $F^1$ to a face $F^0$ of $\mathcal G^0$.  Note that by the definition of a subdivision homomorphism the faces $F^1, F^0$ are identified with the same polygon in the subdivision rule, say $P_i$, so $\mathcal{G} \cap F^1$ and $\mathcal{G} \cap F^0$ are identified with $\mathcal{G}(P_i)$.  Also, the face $F_U$ is identified with a face of $\mathcal G^n(P_i)$. This face corresponds to a face $F_V$ in $F^0$, where $V$ is an open tile of level $n$, and we have $V=\Psi(U)$ by Proposition \ref{prop:induceMarkovMap} \ref{ma:11}. By treating $\mathcal P_{F^1},\mathcal P_{F^0}$ as elements of $\mathcal M(\mathcal G(P_i))$ and $F_U,F_V$ as non-external faces of $\mathcal G^n(P_i)$, we may apply Theorem \ref{thm:expr} to deduce that $d(\mathcal{P}_{F_U}, \mathcal{P}_{F_V}) = O(\delta^n)$ for some constant $\delta\in (0,1)$. Moreover, by Proposition \ref{prop:induceMarkovMap} \ref{ma:1}, $\Psi|_{\Omega_{F^1}}$ is the restriction of a Teichm\"uller mapping between $\mathcal P_{F^1},\mathcal P_{F^0}\in \mathcal M(\mathcal G(P_i))$. By Proposition \ref{prop:ext}, $\Psi|_U$ has an extension to a quasiconformal homeomorphism of $\widehat\C$ between $\mathcal{P}_{F_U}$ and $\mathcal{P}_{F_V}$ with dilatation $K$ satisfying  $\frac{1}{2}\log K = d(\mathcal{P}_{F_U}, \mathcal{P}_{F_V})=O(\delta^n)$. This proves the first part. Since $\sum_{n=1}^\infty \delta^n < \infty$, the second part follows. 
\end{proof}

\section{Geometry of the Markov tiles}\label{sec:gmp}
We assume that all finite subdivision rules appearing in this section satisfy the standing assumptions \ref{condition:siL}, \ref{condition:siG}, and \ref{condition:si} from Section \ref{sec:cpsr}.

For a circle packing $\mathcal P$ with spherical subdivision rule and with tangency graph $\mathcal G=\lim_{\to}\mathcal G^n$ we denote by $C_v$ the circle corresponding to the vertex $v\in \mathcal G$. 

\begin{theorem}\label{thm:cmc}
    Let $\Psi$ be a Markov map for a circle packing $\mathcal P$ with spherical subdivision rule. Let $v\in \mathcal G^1$ such that $C_v\subset \C$ and all open tiles of level $0$ that intersect $C_v$ are contained in $\C$. Suppose that $\Psi(C_v)=C_v$, let $u_i$, $i\in \{0,\dots,r\}$, be the collection of vertices of $\mathcal G^1$ that are adjacent to $v$, numbered in counter-clockwise order, and define $a_i=C_v\cap  C_{u_i}$. The following statements are true.
    \begin{enumerate}[label=\normalfont(\arabic*)]
        \item\label{cmc:1} The map $\Psi|_{C_v}\colon C_v\to C_v$ is an expansive covering map.
        \item\label{cmc:2} The set $\{a_0,\dots,a_r\}$ defines a Markov partition for $\Psi|_{C_v}$ that satisfies conditions \ref{condition:hp}, \ref{condition:uv}, \ref{condition:qs} and \ref{condition:qs_strong}.
        \item\label{cmc:3} Every point in $\{a_0,\dots, a_r\}$ is symmetrically parabolic.  
    \end{enumerate}  
\end{theorem}

Recall the definition of a Markov map from Section \ref{subsec:MM} and of the notions appearing in \ref{cmc:2}, \ref{cmc:3} from Section \ref{sec:epqcm}. The main difficulty of the proof is in establishing that each point of $\{a_0,\dots,a_r\}$ is symmetrically parabolic. We prove this below in Proposition \ref{prop:para}. Assuming this, we give the proof of Theorem \ref{thm:cmc}.

\begin{figure}[htp]
    \centering
    \includegraphics[width=0.75\textwidth]{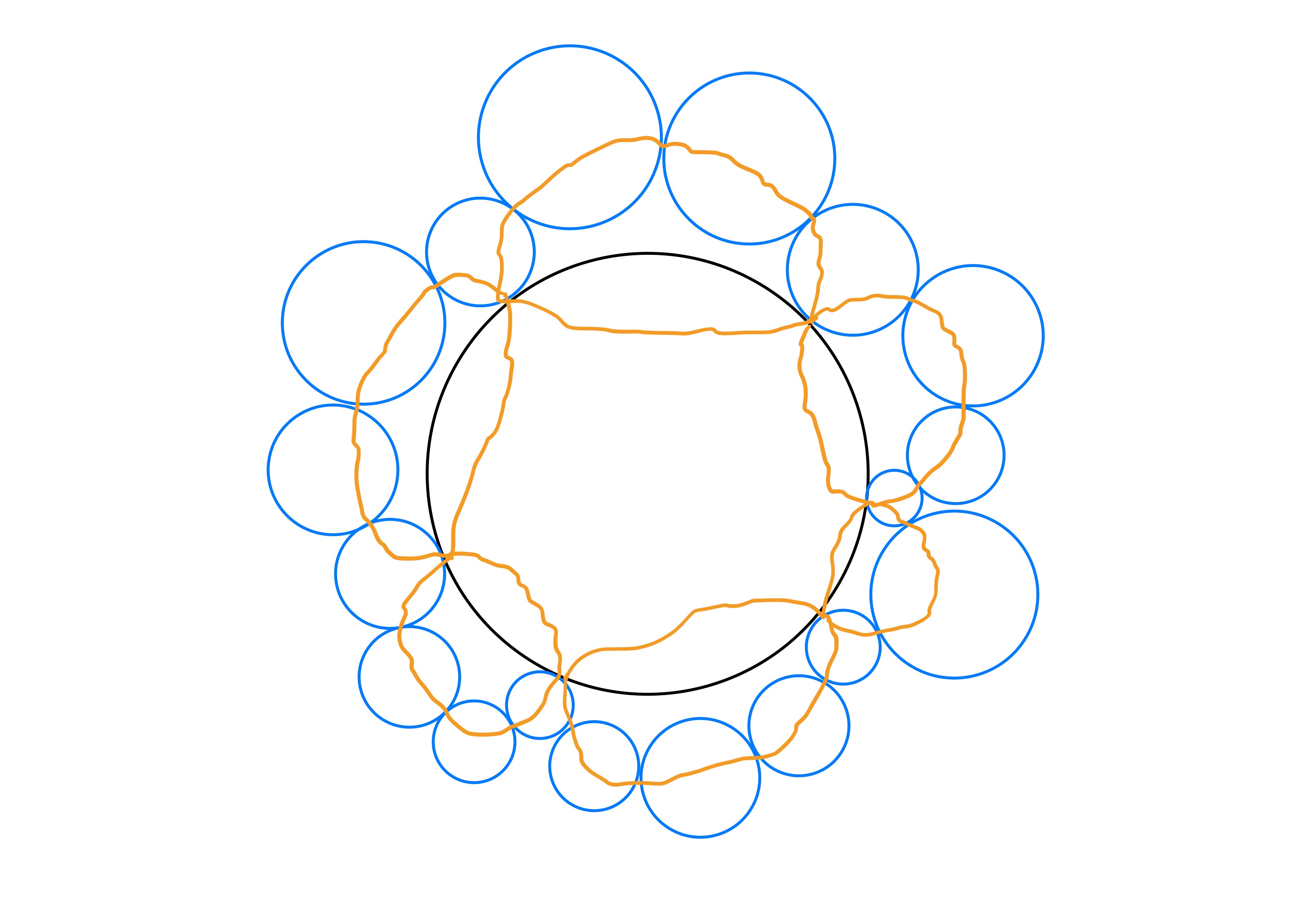}
    \caption{Tiles of level $0$ that intersect the circle $C$.}
    \label{fig:tileC}
\end{figure}

\begin{proof}
    Suppose that $\Psi$ is induced by a subdivision homomorphism $\psi\colon \mathcal G^1\to \mathcal G^0$. Let $C=C_v$. By Proposition \ref{prop:induceMarkovMap} \ref{ma:3}, $\psi(v)=v$ and $\Psi|_{C}$ is a covering map from $C$ onto itself. For $i\in \{0,\dots,r\}$ let $F_i$ be the face of $\mathcal G^1$ that contains the edges $[v,u_i]$ and $[v,u_{i+1}]$ in its $1$-skeleton, where $u_{r+1}\coloneqq u_0$; if $i=1$, there are only two faces. By condition \ref{condition:siG}, the faces $F_0,\dots,F_r$ are distinct.  For each face $F_i$ the face $\psi_*(F_i)$ of $\mathcal G^0$ contains the vertex $v$ in its $1$-skeleton and contains some of the faces $F_j$, $j\in \{0,\dots,r\}$, of $\mathcal G^1$. The $1$-skeleton of $\psi_*(F_i)$ contains the edges $[v,\psi(u_i)]$ and $[v,\psi(u_{i+1})]$, where $\psi(u_i),\psi(u_{i+1})\in \{u_0,\dots,u_r\}$ since $\mathcal G^0\subset \mathcal G^1$.

    We define $U_{i}=\Omega_{F_i}$ for $i\in \{0,\dots,r\}$. Each level $1$ tile $U_i$ is mapped by $\Psi$ homeomorphically to the level $0$ tile $\Omega_{\psi_*(F_i)}$ by Proposition \ref{prop:induceMarkovMap} \ref{ma:1}. We denote by $V_0,\dots,V_l$ the list of level $0$ tiles arising in this way; see Figure \ref{fig:tileC}. Note that the open tiles $V_0,\dots,V_l$ are pairwise disjoint and the open tiles $U_0,\dots,U_r$ are also pairwise disjoint (by Lemma \ref{lem:intersectionBoundary} \ref{lem:boundary:5}). Also, $\partial U_i\cap C=\{a_i,a_{i+1}\}$ (by Lemma \ref{lem:intersectionBoundary} \ref{lem:boundary:7}), where $a_{r+1}\coloneqq a_0$. By Proposition \ref{prop:induceMarkovMap} \ref{ma:4}, $\Psi$ maps $a_i$ to the point $C\cap C_{\psi(u_i)}$ which is contained in $\{a_0,\dots,a_r\}$. Therefore, $\Psi$ is injective in each arc $\arc{(a_i,a_{i+1})}$ and leaves the set $\{a_0,\dots,a_r\}$ invariant. This implies that the set $\{a_0,\dots,a_r\}$ gives a Markov partition for $\Psi|_{C}$; see Definition \ref{definition:markov_partition}. The preimages of each arc $\arc{(a_i,a_{i+1})}$ under $(\Psi|_C)^{\circ n}$ are arcs, each of which is contained in open tiles of level $n$, where $n\in \N$, as a consequence of Proposition \ref{prop:induceMarkovMap} \ref{ma:11}. By Proposition \ref{prop:diam->0} and property \ref{expansive:diameters}, $\Psi$ is expansive on $C$. We have verified \ref{cmc:1}.

    Now, we verify \ref{cmc:2}. As discussed, each open tile $U_i$ is mapped by $\Psi$ homeomorphically to a level $0$ tile. Moreover, if an open tile $U_i$ intersects an open tile $V_k$, then $U_i\subset V_k$; see Lemma \ref{lem:intersectionBoundary} \ref{lem:boundary:5} and  \ref{lem:boundary:6}. Since $V_k\subset \C$ for each $k\in \{0,\dots,r\}$, we have verified condition \ref{condition:uv}. Conditions \ref{condition:qs} and  \ref{condition:qs_strong} follow immediately from Proposition \ref{prop:qce}. 
\end{proof}

Following the notation in Section \ref{sec:epqcm}, we denote by $I_w$ the arc of $C_v$ corresponding to an admissible word $w$ whose length is denoted by $|w|$, and by $U_w$ the corresponding open tile of level $|w|$ for the Markov map $\Psi$. Note that $U_w\supset \inter I_w$.
\begin{proposition}\label{prop:para}
    Let $a \in \{a_0, \dots, a_r\}$. Then $a^+$ and $a^-$ are parabolic with multiplicity $2$.
\end{proposition}
\begin{proof}We will verify that $a^+$ is parabolic with multiplicity $2$; see Definition \ref{definition:parabolic}. The proof for $a^-$ is the same. First suppose that $a$ is periodic. For simplicity of the presentation, we assume that $a$ is fixed by $\Psi$. Then there exists a circle $C_w$ of the circle packing $\mathcal P$ such that $C_v\cap C_w=\{a\}$ and $C_w$ is fixed by $\Psi$. The edge $[v,w]$ is on the $1$-skeleton of a face $F^0$ of $\mathcal G^0$ such that the tile $\Omega_{F^0}$ contains an arc of $C_v$ with an endpoint at $a$ towards the positive orientation. We also consider a face $F^1\subset F^0$ of $\mathcal G^1$ that contains $[v,w]$ in its $1$-skeleton. If $\Psi$ arises from the subdivision homomorphism $\psi\colon \mathcal G^1\to \mathcal G^0$, then we have $\psi_*(F^1)=F^0$. Hence, $\Psi(\Omega_{F^1})=\Omega_{F^0}$; see Proposition \ref{prop:induceMarkovMap} \ref{ma:1}. Moreover, if, for each $n\in \N$, $F^n$ is a face of $\mathcal G^n$ that contains the edge $[v,w]$ in its $1$-skeleton, then we have $F^{n}\subset F^{n-1}$ and $\Psi(\Omega_{F^n})=\Omega_{F^{n-1}}$. We define $U^n=\Omega_{F^n}$, $n\geq 0$.

After normalization by a M\"obius map, we assume that $C_v=C = \R\cup \{\infty\}$, the disk $D_v$ bounded by $C_v$ is the upper half-plane, $\overline{U^0}\subset \C$, and $a = 0$. For each $n\geq 0$  we have $U^{n+1}\subset U^{n}$, $\Psi(U^{n+1})=U^{n}$, and $a\in \partial U^{n}$ with $(a, a+\epsilon) \subset U^{n}$ for some $\epsilon > 0$ depending on $n$.  For $n\geq 1$, let 
    $$
    I^{n} = [a, a_n]\coloneqq \overline{U^{n} \cap C}.
    $$
It suffices to show that there exists $L\geq 1$ such that for all $n\geq 1$ we have
        $$\frac{L^{-1}}{n}\leq \diam I^{n} \leq \frac{L}{n}\quad \textrm{and}\quad \frac{L^{-1}}{n^2}\leq \diam (I^{n}\setminus I^{n+1}) \leq \frac{L}{n^2}.$$
        
    Note that the faces $\{F^n\}_{n\geq 0}$ corresponding to $\{U^{n}\}_{n\geq 0}$ are periodic, and $a$ is a tangency point of two circles $C_v,C_w$ of $\mathcal{P}$. The edge $[v,w]$ lies in $\partial F^{n}$ for all $n\geq 0$. Thus, $\{F^n\}_{n\geq 0}$ is parabolic, as in Definition \ref{defn:pf}.
    By Theorem \ref{thm:rp} there exist 
        \begin{itemize}
            \item the renormalization periodic circle packing $\widetilde{\mathcal{P}} \in \mathcal{M}(F^0 \cap \mathcal{G})$ associated to $\{F^n\}_{n\geq 0}$, a point $\widetilde a\in \Lambda(\widetilde P)$, and for each $n\geq 0$ an open tile $\widetilde U^n \coloneqq \widetilde{\Omega}_{F^n}$ of level $n$ such that $\widetilde a\in \partial \widetilde U^n$ and $\widetilde{U}^{n+1}\subset \widetilde{U}^{n}$ for each $n\geq 0$, and
            \item the M\"obius symmetry of $\widetilde{\mathcal{P}}$ for $\{F^n\}_{n\geq 0}$, denoted by $\widetilde \Psi$, which maps $\widetilde{U}^1$ onto $\widetilde{U}^0$ and has a parabolic fixed point at $\widetilde{a}$.
        \end{itemize}
        Let $\Phi$ be a Teichm\"uller mapping, as provided by Theorem \ref{thm:existenceTeich}, between $\widetilde{\mathcal{P}},\mathcal{P}_{F^0}\in \mathcal{M}(F^0 \cap \mathcal{G})$. Since $\Phi^{-1} \circ \Psi \circ \Phi$ induces the identification of $F^1 \cap \mathcal{G}$ with $F^0\cap \mathcal{G}$, by Corollary \ref{cor:sr} we have $\Phi^{-1} \circ \Psi \circ \Phi = \widetilde{\Psi}$ on ${\Lambda(\widetilde{\mathcal{P}}_{F^1})\cap \overline{\Pi(\widetilde{\mathcal{P}}_{F^1})}}$.  
        
    We normalize by some M\"obius maps so that  $\widetilde{C} = \Phi^{-1}(C) = \R\cup \{\infty\}$, $\widetilde{a} = \Phi^{-1}(a) = 0$, $\widetilde{a}^+$ corresponds to the positive real direction,  $\widetilde{\Psi}(z) = \frac{z}{z+1}$, and $\widetilde{a}_1 = \Phi^{-1}(a_1)= 1$. In particular, $\widetilde{a}_n =\Phi^{-1}(a_n)= \frac{1}{n}$. Define $\widetilde{I}^{n}=\Phi^{-1}(I_n)=[\widetilde{a}, \widetilde{a}_n]$, $n\geq 0$. By our normalization, we have 
    $$
    \diam \widetilde{I}^{n} = \frac{1}{n} \quad \textrm{and} \quad  \diam(\widetilde{I}^{n}\setminus\widetilde{I}^{n+1}) = \frac{1}{n^2}.
    $$
    By Theorem \ref{thm:asympConf}, $\Phi|_{\widetilde{I}^0}$ is $C^{1+\alpha}$-conformal at $0$.
    Therefore, we have 
    \begin{align}
    \label{eq:I_n}    a_n=\diam I^{n} \simeq \frac{1}{n}+ O\left(\frac{1}{n^{1+\alpha}}\right ) \simeq \frac{1}{n}.
    \end{align}
    This verifies the first condition in Definition \ref{definition:parabolic}. However, $C^{1+\alpha}$-conformality does not immediately imply that $\diam(I^{n}\setminus I^{n+1}) \simeq \frac{1}{n^2}$. We will show the latter using a similar argument as in \cite{LuoZhang:CirclePacking}*{Theorem 6.1}.

    By Proposition \ref{prop:ext} and Theorem \ref{thm:expr}, $\Phi|_{\widetilde U^n}$ extends to a $(1+O(\delta^n))$-quasi\-conformal homeomorphism $f_n\colon \widehat \C \to \widehat \C$ between $\widetilde{\mathcal{P}}_{F^n}$ and $\mathcal{P}_{F^n}$. Let $p_n = f_n^{-1}(\infty)$ be the pole of $f_n$. By Corollary \ref{corollary:pn}, applied to the map $f_{n-1}^{-1}\circ f_n$,
 we have 
    \begin{align}
    \label{eq:p_n}   \left|\frac{1}{p_n}\right| = O(1).
    \end{align}
Let $M_n(z)=\frac{A_nz}{1-z/p_n}$ be a M\"obius map so that
    $$
    M_n(0) = f_n(0) = 0,\,\, M_n\left(\frac{1}{n}\right) = f_n\left(\frac{1}{n}\right) = a_n,\,\,\textrm{and}\,\,\,  M_n(p_n) = f_n(p_n) = \infty. 
    $$
    By \eqref{eq:p_n}, we have $|M_n(z)|\simeq |A_nz|$ in a fixed ball $B(0,R)\subset B(0,|p_n|/2)$.  By \eqref{eq:I_n}, we conclude that $|A_n| \simeq 1$.  We have $M_n'(z) = \frac{A_n}{(1-z/p_n)^2}$, so $|M_n'(z)|\simeq 1$ on $B(0,R)$. We conclude that
    \begin{align}\label{eqn:6.1}
    \left|a_n - M_n\left(\frac{1}{n+1}\right)\right| = \left|M_n\left(\frac{1}{n}\right) - M_n\left(\frac{1}{n+1}\right)\right| \simeq \frac{1}{n^2}.
    \end{align}
    Since $f_n$ is $(1+O(\delta^n))$-quasiconformal and agrees with $M_n$ at the points $0, \frac{1}{n}, p_n$, {by Lemma \ref{lem:qc}} we have
    $$
    d_{X_n}\left(a_{n+1}, M_n\left(\frac{1}{n+1}\right)\right) = d_{X_n}\left(f_n\left(\frac{1}{n+1}\right), M_n\left(\frac{1}{n+1}\right)\right) = O(\delta^n),
    $$
    where $d_{X_n}$ is the hyperbolic metric on $X_n \coloneqq \widehat \C \setminus  \{0, a_n, \infty\}$. By Lemma \ref{lemma:hyperbolic_euclidean}, there exists a uniform constant $c'>0$ such that
    \begin{align}\label{eqn:6.2}
        \left|a_{n+1} - M_n\left(\frac{1}{n+1}\right)\right| \leq c'a_n \cdot d_{X_n}\left(a_{n+1}, M_n\left(\frac{1}{n+1}\right)\right) = O(\delta^n).
    \end{align}
    By combining \eqref{eqn:6.1} and \eqref{eqn:6.2}, we conclude that $$
    \diam(I^{n}\setminus I^{n+1}) = |a_n-a_{n+1}| \simeq \frac{1}{n^2}.
    $$
   This completes the proof for the case of fixed points. The proof for the case of higher period is similar.

    Now suppose that $a$ is strictly pre-periodic. We continue to assume that $C=\R\cup \{\infty\}$. For $n\geq 1$, let $I^n$ be the interval of level $n$ (i.e., $I^n$ is the intersection of a level $n$ tile with the circle $C$) having $a$ as an endpoint and so that $(a,a+\epsilon)\subset I^n$ for some $\epsilon>0$ depending on $n$. Let $b= \Psi(a)$, and $\widetilde{I}^n = \Psi(I^{n+1})$ be the corresponding interval of level $n$ that has $b$ as an endpoint.
    We argue by induction and suppose that we have
    $$
    \diam \widetilde{I}^{n} \simeq \frac{1}{n} \quad \textrm{and} \quad \diam(\widetilde{I}^{n}\setminus \widetilde{I}^{n+1}) \simeq \frac{1}{n^2}.
    $$
    By applying the same argument as above and substituting the conjugacy map $\Phi$ in the periodic case with the map $\Psi$ in the strictly periodic case, we conclude that 
    $$
    \diam I^{n} \simeq \frac{1}{n} \quad \textrm{and} \quad \diam(I^n\setminus I^{n+1}) \simeq \frac{1}{n^2}.
    $$
    This completes the proof.
    \end{proof}

\section{Geometric and dynamical criteria}\label{section:geom_dyn}

In this section we establish the main tools for the proofs of Theorems \ref{thm:QU} and \ref{thm:DU} regarding the geometric and dynamical criteria.
We will use the following standard facts from dynamics of rational maps; see \cite{Milnor:dynamics}*{Theorem 11.17}.

\begin{theorem}\label{theorem:facts_dynamics}
    Let $R$ be a rational map without Julia critical points. Then each periodic point of $R$ is either attracting or repelling or parabolic. In particular, there are no Cremer points and Siegel disks. 
\end{theorem}

Recall that a point of a gasket Julia set is called a {contact point} if it is the unique intersection point of the boundaries of two Fatou components.

\begin{lemma}\label{lem:pt}
    Let $R$ be a rational map without Julia critical points whose Julia set $\mathcal J(R)$ is a gasket. The following statements are true.
    \begin{enumerate}[label=\normalfont(\roman*)]
        \item\label{lem:pt:i} Every contact point is eventually mapped to a periodic contact point.
        \item\label{lem:pt:ii} Every periodic contact point is either repelling or parabolic with multiplicity $2$ or $3$.
        \item\label{lem:pt:iii} If there is no periodic point that is parabolic with multiplicity $2$, then every Fatou component is a quasidisk.
    \end{enumerate} 
\end{lemma}

\begin{proof}
    Let $x$ be the contact point of two Fatou components $U, V$; note that $x$ is unique, since $\mathcal J(R)$ is a gasket. Since there are no Julia critical points, $R^{\circ n}(U)$ and $R^{\circ n}(V)$ are two distinct Fatou components touching at $R^{\circ n}(x)$ for each $n\in \N$. Since $U, V$ are pre-periodic, $x$ is pre-periodic. This proves part \ref{lem:pt:i}. Suppose that $x$ is a periodic contact point. Since $x$ is in the Julia set, by Theorem \ref{theorem:facts_dynamics} it is either repelling or parabolic. Suppose $x$ is parabolic with multiplicity strictly greater than $3$. Then there are at least $3$ attracting petals at $x$. Thus there are at least $3$ Fatou components touching at $x$, which contradicts the definition of a gasket. This proves \ref{lem:pt:ii}.

    Finally, we prove \ref{lem:pt:iii}. Since $R$ has no Julia critical points, $R$ is locally bi-Lipschitz on the Julia set. Hence, it suffices to prove the statement for periodic Fatou components. Also, by Theorem \ref{theorem:facts_dynamics}, every point in the Julia set that lies in the post-critical set of $R$ is a parabolic periodic point. Let $U$ be a periodic Fatou component.  
    
    Suppose that $U$ is an immediate basin of an attracting point with period $p$.  Suppose that $\partial U$ intersects the post-critical set, so by the above it contains a parabolic periodic point $a$. By assumption, the multiplicity of this point is at least $3$. This implies that there exist at least two parabolic basins in $\widehat \C\setminus \overline U$ meeting at $a$. This is a contradiction, since $\mathcal J(R)$ is a gasket and no point can be in the boundary of three Fatou components.  Thus, the post-critical set is disjoint from $\partial U$. By applying Koebe's distortion theorem locally to iterates of the map $R^{\circ p}|_{\partial U}$ we conclude that $\partial U$ is a quasidisk. 
    
    Next, suppose that $U$ is an immediate basin of a parabolic point $a$. The above argument shows that the post-critical set intersects $\partial U$ only at the point $a$. By \cite{LN24}*{Theorem 1.7 (3)} we conclude that $U$ is a quasidisk.
\end{proof}

\begin{lemma}\label{lemma:tangent_repelling}
    Let $R$ be a rational map such that two fixed Fatou components $U,V$ are Jordan regions whose boundaries intersect at a repelling fixed point $z_0$. 
    \begin{enumerate}[label=\normalfont(\arabic*)]
        \item\label{lemma:tangent_repelling:1} If $U$ and $V$ are tangent to each other at $z_0$, then $\J(R)$ is a Jordan curve and in particular it is not a gasket.
        \item\label{lemma:tangent_repelling:2} There exists no quasiconformal homeomorphism of the sphere that maps $U$ and $V$ to two disks that are tangent to each other.
    \end{enumerate}
\end{lemma}

\begin{proof}
    There exists a neighborhood $W$ of $z_0$ and a conformal map $\phi\colon W\to \D$ such that $\phi(z_0)=0$ and $\phi$ conjugates $R$ to $z\mapsto \lambda z$, where $|\lambda|>1$. By conformality the tangency is preserved. Specifically, we may assume that for each $\varepsilon>0$ there exists $\delta>0$ with the property that $\{re^{i\theta}: r\in (0,\delta), |\theta|<\pi/2-\varepsilon\}\subset \phi(W\cap U)$ and $\{re^{i\theta}: r\in (-\delta,0), |\theta|<\pi/2-\varepsilon\}\subset \phi(W\cap V)$. Given that $R$ fixes $U$, we see that for each $\theta\in (-\pi/2,\pi/2)$ and all sufficiently small $r>0$ the point $\lambda re^{i\theta}$ cannot lie in the left half-plane $\{z\in \C: \Re(z)<0\}$. This implies that $\lambda$ is a positive real number. Given a point in the right half-plane $\{z\in \C: \Re(z)>0\}$, we have $\lambda^{-n}z\in \phi(W\cap U)$ for some $n\in \N$. The invariance of $U$ implies that $\D\cap \{z\in \C:\Re(z)>0\}\subset \phi(W\cap U)$. Similarly, $\D\cap \{z\in \C:\Re(z)<0\}\subset \phi(W\cap V)$. Thus, $\partial U\cap \partial V$ contains a Jordan arc $E=W\cap \mathcal J(R)$. There exists $N\in \N$ such that $R^{\circ N}(W)\supset \J(R)$. The invariance of $U$ and $V$ implies that $R$ does not have other Fatou components, so $\J(R)=\partial U=\partial V$.

    Suppose that there exists a quasiconformal homeomorphism $\psi\colon \widehat \C\to \widehat \C$ mapping $U$ and $V$ to disks that are tangent to each other. Consider the linearizing coordinate $\phi$ as above. Since $\psi$ is quasisymmetric, each arc in $\phi(W\cap \partial U)$ and in $\phi(W\cap \partial V)$ is a quasiarc. Using the dynamics of $z\mapsto \lambda z$ we can show that each arc in $\phi(W\cap \partial U)\cup \phi(W\cap \partial V)$ is also a quasiarc. This is a contradiction.   
\end{proof}

We derive the following dynamical characterization of fat gasket Julia sets. 

\begin{lemma}\label{lem:tpp}
    Let $R$ be a rational map without Julia critical points whose Julia set $\J(R)$ is a gasket. The following are equivalent.
    \begin{enumerate}[label=\normalfont (\arabic*)]
        \item $\mathcal{J}(R)$ is a fat gasket.
        \item Every contact point between two Fatou components is eventually mapped to a periodic contact point that is parabolic with multiplicity $3$.
    \end{enumerate}
    In that case, each periodic Fatou component $U$ has exactly one periodic contact point on its boundary, and every periodic point on $\partial U$ that is not a contact point is repelling. In particular, each Fatou component is a quasidisk.
\end{lemma}
\begin{proof}
    We first assume that $\mathcal{J}(R)$ is a fat gasket. Let $x$ be a periodic contact point between two Fatou components $U, V$. By Lemma \ref{lem:pt} \ref{lem:pt:ii}, $x$ is either repelling or parabolic with multiplicity $2$ or $3$. If $x$ is repelling, then by Lemma \ref{lemma:tangent_repelling} \ref{lemma:tangent_repelling:1} the two Fatou components cannot intersect tangentially. Thus $x$ is parabolic. If the multiplicity at $x$ is $2$, then there is a unique attracting petal at $x$ and one Fatou component, say $U$, has a cusp at $x$. In particular, the components $U, V$ cannot be tangent at $x$. Therefore, $x$ must be parabolic with multiplicity $3$. By Lemma \ref{lem:pt} \ref{lem:pt:i}, every contact point is eventually mapped to a periodic contact point, which has to be a parabolic point with multiplicity $3$ by the above.
    
    Conversely, suppose that every periodic contact point is parabolic with multiplicity $3$.
    Let $x$ be such a point. By using the Fatou coordinate of $x$, we see that the two adjacent Fatou components are tangent at $x$ (see also \cite{CarlesonGamelin}*{Section II.5}). By Lemma \ref{lem:pt} \ref{lem:pt:i}, every contact point is pre-periodic. Since the tangency condition is invariant under pull-back of local conformal maps, we conclude any two Fatou components with a common point at the boundary are tangent to each other. Thus, $\mathcal{J}(R)$ is a fat gasket. 
    
    We now prove the last assertion. If $U$ is a periodic Fatou component, then $\partial U$ must contain some contact point  because $\mathcal J(R)$ is a gasket. By Lemma \ref{lem:pt} \ref{lem:pt:i}, this contact point is pre-periodic, so there exists a periodic contact point $x\in \partial U$. By the above, $x$ is parabolic with multiplicity $3$ and attracts points in $U$; in particular the point $x$ is the unique periodic contact point on $\partial U$. Each periodic point in $\partial U\setminus \{x\}$ cannot attract orbits so it is not parabolic, but it is repelling by Theorem \ref{theorem:facts_dynamics}. In particular, $\partial U$ has no periodic parabolic points with multiplicity $2$. We conclude that $R$ has no periodic parabolic points with multiplicity $2$. By Lemma \ref{lem:pt} \ref{lem:pt:iii}, every Fatou component is a quasidisk.  This concludes the lemma.
\end{proof}

\begin{remark}
  It is proved in \cite{LuoZhang:Nonequiv}*{Theorem 3.1 and Theorem 4.1} that there is a combinatorial obstruction for the existence of two or more periodic contact points. Specifically, a fat gasket has a unique parabolic cycle and this cycle consists of a single parabolic fixed point. We do not need this more refined classification in this paper. 
\end{remark}

\begin{lemma}\label{lemma:bilip_blowup}
    Let $R$ be a rational map without Julia critical points and with connected Julia set. Then there exists a finite collection of Fatou components $V_1,\dots,V_l$ {that contain all critical orbits} and $L\geq 1$ such that for every Fatou component $W$ the following statements are true.
    \begin{enumerate}[label=\normalfont(\arabic*)]
        \item\label{lemma:bilip_blowup:1} There exists a minimal $n\geq 0$ such that $R^{\circ n}(W)=V_j$ for some $j\in \{1,\dots,l\}$. 
        \item\label{lemma:bilip_blowup:2} $R^{\circ n}$ maps $W$ conformally onto $V_j$ and for every $z,w\in W$ we have
        $$L^{-1} \frac{\sigma(z,w)}{\diam W}\leq \sigma(R^{\circ n}(z),R^{\circ n}(w)) \leq L \frac{\sigma(z,w)}{\diam W}.$$
    \end{enumerate}
    Moreover, we may have that the collection $V_1,\dots,V_l$ contains any given finite collection of Fatou components.  
\end{lemma}

\begin{proof}
    By conjugating $R$ with an isometry of the sphere, we may assume that the point at $\infty$ is a periodic point in the Fatou set. We define the \textit{rank} $r(W)$ of a Fatou component $W$ to be the smallest index $j\geq 0$ such that $R^{\circ j}(W)$ is a periodic Fatou component. Since there are no critical points on the Julia set, by Theorem \ref{theorem:facts_dynamics} each critical point is attracted to a parabolic cycle or to an attracting cycle. Thus, there exist finitely many Fatou components that intersect the critical orbits of $R$. Note there are only finitely many Fatou components attached to a parabolic periodic point (namely, if a Fatou component contains a (parabolic) periodic point in its boundary, it must be itself periodic). Therefore, there exist finitely many Fatou components whose closure intersects the post-critical set. We also consider a given finite collection of Fatou components as in the very last part of the lemma. Let $m\geq 0$ be the maximum of the ranks among the given family of components and among components whose closure intersects the post-critical set. Denote by $V_1,\dots,V_l$ the family of all Fatou components of $R$ with rank at most $m+1$. 
    
    Let $W$ be a Fatou component of rank $r(W)>m+1$. Then $W$ is mapped onto $V_j$ for some $j\in \{1,\dots,l\}$ after $n(W)=r(W)-m-1$ iterations. Note that $n(w)$ is the smallest number with this property, as required in \ref{lemma:bilip_blowup:1}. The Fatou components $W,V_j$ are subsets of $\C$ and they are both simply connected, since the Julia set is connected. By our choice of $m$, there is a neighborhood of $\overline{V_j}$ that is disjoint from the the post-critical set of $R$. This implies that $R^{\circ n(W)}$ maps a neighborhood of $\overline{W}$ conformally onto a neighborhood of $\overline{V_j}$.  Moreover, by Koebe's distortion theorem, applied to a suitable branch of $(R^{\circ n(W)})^{-1}$, \ref{lemma:bilip_blowup:2} holds with the Euclidean rather than the spherical metric. Since the two metrics are comparable away from the periodic Fatou component that contains the point at $\infty$, we obtain the desired conclusion.    
\end{proof}

\begin{corollary}\label{corollary:quasidisks}
    Let $R$ be a rational map without Julia critical points whose Julia set $\mathcal J(R)$ is a gasket. If there is no periodic point that is parabolic with multiplicity $2$, then the Fatou components of $R$ are uniform quasidisks.
\end{corollary}
\begin{proof}
    By Lemma \ref{lem:pt}, each Fatou component is a quasidisk. In particular, the Fatou components $V_1,\dots,V_l$ from Lemma \ref{lemma:bilip_blowup} are quasidisks. By Lemma \ref{lemma:bilip_blowup}, every other Fatou component is mapped to one of $V_1,\dots,V_l$ with an iterate of $R$ that is uniformly quasisymmetric in the spherical metric. Hence, all Fatou components are uniform quasidisks. 
\end{proof}

\section{Construction of a subdivision rule from a gasket Julia set}

In this section we show how a gasket Julia set gives rise to a finite subdivision rule. The following statement is the main result of the section. Recall the discussion in Section \ref{subsec:MM} regarding subdivision homomorphisms.

\begin{proposition}\label{prop:subdivision_g}
    Let $f$ be a rational map without Julia critical points whose Julia set $\mathcal J(f)$ is a gasket. Then there exists a simple, irreducible, acylindrical finite subdivision rule $\mathcal R$ and a graph $\mathcal G^0$ that subdivides the sphere into polygons giving rise to a CW complex $X$ so that the following statements are true.
    \begin{enumerate}[label=\normalfont(\arabic*)]
        \item\label{sg:1} The $1$-skeletons $\mathcal G^n$ of the subdivision $\mathcal R^n(X)$, $n\geq 0$, have the contact graph $\mathcal G$ of the gasket $\mathcal J(f)$  as a direct limit.
        \item\label{sg:2} Each periodic Fatou component of $f$ corresponds to a vertex of $\mathcal G$ that is also a vertex of $\mathcal G^0$.   
        \item\label{sg:3} There exists $N\in \N$ so that each periodic Fatou component of $f^{\circ N}$ is fixed and the map $f^{\circ N}$ induces a subdivision homomorphism between $\mathcal G^1$ and $\mathcal G^0$.
        \item\label{sg:4} The standing assumptions \ref{condition:siL}, \ref{condition:siG}, \ref{condition:si} in Section \ref{sec:cpsr} are satisfied.
    \end{enumerate} 
\end{proposition}

The proof relies on the following conjugacy result.

\begin{lemma}\label{thm:pcfrepresentative}
    Let $f$ be a rational map without Julia critical points whose Julia set $\mathcal J(f)$ is a gasket. There exists a hyperbolic and post-critically finite rational map $g$, and a homeomorphism $\phi\colon \widehat\C \to \widehat\C$ so that 
    $$
        g \circ \phi (z) = \phi \circ f(z) \quad \text{for all} \quad z\in \mathcal J(f).
    $$ 
\end{lemma}
\begin{proof}
    By \cite{CT18}*{Theorem 1.3 and Theorem 1.4}, there exists a sub-hyperbolic rational map $\tilde{g}$ and a homeomorphism $\tilde{\phi}\colon \mathcal J(f) \to \mathcal J(\tilde{g})$ conjugating the dynamics of $f$ and $\tilde g$. Moreover, $\tilde{\phi}$ is a uniform limit of homeomorphisms of $\widehat \C$. If $U$ is a complementary component of $\mathcal{J}(f)$, then $U$ is a Jordan region because $\mathcal J(f)$ is a gasket. Since $\tilde \phi|_{\mathcal J(f)}$ is a homeomorphism, $\tilde \phi(\partial U)$ is a Jordan curve. Since $\tilde \phi$ is a limit of homeomorphisms, there exists a Jordan region $V$ bounded by $\tilde \phi(\partial U)$ such that $V$ is disjoint from $\mathcal J(\tilde g)$. Thus, $\tilde \phi$ extends to a homeomorphism from the $\overline U$ to $\overline V$. We repeat this procedure for every complementary component of $\mathcal J(f)$. Using the local connectivity of the gaskets $\mathcal J(f),\mathcal J(\tilde g)$, we conclude that $\tilde{\phi}$ extends to a homeomorphism of $\widehat\C$.
    
    Since $f$ has no Julia critical points, the topological conjugacy implies that $\tilde{g}$ has no Julia critical points. Thus, $\tilde{g}$ is hyperbolic. Since $\mathcal J(\tilde{g})$ is connected, there exists a hyperbolic post-critically finite rational map $g$ and a homeomorphism of $\widehat\C$ which restricts to a topological conjugacy between the Julia set $\mathcal{J}(\tilde{g})$ and $\mathcal{J}(g)$ (see \cite{McM88}*{Corollary 3.6}). This proves the lemma.
\end{proof}

\begin{proof}[Proof of Proposition \ref{prop:subdivision_g}]

Let $f$ be a rational map without Julia critical points whose Julia set $\mathcal J(f)$ is a gasket. Let $g$ be the hyperbolic post-critically finite rational map corresponding to $f$ as provided by Lemma \ref{thm:pcfrepresentative}. Using the map $g$ we will construct a spherical subdivision rule that has the contact graph $\mathcal G$ of the gasket $\mathcal J(g)$ as a spherical subdivision graph so that the conclusions of Proposition \ref{prop:subdivision_g} are satisfied for $g$. Using the homeomorphism $\phi$ from Lemma \ref{thm:pcfrepresentative}, which identifies the contact graph of $\mathcal J(g)$ with the contact graph of $\mathcal J(f)$ and is a conjugacy on the Julia sets, we will draw the desired conclusion for $f$.

Since the map $g$ is post-critically finite, in each Fatou component, there exists a unique point whose orbit contains critical points, which is called the \textit{center} of the corresponding Fatou component. Suppose $U, V$ are two Fatou components whose boundaries intersect at a point $x$.
By Lemma \ref{lem:pt} \ref{lem:pt:i}, $x$ is pre-periodic.
Therefore, there is a pre-periodic internal ray connecting the center of $U$ (resp.\ $V$) to $x$. Thus, any finite subgraph $\mathcal{H}$ of the contact graph $\mathcal{G}$ of the gasket $\mathcal J(g)$ can be realized canonically as a plane graph by identifying 
    \begin{itemize}
        \item the vertices of $\mathcal{H}$ with the centers of the corresponding Fatou components of $g$, and
        \item each edge of $\mathcal{H}$ with the closure of the union of two internal rays connecting the corresponding centers of the Fatou components and the contact point.
    \end{itemize}

Abusing notation, we shall not distinguish a finite subgraph of $\mathcal{G}$ from its canonical realization. Each component of $\widehat\C \setminus \mathcal{H}$ gives rise to a CW complex consisting of one $2$-cell that we call a \textit{face} of $\mathcal{H}$. The interior of that $2$-cell is naturally identified with the corresponding component of $\widehat \C\setminus \mathcal H$.

There exists a connected, induced, finite subgraph $\mathcal G_0$ of $\mathcal{G}$ that is forward invariant and contains the post-critical set $P(g)$. Here, induced means that $\mathcal{G}^0$ contains all edges in $\mathcal{G}$ connecting vertices in $\mathcal{G}^0$. The graph $\mathcal{G}^0$ is not unique and can be constructed in the following way. Since $\mathcal J(g)$ is a gasket, there exists a finite connected graph $\mathcal{H}^0$ containing all critical values of $g$. Since each Fatou component is eventually periodic, the union $\bigcup_{n=0}^\infty g^n(\mathcal{H}^0)$ is a finite and forward invariant graph that contains the post-critical set. We obtain $\mathcal{G}^0$ by adding to the graph $\bigcup_{n=0}^\infty g^n(\mathcal{H}^0)$ all edges of $\mathcal G$ connecting vertices of $\bigcup_{n=0}^\infty g^n(\mathcal{H}^0)$. We define
$$
\mathcal{G}^{n}\coloneqq g^{-n}(\mathcal{G}^0), \, n\in \N.
$$

Since $\mathcal{G}^0$ is connected, each component of $\widehat\C\setminus \mathcal G^0$ is simply connected. Since $\mathcal{G}^0$ contains the post-critical set of $g$, each component of $\widehat \C\setminus \mathcal{G}^1$ is simply connected and is mapped conformally by $g$  to a component of $\widehat \C\setminus \mathcal{G}^0$. Thus, $\mathcal{G}^1$ is also connected. Since $\mathcal{G}^0$ is forward invariant, we have $\mathcal{G}^0 \subset \mathcal{G}^1$. Therefore, each face of $\mathcal{G}^0$ is the union of the faces of $\mathcal{G}^1$ that are contained in that face.

We can construct a finite subdivision rule $\mathcal{R}$ as follows.
Associate to each face $F_i$ of $\mathcal{G}^0$ a polygon $P_i$, whose number of sides is equal to the number of sides in the ideal boundary of the face. The polygon $P_i$ is subdivided into finitely many polygons $P_{i,j}$, where each polygon $P_{i,j}$ corresponds to a face $F_{i,j}$ of $\mathcal{G}^1$ in $F_i$. Note that the image $g(F_{i,j})$ is a face $F_{\sigma(i,j)}$ of $\mathcal{G}^0$. Each conformal map $g^{-1}\colon  \inter F_{\sigma(i,j)} \to \inter F_{i,j}$ induces a cellular map $\psi_{i,j}\colon P_{\sigma(i,j)} \to P_{i,j}$.

Since every vertex of $\mathcal{G}$ is eventually mapped to a vertex in $\mathcal{G}^0$, every edge of $\mathcal{G}$ is eventually mapped to an edge in $\mathcal{G}^0$ as well.
Thus, the contact graph $\mathcal{G}$ of the gasket $\mathcal J(g)$ is the exactly the direct limit of the graphs $\mathcal{G}^n$, $n\geq 0$, so $\mathcal{G}$ is a spherical subdivision graph for $\mathcal{R}$. Since $\mathcal{G}^0$ is an induced subgraph of $\mathcal{G}$, i.e., $\mathcal{G}^0$ contains all edges in $\mathcal{G}$ connecting vertices in $\mathcal{G}^0$, we also conclude that $\mathcal{R}$ is irreducible. By Lemma \ref{lemma:gasket_subdivision_rule}, the subdivision rule $\mathcal R$ is simple and acylindrical. 

Note that the faces of the graphs $\mathcal G^n$, $n\geq 0$, are not necessarily polygons. To amend this, we will modify the graph $\mathcal G^0$ and the subdivision rule $\mathcal R$.

By the above, $\mathcal{R}$ is simple, irreducible and acylindrical. Thus, by Proposition \ref{prop:jd}, we can obtain a graph $\widetilde {\mathcal G}^0$ with
$$
\mathcal{G}^0 \subset \widetilde{\mathcal{G}}^0 \subset \mathcal{G}^K
$$
for some constant $K\in \N$, so that $\widetilde{\mathcal{G}}^0$ is connected, every face $F$ of $\widetilde{\mathcal{G}}^0$ is a polygon and $\partial F$ is an induced subgraph of $\mathcal{G}$. Note that $\widetilde{\mathcal{G}}^0$ is forward invariant under $g^{\circ N}$ for all $N \geq K$ and contains the post-critical set $P(g^{\circ N}) = P(g)$. Thus, we can construct a subdivision rule $\widetilde{\mathcal R}$ from $\widetilde{\mathcal{G}}^0$ using the map $g^{\circ N}$ in the same way as we obtained $\mathcal{R}$ from $\mathcal G^0$ using the map $g$. Since the $1$-skeleton of every face of $\mathcal G^0$ is an induced subgraph of $\mathcal G$, we conclude that $\widetilde{\mathcal R}$  is irreducible. 

Since all faces of $\widetilde{\mathcal G}^0$ are polygons and all post-critical points are vertices of $\widetilde{\mathcal G}^0$, we conclude that every face $F^1$ of $\widetilde{\mathcal{G}}^1 = g^{-N}(\widetilde{\mathcal{G}}^0)$ is a polygon and $\partial F^1$ is an induced subgraph of $\mathcal{G}$; see \cite{BM17}*{Lemma 5.12}.
Thus, inductively, for each $n\geq 0$, every face $F^n$ of $\widetilde{\mathcal{G}}^n$ is a polygon and $\partial F^n$ is an induced subgraph of $\mathcal{G}$. 

By construction, conditions \ref{condition:siL} and \ref{condition:siG} are satisfied. We can now invoke again Lemma \ref{lemma:gasket_subdivision_rule} to show that $\mathcal {\widetilde R}$ is simple and acylindrical. Finally, we may replace $\widetilde {\mathcal R}$ and $g^{\circ N}$ with a suitable iterate to ensure that condition \ref{condition:si} is met and that each periodic Fatou component of $g$ is fixed under that iterate.
\end{proof}

\section{Proof of the uniformization theorems}
In this section we complete the proofs of Theorems \ref{thm:QU} and \ref{thm:DU}. The proofs are based on the next result. 

\begin{theorem}\label{thm:JuliaHomeoCP}
    Let $f$ be a rational map without Julia critical points whose Julia set $\mathcal J(f)$ is a gasket. Then there exists a simple, irreducible, acylindrical finite subdivision rule $\mathcal R$, a circle packing $\mathcal{P}$ with spherical subdivision rule of $\mathcal R$, a Markov map $\Psi$, and a homeomorphism $h\colon\widehat\C \to \widehat\C$ so that the following statements are true.
    \begin{enumerate}[label=\normalfont(\arabic*)]
        \item The contact graph of $\mathcal J(f)$ is a spherical subdivision graph for $\mathcal R$.
        \item $h(\mathcal J(f)) = \Lambda(\mathcal P)$.
        \item There exists $N\in \N$ such that $h$ conjugates $f^{\circ N}|_{\mathcal J(f)}$ to $\Psi|_{\Lambda(\mathcal{P})}$.
        \item If each Fatou component of $f$ is a quasidisk, then $h$ is a David map on $\widehat\C\setminus \mathcal J(f)$.
        \item  If $\mathcal J(f)$ is a fat gasket, then $h$ is quasiconformal on $\widehat\C\setminus \mathcal J(f)$.
    \end{enumerate}
\end{theorem}

We establish some preliminaries before giving the proof. We apply Proposition \ref{prop:subdivision_g}, so we obtain a finite subdivision rule $\mathcal R$ corresponding to $f$. By part \ref{sg:3}, there exists $N$ such that every periodic Fatou component is fixed under $R=f^{\circ N}$. By Proposition \ref{prop:subdivision_g} \ref{sg:2}, each vertex corresponding to a fixed Fatou component of $R$ is contained in $\mathcal{G}^0$. By Proposition \ref{prop:subdivision_g} and Theorem \ref{thm:qch}, there exists a circle packing $\mathcal{P}$ with tangency graph $\mathcal{G}$. By Proposition \ref{prop:subdivision_g} \ref{sg:3}, $R$ induces a subdivision homomorphism between $\mathcal{G}^1$ and $\mathcal{G}^0$. By invoking Proposition \ref{prop:induceMarkovMap}, we see that this subdivision homomorphism induces a Markov map
$$
\Psi\colon \Omega^1 \to \Omega^0
$$
for the circle packing $\mathcal P$ that extends to a continuous map from $\overline{\Omega^1}$ to $\overline{\Omega^0}$. The idea for the proof of Theorem \ref{thm:JuliaHomeoCP} is to use David or quasiconformal surgery to construct a David or quasiconformal homeomorphism on the union of Fatou components associated to the vertices in $\mathcal{G}^0$. Then we use the dynamics to pull back this homeomorphism.

Let $n\geq 0$. Let $W_v$ be the Fatou component of $R$ associated to the vertex $v \in \mathcal{G}^n$.
Similarly, let $D_v$ be the disk of $\mathcal{P}$ associated to the vertex $v \in\mathcal{G}^n$.
We define
$$
W^n \coloneqq\bigcup_{{v\in \mathcal{G}^n}} W_v\quad \text{and}\quad  D^n \coloneqq\bigcup_{{v\in \mathcal{G}^n}} D_v.
$$

\begin{lemma}\label{lem:fixconjugatedynamics}
Let $W_v$ be a fixed Fatou component of $R$. Let $u_i$, $i\in \{0,\dots,r\}$, be the collection of vertices of $\mathcal G^1$ that are adjacent to $v$, numbered in counter-clockwise order. For $i\in \{0,\dots,r\}$, let $a_i=\partial W_v\cap \partial W_{u_i}$ and $b_i=\partial D_v\cap \partial D_{u_i}$. The following statements are true.
\begin{enumerate}[label=\normalfont(\arabic*)]
    \item\label{f:1} The set $\{a_0,\dots,a_r\}$ defines a Markov partition for $R|_{\partial W_v}$.
    \item\label{f:2} The set $\{b_0,\dots,b_r\}$ defines a Markov partition for $\Psi|_{\partial D_v}$.
    \item\label{f:3} There exists a homeomorphism $h_v\colon \partial W_v \to \partial D_v$ that maps $a_i$ to $b_i$ for each $i\in \{0,\dots,r\}$ and conjugates $R|_{\partial W_v}$ and $\Psi|_{\partial D_v}$.
    \item\label{f:4} If $W_v$ is a quasidisk, then $h_v$ extends to a David map from $W_v$ onto $D_v$. 
    \item\label{f:5} If $\mathcal J(f)$ is a fat gasket, then $h_v$ extends to a quasiconformal map from $W_v$ onto $D_v$.
\end{enumerate}
\end{lemma}

\begin{proof}
    Since $R$ has no critical points on the Julia set, it is a covering map from $\partial W_v$ onto itself.  By Proposition \ref{prop:subdivision_g} \ref{sg:3}, $R$ induces a subdivision homomorphism between $\mathcal G^1$ and $\mathcal G^0$. This implies that the set $\{a_0,\dots,a_r\}$ gives a Markov partition for $R|_{\partial W_v}$ (see also the proof of Theorem \ref{thm:cmc} \ref{cmc:2}). Consider the preimages under $R^{\circ n}$ of the arcs $A_k=\arc{[a_k,a_{k+1}]}$. Each preimage is contained in the boundary of $W(\inter F^n)$ for some face $F^n$ of $\mathcal G^n$ (recall the notation from Section \ref{section:gasket_subdivision}). Lemma \ref{lemma:gasket_sub} \ref{gs:0} implies that the diameters of those sets converge to $0$. In combination with \ref{expansive:diameters}, we see that $R|_{\partial W_v}$ is expansive.  By Theorem \ref{thm:cmc}, $\Psi|_{\partial D_v}$ is an expansive covering map and the set $\{b_0,\dots,b_r\}$ defines a Markov partition for $\Psi|_{\partial D_v}$.
    
    Consider the map $h_v\colon \{a_0,\dots,a_r\}\to \{b_0,\dots,b_r\}$ defined by $h_v(a_i)=b_i$, $i\in \{0,\dots,r\}$. By Proposition \ref{prop:induceMarkovMap} \ref{ma:4} and Proposition \ref{prop:subdivision_g} \ref{sg:3}, we see that $h_v$ conjugates $R$ to $\Psi$ on $\{a_0,\dots,a_r\}$. The expansivity of the maps implies that $h_v$ extends to a topological conjugacy between $R|_{\partial W_v}$ and $\Psi|_{\partial D_v}$ (see \cite{LN24}*{Lemma 3.4} or \cite{CovenReddy:expansive}*{Property $(2')$, p.~99}). We have therefore proved  \ref{f:1}, \ref{f:2}, and \ref{f:3}.

    Suppose that $W_v$ is a quasidisk, as in \ref{f:4}. Let $\phi_v\colon W_v \to \D$ be a conformal map that extends homeomorphically to the closures so that $B_v\coloneqq \phi_v \circ R \circ \phi_v^{-1}\colon  \overline \D \to \overline \D$ is a Blaschke product of degree $d=d_v$. Let $\{\widetilde{a_0},\dots, \widetilde{a_r}\} = \phi_v(\{a_0, \dots, a_r\})$ be the corresponding Markov partition of $B_v$ in $\mathbb S^1$. Since $W_v$ is a quasidisk, $\partial W_v$ contains no parabolic point of multiplicity $2$ (otherwise, $\partial W_v$ would contain a cusp).  We claim that $\{\widetilde{a_0},\dots, \widetilde{a_r}\}$ contains the parabolic fixed point of $B_v$ on $\mathbb S^1$, if there exists one. Suppose  that $\widetilde{a}$ is the parabolic fixed point of $B_v$ in $\partial \mathbb S^1$, so it attracts orbits of $\D$. Then $\widetilde{a}$ corresponds to a parabolic point $\phi_v^{-1}(\widetilde{a})$ in $\partial W_v$ that has multiplicity $3$ by Lemma \ref{lem:pt}. So it is a contact point between $W_v$ and another fixed Fatou component $W_u$ (recall that every periodic Fatou component of $R$ is fixed). Note that $[v,u]$ must be an edge of $\mathcal{G}^0 \subset \mathcal{G}^1$, so $\widetilde{a} \in \{\widetilde{a_0},\dots, \widetilde{a_r}\}$, as claimed.

    By \cite{LN24}*{Theorem 1.4} and the above claim, the Markov partition $\{\widetilde{a_0},\dots, \widetilde{a_r}\}$ satisfies conditions \ref{condition:hp}, \ref{condition:uv}, \ref{condition:qs}, as defined in Section \ref{section:extension}.  By Theorem \ref{thm:cmc},  $\{b_0,\dots,b_r\}$ satisfies conditions \ref{condition:hp}, \ref{condition:uv}, \ref{condition:qs} and \ref{condition:qs_strong} with every point in $\{b_0,\dots, b_r\}$ being symmetrically parabolic. Consider the map $H_v=h_v\circ \phi_v^{-1}\colon \D\to D_v$, which is a topological conjugacy between $B_v$ and $\Psi$. By Theorem \ref{theorem:extension_generalization}, $H_v$ extends to a David homeomorphism map from $\D$ onto the disk $D_v$. Since $W_v$ is a quasidisk, by Proposition \ref{prop:david_qc_invariance} (iv), $h_v$ extends to a David map from $W_v$ onto $D_v$, as required in \ref{f:4}.

    If $\J (f)$ and thus $\mathcal J(R)$ is a fat gasket, as in \ref{f:5}, then $W_v$ is a quasidisk by Lemma \ref{lem:tpp}. We can obtain a quasiconformal extension of $H_v$ as follows. By Remark \ref{remark:hyp_par_analytic}, each point in $\{\widetilde{a_0},\dots, \widetilde{a_r}\}$ is either symmetrically hyperbolic or symmetrically parabolic for $B_v$. Since each point in $\{\widetilde{a_0},\dots, \widetilde{a_r}\}$ corresponds to a contact point of two Fatou components, by Lemma \ref{lem:tpp} there is a unique periodic point $\widetilde a\in \{\widetilde{a_0},\dots, \widetilde{a_r}\}$. By the same lemma, this point corresponds to a parabolic contact point $a\in \partial W_v$ with multiplicity $3$, which therefore attracts orbits in $W_v$ (given that $\mathcal J(R)$ is a gasket so $a$ can be on the boundary of at most two Fatou components). We conclude that the point $\widetilde a\in \mathbb S^1$ is symmetrically parabolic. Each other point in $\{\widetilde{a_0},\dots, \widetilde{a_r}\}$ is mapped to $\widetilde a$ with an iterate of $B_v$, which is locally bi-Lipschitz. Hence, each point in $\{\widetilde{a_0},\dots, \widetilde{a_r}\}$ is symmetrically parabolic. Also, as discussed above, each point in $\{b_0,\dots,b_r\}$ is symmetrically parabolic by Theorem \ref{thm:cmc}. By Theorem \ref{theorem:extension_generalization}, the topological conjugacy $H_v\colon  \partial \D \to \partial D_v$ extends to a quasiconformal homeomorphism between the disks $\D$ and $D_v$. Therefore, $h_v$ extends to a quasiconformal homeomorphism from $W_v$ onto $D_v$.
\end{proof}

\begin{lemma}\label{lem:preperiodic}
    Let $W_v$ be a Fatou component of $R$, where $v$ is a vertex of $\mathcal G^1$.  Let $u_i$, $i\in \{0,\dots,r\}$, be the collection of vertices of $\mathcal G^1$ that are adjacent to $v$, numbered in counter-clockwise order. For $i\in \{0,\dots,r\}$, let $a_i=\partial W_v\cap \partial W_{u_i}$ and $b_i=\partial D_v\cap \partial D_{u_i}$. There exists a homeomorphism $h_v\colon \partial W_v\to \partial D_v$ satisfying the following conditions. 
    \begin{enumerate}[label=\normalfont(\arabic*)]
    \item\label{fp:1} If $u$ is a vertex of $\mathcal G_1$ such that $R(W_v)=W_u$, then $\Psi \circ h_v=h_u\circ R$ on $\partial W_v$.
    \item\label{fp:2} $h_v(a_i)=b_i$ for each $i\in \{0,\dots,r\}$. 
    \item\label{fp:3} If $W_v$ is a quasidisk, then $h_v$ extends to a David map from $W_v$ onto $D_v$. 
    \item\label{fp:4} If $\mathcal J(f)$ is a fat gasket, then $h_v$ extends to a quasiconformal map from $W_v$ onto $D_v$.
\end{enumerate}
\end{lemma}
\begin{proof}
     If $W_v$ is a fixed Fatou component of $R$, then $h_v$ is provided by Lemma \ref{lem:fixconjugatedynamics} and satisfies \ref{fp:1}, \ref{fp:3}, and \ref{fp:4}. Moreover, if $u$ is a vertex of  $\mathcal G^1$ different from $v$ such that $\partial W_v\cap \partial W_u=\{a\}$, then by Lemma \ref{lem:fixconjugatedynamics},  $h_v$ maps  the point $a$ to the point $b=\partial D_{v}\cap \partial D_u$. Hence \ref{fp:2} is also true in that case.
    
    Next, we define $h_v$ for components $W_v$ that are not fixed, where $v$ is a vertex of $\mathcal G^1$.  Note that $W_v$ is eventually mapped to a fixed component. Suppose that $W_u=R(W_v)$ is a fixed component.  Let $h_u\colon W_u \to D_u$ be the David or quasiconformal homeomorphism provided by Lemma \ref{lem:fixconjugatedynamics}, which extends to a homeomorphism of the closures, and define $H_u = h_u \circ \phi_u^{-1}\colon \bar \D \to \bar{D_u}$. Let $\phi_u, \phi_v$ be conformal maps from $W_u, W_v$ to $\D$, which extend to homeomorphisms from $\overline{W_u}, \overline{W_v}$ to $\overline \D$, respectively. By Schwarz reflection we see that the map $H_{{v\to u}}= \phi_u \circ R \circ \phi_v^{-1}\colon \partial \D \to \partial \D$ is an orientation-preserving analytic covering map of degree $d_v \geq 1$. Therefore, by Lemma \ref{lem:qrextension} (if $d_v\geq 2$) or by the Beurling--Ahlfors extension theorem \cite{BeurlingAhlfors:extension} (if $d_v=1$), we can extend $H_{v\to u}$ to a quasiregular map $\widetilde{H_{v\to u}}\colon \D \to \D$ with $\widetilde{H_{v\to u}}(0) = 0$ so that $0$ is the unique critical point if $d_v\geq 2$.
    
    We will also extend the map $\Psi\colon \partial D_v \to \partial D_u$, which is a covering map of degree $d_v$ (see Proposition \ref{prop:induceMarkovMap} \ref{ma:3}). By the properties of a Markov map (see Proposition \ref{prop:induceMarkovMap} \ref{ma:1}), $\Psi$ is quasisymmetric on each arc of the partition of $\partial D_v$ induced by the tiles $\Omega_F$, where $F$ is a face of $\mathcal G^1$. Let $I^\pm$ be two adjacent such arcs. By Theorem \ref{thm:asympConf}, $\Psi|_{\overline {I^+}}$ and $\Psi_{\overline {I^-}}$ are $C^{1+\alpha}$-conformal at the common endpoint $b = \overline {I^+} \cap \overline {I^-}$. By Lemma \ref{lemma:qs_glue}, we conclude that $\Psi|_{\overline {I^+\cup I^-}}$ is quasisymmetric.  Thus, $\Psi|_{\partial D_v}$ is locally quasisymmetric. By Lemma \ref{lem:qrextension}, we can extend $\Psi|_{\partial D_v}$ to a quasiregular map $\widetilde{\Psi}\colon D_v\to D_u$ having a unique branch point at the center $x_v$ of $D_v$ if $d_v \geq 2$. If $d_v=1$, we obtain a quasisymmetric extension as in the previous paragraph by the Beurling--Ahlfors theorem. By post-composing with a M\"obius transformation of $D_u$, we may also assume that $\widetilde \Psi(x_v)=H_u(0)$.

    Then $H_u\circ \widetilde{H_{v\to u}}\colon \bar \D \setminus \{0\} \to \bar{D_u} \setminus \{H_u(0)\}$ and $\widetilde{\Psi}\colon \bar{D_v} \setminus\{x_v\} \to \bar{D_u} \setminus \{H_u(0)\}$ are both coverings of degree $d_v$. Therefore, by the lifting property of covering maps, we can lift $H_u$  to a homeomorphism $H_v \colon \bar\D \setminus \{0\} \to \bar{D_v} \setminus\{x_v\}$; see the diagram in Figure \ref{fig:diagram}. We remark that the deck transformations of the covering $\widetilde{\Psi}\colon \bar{D_v} \setminus\{x_v\} \to \bar{D_u} \setminus \{H_u(0)\}$ are isomorphic to the group $\Z/(d_v\Z)$, so there are $d_v$ different choices of a lift of $H_u$, and any two of them differ by post-composition with an element of $\Z/(d_v\Z)$. Let $[u,q]$ be an edge of $\mathcal{G}^0$ adjacent to $u$ and let $[v,q']$ be an edge of $\mathcal{G}^1$ adjacent to $v$ that is mapped to $[u,q]$; equivalently, there exist points $a,a'$ such that $\partial W_u\cap \partial W_q=\{a\}$, $\partial W_v\cap\partial W_{q'}=\{a'\}$, and $R$ maps $W_{q'}$ to $W_q$. Let $\widetilde{a}=\phi_u(a)$, $\widetilde a'=\phi_v(a')$, $\{b\}=\partial D_u\cap \partial D_q$, and $\{b'\}=\partial D_v\cap \partial D_{q'}$. Note that $H_u(\widetilde{a}) = b$ and $H_{v\to u}(\widetilde a')=\widetilde a$. We choose the lift $H_v$ so that $H_v(\widetilde{a}') = b'$.
    Then $H_v$ extends to a homeomorphism between $\bar \D$ and $\bar{D_v}$.  We define $h_v= H_v \circ \phi_v\colon \bar{W_v} \to \bar{D_v}$. On $\partial W_v$ we have $\Psi\circ h_v= h_u\circ R$ by construction, so \ref{fp:1} holds.
    By this equality and the fact that the map $h_v$ preserves the cyclic ordering of the points associated to the edges adjacent to $v$, we see that \ref{fp:2} holds by our choice of the lift.
    
    \begin{figure}
        \centering
            $$
            \begin{tikzcd}
            W_v\setminus\phi_v^{-1}(0)\arrow{r}{\phi_v} & \D \setminus \{0\} \arrow{r}{H_v} \arrow[swap]{d}{\widetilde{H_{v\to u}}} & D_v \setminus\{x_v\} \arrow{d}{\widetilde{\Psi}} \\%
            W_u\setminus\phi_u^{-1}(0) \arrow{r}{\phi_u} & \D \setminus \{0\} \arrow{r}{H_u}& D_u \setminus \{H_u(0)\}
            \end{tikzcd}
            $$
        \caption{The commutative diagram in the proof of Lemma \ref{lem:preperiodic}.}
        \label{fig:diagram}
    \end{figure}

    If each Fatou component is a quasidisk, then $H_u$ is a David homeomorphism (see Lemma \ref{lem:fixconjugatedynamics} \ref{f:4} and its proof). Note that locally away from $0$, $H_v$ can be expressed as the composition of quasiconformal maps and a David map. By Proposition \ref{prop:david_qc_invariance} (i) and (iii), we conclude that $H_v$ is locally away from $0$ a David homeomorphism. Let $\mu$ be the Beltrami coefficient of $H_v$. Note that the preimage of $\D \setminus [0,1)$ under $\widetilde{H_{v\to u}}$ consists of $d_v$ components, denoted by $U_1,\dots, U_{d_v}$. Then $H_v|_{U_i}$ is a composition of quasiconformal and David maps. Thus, by Proposition \ref{prop:david_qc_invariance} (i) and (iii), there exist constants $C,\alpha>0$ so that for all $i=1,\dots, d_v$ and any $0 < \varepsilon < 1$, we have
    $$
    \sigma\left(\{z\in U_i\colon|\mu(z)|\geq 1-\varepsilon\}\right) \leq C e^{-\alpha/\varepsilon}.
    $$
    We conclude that 
    $$
    \sigma\left(\{z\in \D\colon|\mu(z)|\geq 1-\varepsilon\}\right)\leq d_vC e^{-\alpha/\varepsilon}.
    $$
    Thus, $H_v$ is a David homeomorphism on $\D$.
    By Proposition \ref{prop:david_qc_invariance} (iv), we conclude that $h_v$ is a David homeomorphism on $W_v$. Thus, \ref{fp:3} holds.
    
    If $\mathcal J(f)$ is a fat gasket, then $H_u$ is a quasiconformal homeomorphism. Therefore, $H_v$, and hence $h_v$, is quasiconformal.
    Thus, \ref{fp:4} holds.

    The general case that $R^{\circ n}(W_v)$ is a fixed component for some $n\geq 1$ is treated similarly by induction.  
\end{proof}

\begin{proof}[Proof of Theorem \ref{thm:JuliaHomeoCP}]
    By Lemma \ref{lem:preperiodic}, there exists a homeomorphism $h^1\colon \overline{W^1}\to \overline{D^1}$ such that $h^1=h_v$ on $\overline{W_v}$ for each vertex $v\in \mathcal G^1$ and $\Psi\circ h^1=h^1\circ R$ on $\partial W^1$. Note that the components of $\widehat\C \setminus \overline{W^n}$ (or $\widehat\C \setminus \overline{D^n})$ are in correspondence with the faces of $\mathcal G^n$, $n\geq 0$, and each component is a Jordan region. For each component $U_F$ of $\widehat \C\setminus \overline{W^1}$ and for the corresponding component $V_F$ of $\widehat \C\setminus \overline {D^1}$ we extend $h^1|_{\partial U_F}\colon \partial U_F\to \partial V_F$ arbitrarily to a homeomorphism between $\overline {U_F}$ and $\overline{V_F}$. In this way, $h^1$ is extended to a homeomorphism of $\widehat \C$.

    Now, we define a homeomorphism $h^2$ as follows. Let $F$ be a face of $\mathcal{G}^1$ that is mapped to a face $F'$ of $\mathcal{G}^0$ under the map induced by $R$. Let $U_F$ and $U_{F'}$ be the components of $\widehat\C \setminus \overline{W^1}$ and $\widehat\C \setminus \overline{W^0}$ associated to $F$ and $F'$. Similarly, let $V_F$ and $V_{F'}$ be the corresponding components of $\widehat\C \setminus \overline{D^1}$ and $\widehat\C \setminus \overline{D^0}$. Then $R\colon U_F \to U_{F'}$ and $\Psi\colon V_F \to V_{F'}$ are homeomorphisms. Moreover, we have $h_1(U_{F'}) = V_{F'}$ and $\Psi \circ h_1 = h_1 \circ R$ on $\partial U_F$. We define
    $$
    h^2(z)= \begin{cases}
        h^1(z), & \text{ if } z\in \overline{W^1},\\
        (\Psi|_{V_F})^{-1} \circ h^1 \circ R(z), & \text{ if } z\in U_F, \text{ where $F$ is a face of $\mathcal G^1$}.
    \end{cases}
    $$
    By construction, $h^2$ is a homeomorphism on $\widehat\C$ that sends $W^2$ to $D^2$ and satisfies $\Psi\circ h^2=h^1\circ R$ on $\widehat \C\setminus W^1$.
    Inductively, for each $n\geq 1$ we construct a homeomorphism $h^{n+1}$ that sends $W^{n+1}$ to $D^{n+1}$, satisfies $\Psi \circ h^{n+1}= h^{n}\circ R$ on $\widehat \C\setminus W^n$, and  satisfies $h^{n+1} = h^m$ on $W^m$ if $n+1\geq m$. By Lemma \ref{lemma:gasket_sub} \ref{gs:0} (see also Proposition \ref{prop:diam->0}), the diameter of each component of $\widehat\C \setminus \overline{W^n}$ and $\widehat\C \setminus  \overline{D^n}$ is shrinking to $0$.  Thus, we conclude that $h^n$ converges to some homeomorphism $h\colon \widehat\C \to \widehat\C$ such that 
    \begin{align}\label{h_conj}
        \textrm{$\Psi \circ h=h\circ R$ on $\widehat \C \setminus  W^1$.}
    \end{align}
    By construction, we have $h(\mathcal J(f)) = \Lambda(\mathcal P)$ and $h$ conjugates the dynamics of $R|_{\mathcal J(R)}$ and $\Psi|_{\Lambda(\mathcal P)}$. It remains to verify the last two conditions of the theorem.

We assume that each Fatou component is a quasidisk. Let $\mu$ be the Beltrami coefficient of $h$ on $\widehat\C\setminus \mathcal J(f)$. Let $V_1,\dots , V_l$ be the collection of Fatou components in Lemma \ref{lemma:bilip_blowup}. By the last part of the lemma, we may also include in this collection all Fatou components associated to vertices in $\mathcal{G}^1$. Let $W_v\in \{V_1,\dots,V_l\}$ be a Fatou component corresponding to a vertex $v$ of $\mathcal G$. If $v$ is a vertex of $\mathcal G^1$, by Lemma \ref{lem:preperiodic} we see that $\mu$ is a David--Beltrami coefficient on $W$. Suppose $v$ is a vertex of $\mathcal G_{n+1}\setminus \mathcal G_n$ for some $n\geq 1$. The disk $D_v$ is contained in an open tile $\Omega_F$ of level $n$ and the map $\Psi^{\circ n}$ is uniformly quasiconformal on $\Omega_F$ by Proposition \ref{prop:qce}. Hence, by \eqref{h_conj} and Proposition \ref{prop:david_qc_invariance} (i) and (iii), we conclude that $\mu$ is also a David--Beltrami coefficient on $W_v$. In particular, $\mu$ is a David--Beltrami coefficient in $\bigcup_{i=1}^l V_l$ and satisfies
\begin{align}\label{surgery:periodic}
    \sum_{j=1}^l \sigma(\{z\in V_j: |\mu(z)|\geq 1-\varepsilon\}) \leq  M e^{-\alpha/\varepsilon}
\end{align}
for some constants $M,\alpha>0$ and for all $\varepsilon \in (0,1)$. 

We now check the David condition on the union of all Fatou components. Let $W\notin \{V_1,\dots,V_l\}$ be a Fatou component and $D$ be the corresponding disk of $\mathcal P$. Note that $W$ and $D$ correspond to a vertex of $\mathcal G^{n+1}\setminus \mathcal G^n$ for some $n\geq 1$. Thus, $D$ is contained in an open tile $\Omega_F$ of level $n$. By Lemma \ref{lemma:bilip_blowup}, there exists a minimal $n(W)\geq 1$ such that $W$ is mapped onto $V_j$ for some $j\in \{1,\dots,l\}$ after $n(W)$ iterations. Let $D_j$ be the corresponding disk of $\mathcal{P}$. By the minimality of $n(W)$, $V_j$ and $D_j$ correspond to a vertex of $\mathcal G^k$ for some $k\geq 1$, thus $n\geq n(W)$. Then by Proposition \ref{prop:qce}, $\Psi^{n(W)}\colon D \to  D_j$ is $K$-quasiconformal for some uniform $K\geq 1$.

By Remark \ref{remark:composition}, the conjugation relation \eqref{h_conj}, and Lemma \ref{lemma:bilip_blowup}, we have
\begin{align}\label{surgery:area}
  \notag\sigma\left(\{z\in W\colon|\mu(z)|\geq 1-\varepsilon\}\right)&\leq \sigma\left( \{z\in W\colon|\mu_{\Psi^{\circ n(W)}\circ h}(z)|\geq 1-K\varepsilon\}\right)\\
  \notag&=\sigma\left( \{z\in W\colon|\mu_{h\circ R^{\circ n(W)}}(z)|\geq 1-K\varepsilon\}\right)\\
  \notag&=\sigma\left( \{z\in W\colon|\mu_{h}(R^{\circ n(W)}(z))|\geq 1-K\varepsilon\}\right)\\
  &\leq L^2(\diam{W})^2\sigma\left(\{z\in V_j\colon|\mu(z)|\geq 1-K\varepsilon\}\right).  
\end{align}
Since $W$ is a quasidisk with a uniform constant by Corollary \ref{corollary:quasidisks}, we conclude (see \cite{Bonk:uniformization}*{Proposition 4.3}) that there exists a uniform constant $C>0$ such that
\begin{align}\label{surgery:diameter}
  (\diam{W})^2\leq C\sigma\left(W\right).
\end{align}
Therefore, by \eqref{surgery:diameter}, \eqref{surgery:area}, and \eqref{surgery:periodic}, we have
\begin{equation}\notag
\begin{split}
\sigma&\left(\{z\in\widehat{\mathbb C}\colon|\mu(z)|\geq 1-\varepsilon\}\right)\\&=\sum_{W \notin \{V_1,..., V_l\}}\sigma\left(\{z\in{W}\colon|\mu(z)|\geq 1-\varepsilon\}\right)+\sum_{j=1}^l \sigma\left(\{z\in{V_j}\colon|\mu(z)|\geq 1-\varepsilon\}\right)\\
&\leq \left(L^2 C\left(\sum_{W \notin \{V_1,..., V_l\}}\sigma\left(W\right)\right)+1\right)\cdot\left(\sum_{j=1}^l\sigma\left(\{z\in V_j\colon|\mu(z)|\geq 1-K\varepsilon\}\right)\right)\\
&\leq \left(L^2C\sigma\left(\widehat{\mathbb C}\right)+1\right)Me^{-\alpha/(K\varepsilon)}
\end{split}
\end{equation}
for $\varepsilon \in (0,1/K)$. This completes the proof that $\mu$ is a David--Beltrami coefficient. 

In the case that $\mathcal J(R)$ is a fat gasket, then $|\mu|$ is bounded away from $1$ by Lemma \ref{lem:preperiodic} and Proposition \ref{prop:qce}. Thus, we conclude that $h$ is quasiconformal on $\widehat\C\setminus \mathcal J(f)$.
\end{proof}

\begin{proof}[Proof of Theorem \ref{thm:QU}] 
    By Lemma \ref{lem:tpp}, we have $(2) \Leftrightarrow (3)$. We show that $(1) \Rightarrow (3)$. Suppose that $(3)$ does not hold. Then by Lemma \ref{lem:pt}, there exists a periodic contact point $x$ that is either repelling or parabolic with multiplicity $2$. In the first case, by Lemma \ref{lemma:tangent_repelling} we obtain that  $\J(R)$ cannot be quasiconformally uniformized by a round gasket. In the second case, one of the Fatou components has a cusp at $x$, so it is not a quasidisk and $\J(R)$ cannot be quasiconformally uniformized by a round gasket.   
    
    To show that $(2) \Rightarrow (1)$, suppose that $\J(R)$ is a fat gasket. By Lemma \ref{lem:tpp}, there are no periodic parabolic points with multiplicity $2$. Corollary \ref{corollary:quasidisks} implies that the Fatou components of $R$ are uniform quasidisks. By Theorem \ref{thm:JuliaHomeoCP}, the contact graph of the gasket $\mathcal J(R)$ is a spherical subdivision graph for a simple, irreducible, acylindrical finite subdivision rule. Moreover, there exists a homeomorphism $\phi\colon \widehat \C\to \widehat \C$ that maps $\mathcal J(R)$ onto the limit set $\Lambda(\mathcal P)$ of a circle packing $\mathcal P$ such that $\phi$ is quasiconformal in each Fatou component. By Theorem \ref{theorem:qc_david_gasket} we conclude that $\phi$ is quasiconformal on $\widehat \C$.
    
    Regarding uniqueness, the contact graph of the circle packing $\mathcal P$ is a spherical subdivision graph $\mathcal G(X)$. By the last part of Theorem \ref{thm:qch}, there exists a unique circle packing up to M\"obius transformations with tangency graph $\mathcal G(X)$. This implies the uniqueness of the map $\phi|_{\J(R)}$.
\end{proof}

\begin{proof}[Proof of Theorem \ref{thm:DU}]
    Suppose that $R$ has a parabolic point $a$ of multiplicity $2$. Then the immediate basin $\Omega$ of $a$ is not a quasidisk as it has a cusp at $a$. Therefore, $(2) \Rightarrow (3)$.
    More generally, by \cite{LN24}*{Theorem 1.7 (1)}, there is no David map $\phi\colon \widehat\C \to \widehat\C$ so that $\phi(\Omega) = \D$. In particular, $\J(R)$ cannot be mapped to a round gasket with a David map. Therefore, $(1) \Rightarrow (3)$.
    By Lemma \ref{lem:pt} \ref{lem:pt:iii}, $(3) \Rightarrow (2)$. 
    
    Finally, we show that $(3) \Rightarrow (1)$ by arguing as in the previous proof.  Corollary \ref{corollary:quasidisks} implies that the Fatou components of $R$ are uniform quasidisks. By Theorem \ref{thm:JuliaHomeoCP}, the contact graph of the gasket $\mathcal J(R)$ is a spherical subdivision graph for a simple, irreducible, acylindrical finite subdivision rule.  Moreover, there exists a homeomorphism $\phi\colon \widehat \C\to \widehat \C$ that maps $\mathcal J(R)$ onto the limit set $\Lambda(\mathcal P)$ of a circle packing $\mathcal P$ such that $\phi$ is a David map in $\widehat \C\setminus \mathcal J(R)$. By Theorem \ref{theorem:qc_david_gasket} we conclude that $\phi$ is a David map on $\widehat \C$.  The uniqueness of $\phi|_{\J(R)}$ follows as in the proof of Theorem \ref{thm:QU}. Note here that the composition of a David map with a quasiconformal map is a David map; see Proposition \ref{prop:david_qc_invariance}.
\end{proof}

\section{Local quasiconformal conjugacy}\label{sec:lqc}
In this section we prove Theorem \ref{theorem:nlo}. Consider the finite subdivision rule $\mathcal{R}$ that is illustrated in Figure \ref{fig:subda}. The subdivision rule consists of a single quadrilateral $P$ that is subdivided into two smaller quadrilaterals $P_1, P_2$. We identify each new quadrilateral $P_i$ with $P$ so that the new vertex (the red vertex in Figure \ref{fig:subda}) is identified with the {lower right} vertex of $P$. The first four graphs $\mathcal{G}^i, i=0, 1, 2, 3$, are illustrated in Figure \ref{fig:subda}.

It is easy to verify that $\mathcal{R}$ is simple, irreducible, and acylindrical.
Let $\mathcal{G} = \lim_{\rightarrow} \mathcal{G}^n$ be the subdivision graph for $\mathcal R$.
We can glue two copies of $\mathcal{G}$ along the boundary of $P$ and obtain a spherical subdivision graph.
Depending on the gluing map on the boundary $\partial P$, we will obtain two different spherical subdivision graphs (see Figure \ref{fig:SSG}).
\begin{figure}[htp]
    \centering
    \begin{tikzpicture}
        \draw (-1,-1) -- (1,-1)--(1,1)--(-1,1)--cycle;
        \draw (-1.5,1.5)--(1.5,-1.5);
        \fill (-1,-1) circle (1.5pt);
        \fill (1,-1) circle (1.5pt);
        \fill (1,1) circle (1.5pt);
        \fill (-1,1) circle (1.5pt);
        \fill (0,0) circle (1.5pt);

        \begin{scope}[shift={(5,0)}]
            \draw (-1,-1) -- (1,-1)--(1,1)--(-1,1)--cycle;
            \draw (-1,1)--(1,-1);
            \draw (1,1)--(1.5,1.5);
            \draw (-1,-1)--(-1.5,-1.5);
            \fill (-1,-1) circle (1.5pt);
            \fill (1,-1) circle (1.5pt);
            \fill (1,1) circle (1.5pt);
            \fill (-1,1) circle (1.5pt);
            \fill (0,0) circle (1.5pt);
        \end{scope}
    \end{tikzpicture}
    \caption{The first level of two different spherical subdivision graphs constructed from $\mathcal{G}$, where one vertex is at $\infty$.}
    \label{fig:SSG}
\end{figure}
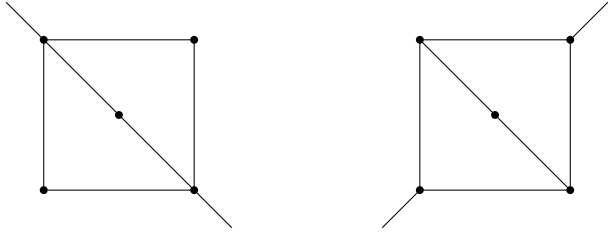

We denote the spherical subdivision graph from the left gluing in Figure \ref{fig:SSG} by $\mathcal{G}_1$ and the right one by $\mathcal{G}_2$.
By Theorem \ref{thm:qch} (I), there exist circle packings $\mathcal{P}_1$ and $\mathcal{P}_2$ corresponding to $\mathcal{G}_1$ and $\mathcal{G}_2$. As discussed, $\mathcal{G}_1$ and $\mathcal{G}_2$ are constructed by gluing two copies of $\mathcal{G}$. In fact, we have
\begin{align*}
    \mathcal{G}_1 &= \mathcal{G}_{1,+} \cup \mathcal{G}_{1,-} = \widetilde{\mathcal{G}}_{1,+} \cup \widetilde{\mathcal{G}}_{1,-}\quad \textrm{and} \quad  \mathcal{G}_2 = \mathcal{G}_{2,+} \cup \mathcal{G}_{2,-},
\end{align*}
where $\mathcal{G}_{i, \pm}, i =1, 2$, correspond to the graphs in the interior and exterior, respectively, of the squares in Figure \ref{fig:SSG}, and $\widetilde{\mathcal{G}}_{1,\pm}$ correspond to the graphs in the left and right side, respectively, of the diagonal line in the left of Figure \ref{fig:SSG}. Note that each graph $\mathcal{G}_{i, \pm}$ and $\widetilde{\mathcal{G}}_{1,\pm}$ is isomorphic to $\mathcal{G}$. By Theorem \ref{thm:qch} (ii)(a), the sub-circle packings of $\mathcal{P}_1$ and $\mathcal{P}_2$ corresponding to $\mathcal{G}_{i,\pm}$ and $\widetilde{\mathcal{G}}_{1,\pm}$ are quasiconformally homeomorphic.

To prove Theorem \ref{theorem:nlo}, we will construct a rational map $R$ with fat gasket Julia set whose contact graph is $\mathcal{G}_1$ and a Kleinian group $G$ whose limit set is a circle packing with tangency graph $\mathcal{G}_2$ (see Figure \ref{fig:FGJ}).

\subsection{Construction of the rational map $R$}
Consider an \textit{orientation-preserving} topological branched covering of degree $2$ defined as in Figure \ref{fig:TMM}, where each quadrilateral on the left is mapped homeomorphically to a quadrilateral on the right with the images of the vertices indicated. Shaded quadrilaterals in the left are mapped to the bounded quadrilateral in the right and unshaded quadrilaterals are mapped to the unbounded quadrilateral.
\begin{figure}[htp]
    \centering
    \begin{tikzpicture}
        \begin{scope}
        \clip (-1.8,-2) rectangle (2,1.5);
        \fill[color=black!20!white] (-1,1)--(-1,-1)--(1,-1)--cycle;
        \fill[color=black!20!white] (-2,2)--(-1,1)--(1,1)--(1,-1)--(2,-2)--(2,2)--cycle;
        \draw (-1,-1) node[left] {$D$}-- (1,-1) node[right] {$C$}--(1,1) node[right] {$B$}--(-1,1)node[left]{$A$}--cycle;
        \draw (-1.5,1.5)--(2,-2);
        \fill (1.8,-1.8) circle (1.5pt) node[left] {$F=\infty$};
        \fill (-1,-1) circle (1.5pt);
        \fill (1,-1) circle (1.5pt);
        \fill (1,1) circle (1.5pt);
        \fill (-1,1) circle (1.5pt);
        \fill (0,0) circle (1.5pt) node[anchor=south west] {$E$};
        \end{scope}

        \draw[->] (2.3,0)--(3,0) node[pos=0.5, above] {$S$};
        
        \begin{scope}[shift={(5,0)}]
            \draw[fill=black!20!white] (-1,-1) node[below] {$\scriptstyle D=S(C)$}-- (1,-1) node[below] {$\scriptstyle C=S(E)=S(F)$}--(1,1) node[above] {$\scriptstyle B=S(A)$}--(-1,1)node[above]{$\scriptstyle A=S(B)=S(D)$}--cycle;
            \fill (-1,-1) circle (1.5pt);
            \fill (1,-1) circle (1.5pt);
            \fill (1,1) circle (1.5pt);
            \fill (-1,1) circle (1.5pt);
            \end{scope}
    \end{tikzpicture}
    \caption{A degree $2$ topological branched covering.}
    \label{fig:TMM}
\end{figure}
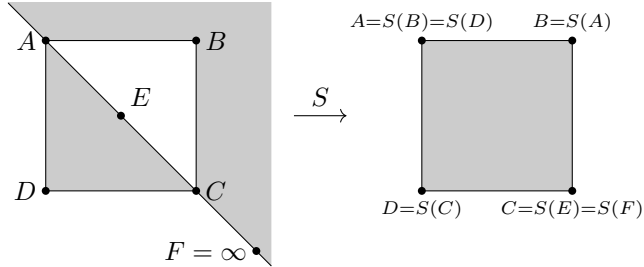

The two critical points are $A$ and $C$, where two shaded and two unshaded regions meet. The point $A$ has period $2$ while $C$ is strictly pre-periodic and is eventually mapped to the cycle of $A$. 
Note that the map $S$ is consistent with the subdivision rule, i.e., for each $n\geq 0$ we have we 
$$
S^{-n}(\partial P) = \mathcal{G}^n_1,
$$
where $P$ is the quadrilateral $ABCD$. Also, every edge is eventually mapped to the edge $AB$. By \cite{LuoZhang:Nonequiv}*{Theorem 4.1}, we obtain a quadratic rational map $R$ with fat gasket Julia set whose contact graph is $\mathcal{G}_1$. Moreover, since every critical point is mapped to a critical cycle, no critical point is on the Julia set. We remark that the proof of \cite{LuoZhang:Nonequiv}*{Theorem 4.1} consists of two steps.
\begin{itemize}
    \item First, one shows that there is no Thurston obstruction for the topological branched covering, so it is equivalent to some post-critically finite rational map whose Julia set is a gasket with contact graph $\mathcal{G}_1$.
    \item Second, since every edge is eventually mapped to the edge $AB$, every contact point is eventually mapped to a unique fixed contact point. One can perform a pinching deformation to obtain a parabolic rational map $R$ with a parabolic fixed point of multiplicity $3$ at the corresponding contact point, and thus the gasket is fat.
\end{itemize}

\subsection{Construction of the Kleinian group $G$}
We outline the construction of the Kleinian group.
The idea is to construct two conformal symmetries of $\mathcal{P}_2$ that generate a large group.
By the rigidity of circle packings (Theorem \ref{thm:qch} (II)), the conformal symmetries of $\mathcal{P}_2$ are in one-to-one correspondence with the symmetries of $\mathcal{G}_2$, i.e., isomorphisms of $\mathcal{G}_2$ as a plane graph. Therefore, it suffices to construct symmetries of the graph $\mathcal{G}_2$.

\begin{figure}[htp]
    \centering
    \begin{tikzpicture}
            \draw (-1,-1) node[left] {$D$}-- (1,-1) node[right] {$C$}--(1,1) node[right] {$B$}--(-1,1)node[left]{$A$}--cycle;
            \fill (1.8,1.8) circle (1.5pt) node[left] {$F=\infty$};
            \draw (-1,1)--(1,-1);
            \draw (1,1)--(2,2);
            \draw (-1,-1)--(-1.5,-1.5);
            \fill (-1,-1) circle (1.5pt);
            \fill (1,-1) circle (1.5pt);
            \fill (1,1) circle (1.5pt);
            \fill (-1,1) circle (1.5pt);
            \fill (0,0) circle (1.5pt) node[anchor=south west] {$E$};

            \draw[->] (2,0)--(3,0) node[pos=0.5, above]{$g_1$};
            \begin{scope}[shift={(5,0)}]
                \draw (-1,-1) node[left] {$\scriptstyle D=g_1(C)$}-- (1,-1) node[right] {$\scriptstyle C=g_1(E)$}--(1,1) node[right] {$\scriptstyle B=g_1(A)$}--(-1,1)node[left]{$\scriptstyle A=g_1(D)$}--cycle;
                \fill (1.8,1.8) circle (1.5pt) node[left] {$F=\infty$};
                \draw (-1,1)--(1,-1);
                \draw (1,1)--(2,2);
                \draw (-1,-1)--(-1.5,-1.5);
                \fill (-1,-1) circle (1.5pt);
                \fill (1,-1) circle (1.5pt);
                \fill (1,1) circle (1.5pt);
                \fill (-1,1) circle (1.5pt);
                \fill (0,0) circle (1.5pt);
            \end{scope}
        \end{tikzpicture}
    \caption{The construction of M\"obius symmetries.}
    \label{fig:MS}
\end{figure}
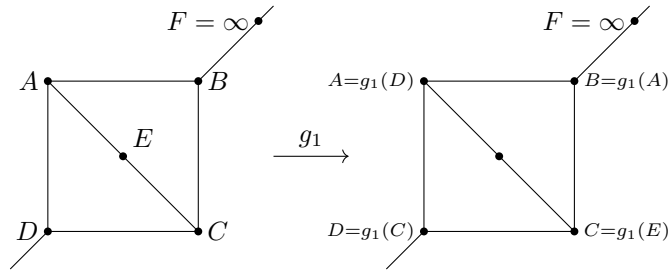
Consider a map $g_1$ as in Figure \ref{fig:MS}, which sends the quadrilateral $DAEC$ to $ABCD$.
Inductively, one can check that this map $g_1$ extends uniquely to a symmetry of $\mathcal{G}_2$. For example, the extension will send the quadrilateral $ABCE$ to $BFDC$.
Similarly, one can construct a symmetry $g_2$ that sends $EABC$ to $ABCD$.

Let $M_1$ and $M_2$ be the corresponding M\"obius symmetries of $\mathcal{P}_2$.
Consider the group $G$ generated by $M_1, M_2$.
Then $G$ acts almost transitively on the tangency points; namely, any tangency point can be mapped by an element of $G$ to one of the four tangency points associated to the edges of $ABCD$.
Since $\Lambda(\mathcal{P}_2)$ is invariant under $G$, the limit set of $G$ is contained in $\Lambda(\mathcal{P}_2)$. On the other hand, the limit set of $G$ contains all the tangency points as they are parabolic fixed points of some elements in $G$, and the set of tangency points is dense in $\Lambda(\mathcal{P}_2)$ (see Lemma \ref{cor:density} \ref{d:4}).
Thus the limit set of $G$ is $\Lambda(\mathcal{P}_2)$. 

We remark that this Kleinian circle packing also appears in the work on generalized Apollonian packings, where the construction comes from the $\mathbb{Q}(\sqrt{-7})$-Bianchi group (see \cite{Stange18}*{Figure 19} for a different presentation of this circle packing).

We now give a more precise statement of Theorem \ref{theorem:nlo} (see Figure \ref{fig:FGJ}).

\begin{theorem}\label{theorem:snlo}
    Let $R$ and $G$ be the rational map and Kleinian group, respectively,  constructed above.
    There exist Jordan domains $U^\pm, \widetilde{U}^\pm, V^\pm\subset \widehat\C$, and quasi\-con\-formal maps $\phi^\pm\colon \widehat \C\to \widehat\C$ and $\widetilde{\phi}^\pm\colon \widehat \C\to \widehat\C$ so that
    \begin{itemize}
        \item $U^+,U^-$ are disjoint and $\widetilde{U}^+,\widetilde{U}^-$ are disjoint, 
        \item $\overline{U^+ \cup U^-} = \overline{\widetilde{U}^+ \cup \widetilde{U}^-} = \widehat\C$, 
        \item $\mathcal{J}(R) \subset U^+ \cup U^- \cup \widetilde{U}^+ \cup \widetilde{U}^-$,
        \item $V^+,V^-$ are disjoint with $\overline{V^+ \cup V^-}=\widehat\C$, 
        \item $\phi^{\pm}(U^{\pm})=\widetilde{\phi}^{\pm}(\widetilde{U}^{\pm})=V^{\pm}$, and
        \item $\phi^\pm (\overline{U^\pm} \cap \mathcal{J}(R)) =\widetilde{\phi}^{\pm}(\overline{\widetilde{U}^{\pm}}\cap \mathcal{J}(R)) = \overline{V^\pm} \cap \Lambda(G)$.
    \end{itemize}
\end{theorem}

\begin{proof}
    Recall that for $i=1,2$, we denote by $\mathcal{P}_i$ the circle packing for the spherical subdivision graph $\mathcal{G}_i$. By Theorem \ref{thm:QU}, the fat gasket Julia set $\mathcal{J}(R)$ is quasiconformally homeomorphic to a circle packing with tangency graph $\mathcal G_1$. By Theorem \ref{thm:qch}, $\mathcal J(R)$ is quasiconformally homeomorphic to the circle packing $\mathcal{P}_1$. Therefore, in order to prove the theorem, it suffices to construct appropriate quasiconformal maps between $\mathcal{P}_1$ and $\mathcal{P}_2$.  Recall that
    \begin{align*}
        \mathcal{G}_1 &= \mathcal{G}_{1,+} \cup \mathcal{G}_{1,-} = \widetilde{\mathcal{G}}_{1,+} \cup \widetilde{\mathcal{G}}_{1,-} \quad \textrm{and} \quad  \mathcal{G}_2 = \mathcal{G}_{2,+} \cup \mathcal{G}_{2,-},
    \end{align*}
    where each graph $\mathcal{G}_{i, \pm}$ and $\widetilde{\mathcal{G}}_{1,+}$ is isomorphic to the initial subdivision graph $\mathcal{G}$. We now construct open sets $U^\pm$ and maps $\phi^\pm$ associated to $\mathcal{G}_{1,\pm}$ as follows. Let $\mathcal{P}_{i, \pm}$ and $\widetilde{\mathcal{P}}_{1, \pm}$ be the corresponding sub-circle packings of $\mathcal{P}_i$, $i=1,2$. Note these circle packings all have the same combinatorics.  Let $\widehat{U}^\pm$ (resp.\ $V^\pm$) be the Jordan domains bounded by the limit set  of the group generated by reflections along the circles in $\mathcal{P}_1$ (resp.\ $\mathcal{P}_2$) associated to $A, B, C, D$ as illustrated in Figure \ref{fig:FGJ}. By Theorem \ref{thm:existenceTeich}, there exist quasiconformal maps $\widehat{\phi}^\pm\colon (\widehat\C, \widehat{U}^\pm) \to (\widehat\C, V^\pm)$ so that $\widehat{\phi}^\pm(\overline{\widehat{U}^\pm} \cap \Lambda(\mathcal{P}_1)) = \overline{V^\pm} \cap \Lambda(\mathcal{P}_2)$.  By pulling back $\widehat{U}^\pm$ and $\widehat{\phi}^\pm$ under the quasiconformal map between $\mathcal{J}(R)$ and $\Lambda(\mathcal{P}_1)$, we obtain regions $U^\pm$ and maps $\phi^\pm$ that satisfy the desired conclusion. Similarly, we can construct regions $\widetilde{U}^\pm$ and maps $\widetilde{\phi}^\pm$ for $\widetilde{\mathcal{G}}_{1,+}$. By construction we have $\mathcal{J}(R) \subset U^+ \cup U^- \cup \widetilde{U}^+ \cup \widetilde{U}^-$ and the theorem follows. 
\end{proof}

\begin{figure}[htp]
    \centering    
    \begin{tikzpicture}
    \node[anchor=south west,inner sep=0] at (0,0) {\includegraphics[width=0.6\textwidth]{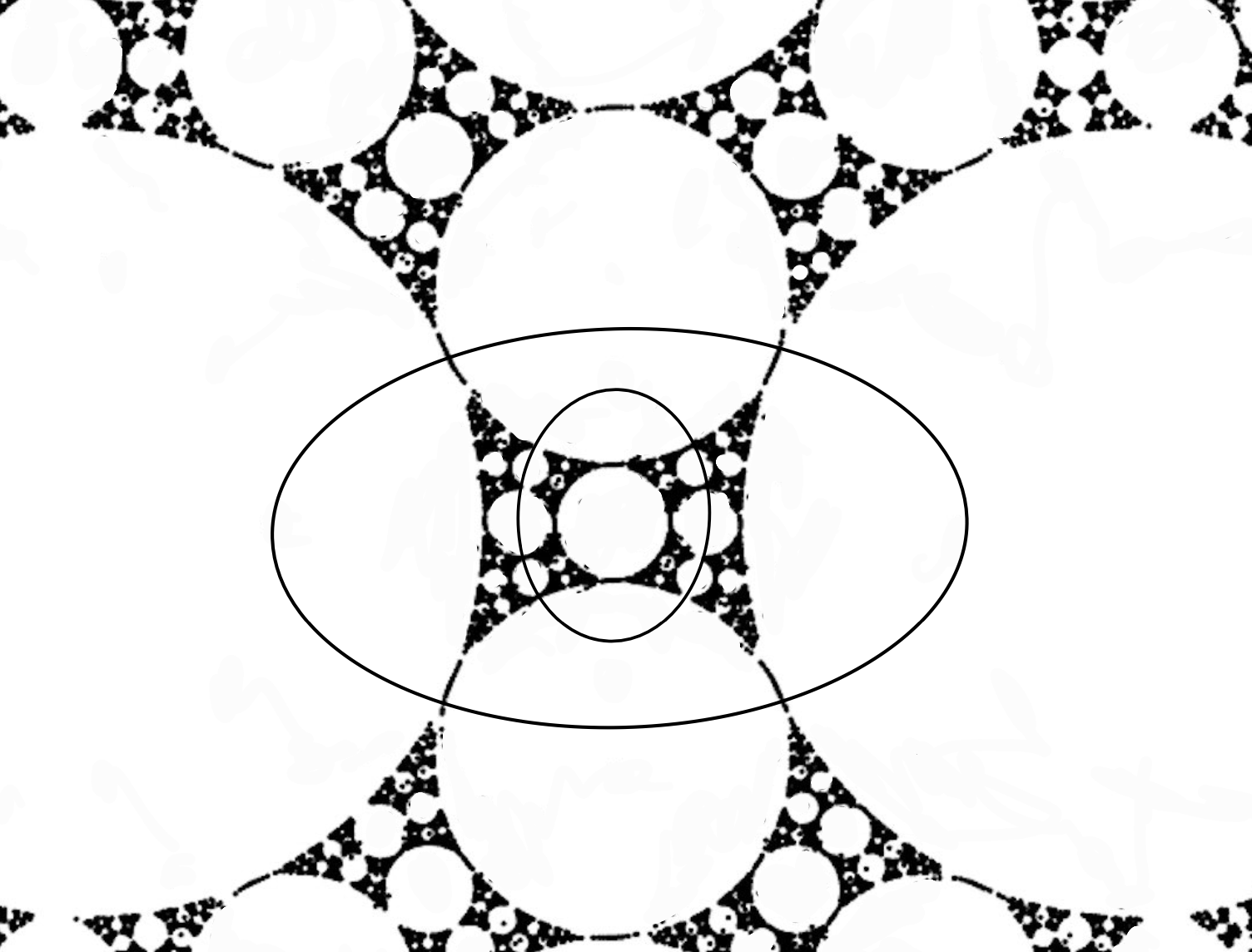}};
    \node at (2.2,2.8) {\begin{footnotesize}$V$\end{footnotesize}};
    \node at (3.8,3.2) {\begin{footnotesize}$W$\end{footnotesize}};
    \end{tikzpicture}
    \caption{A magnification of the limit set in Figure \ref{fig:FGJ}. The open sets $V, W$ for the quasiregular symmetry are illustrated.}
    \label{fig:QRS}
\end{figure}

\begin{proof}[Proof of Corollary \ref{qrsymmetry}]
    Let $\widetilde{U} = \widetilde{U}^+$, where $\widetilde U^+$ is as in Theorem \ref{theorem:snlo} (see Figure \ref{fig:FGJ}). By construction, the Fatou component of $R$ associated to the vertex $B$ is fixed under $R^{\circ 2}$. Let $\widetilde{W}$ be the component of $R^{-2}(\widetilde{U})$ that contains that Fatou component. Then by the construction of $R$ (see Figure \ref{fig:TMM}), one can check that $\widetilde{W}\cap \mathcal{J}(R) \subset \widetilde{U}$ and $R^{\circ 2}\colon  \widetilde{W} \to \widetilde{U}$ is a degree $2$ branched covering map. The corollary follows by setting $V = \widetilde{\phi}^+(\widetilde{U})$, $W = \widetilde{\phi}^+(\widetilde{W})$, and $f = \widetilde{\phi}^+ \circ R^2 \circ (\widetilde{\phi}^+)^{-1}$. See Figure \ref{fig:QRS}.
\end{proof}

\bibliography{biblio}
\end{document}